\documentclass[a4paper,11pt]{amsart}

\usepackage[all]{xypic}
\usepackage{latexsym,amsmath,amscd,amssymb,amsthm,epsfig,times}
\usepackage{graphicx}
\usepackage[T1]{fontenc}
\usepackage{tikz}

\usepackage{color}
\definecolor{maus}{rgb}{0.2,0.2,0.2}

\newtheorem{thm}{Theorem}[section]
\newtheorem{prop}[thm]{Proposition}
\newtheorem{lem}[thm]{Lemma}
\newtheorem{cor}[thm]{Corollary}

\newtheorem*{chekanov*}{Theorem \ref{t:chekanov}}
\newtheorem*{non-loose class*}{Theorem \ref{t:non-loose class}}

\newtheorem*{thm*}{Theorem}
\newtheorem*{lem*}{Lemma}
\newtheorem*{prob*}{Problem}

\theoremstyle{definition}
\newtheorem{defn}[thm]{Definition}
\newtheorem{ex}[thm]{Example}
\newtheorem{rem}[thm]{Remark}
\newtheorem*{ex*}{Example}

\newcommand{\thmref}[1]{Theorem~\ref{#1}}
\newcommand{\lemref}[1]{Lemma~\ref{#1}}
\newcommand{\propref}[1]{Proposition~\ref{#1}}
\newcommand{\defref}[1]{Definition~\ref{#1}}
\newcommand{\exref}[1]{Example~\ref{#1}}
\newcommand{\remref}[1]{Remark~\ref{#1}}
\newcommand{\corref}[1]{Corollary~\ref{#1}}
\newcommand{\figref}[1]{Figure~\ref{#1}}
\newcommand{\secref}[1]{Section~\ref{#1}}

\newcommand{\lra}{\longrightarrow}

\newcommand{\FF}{\mathcal{ F}}

\newcommand{\LL}{\mathcal{ L}}

\newcommand{\R}{\mathbb{ R}}

\newcommand{\Z}{\mathbb{ Z}}

\newcommand{\N}{\mathbb{N}}

\newcommand{\tb}{\mathrm{tb}}

\newcommand{\eing}[1]{\big|_{#1}}

\newcommand{\eps}{\varepsilon}

\newcommand{\lmt}{\longmapsto}

\DeclareMathAccent{\ring}{\mathalpha}{operators}{"17}

\author{T. Vogel}

\title{Non-loose unknots, overtwisted discs,  and the contact mapping class group of $S^3$} 

\address{Mathematisches Institut, Ludwig-Maximilians-Universit{\"a}t, Theresienstr. 39, 80333 M{\"u}nchen, Germany}

\email{tvogel@math.lmu.de}

\date{\today; MSC 2000: 57R17, 57R30}

\begin{document}


\begin{abstract}
We classify  Legendrian unknots in overtwisted contact structures on $S^3$. In particular, we show  that up to contact isotopy for every pair $(n,\pm(n-1))$ with $n>0$ there are exactly two oriented non-loose Legendrian unknots in $S^3$ with Thurston-Bennequin invariant $n$ and rotation number $\pm(n-1)$.   (Only one overtwisted contact structure on $S^3$ admits a non-loose unknot $K$ and the classical invariants have to be $\mathrm{tb}(K)=n$ and $\mathrm{rot}(K)=\pm(n-1)$ for $n>1$.) 

This can be used to prove two results attributed to Y.~Che\-kan\-ov: The first implies that the contact mapping class group of an overtwisted contact structure on $S^3$ depends on the contact structure. The second result is that the identity component of the contactomorphism group of an overtwisted contact structure on $S^3$ does not always act transitively on the set of boundaries of overtwisted discs. 
\end{abstract}

\maketitle


\markright{Non-loose unknots and overtwisted discs in $S^3$}
\pagestyle{myheadings}

\section{Introduction}

In $3$-dimensional contact topology there is a fundamental dichotomy between tight and overtwisted contact structures pioneered by D.~Bennequin \cite{bennequin} and developed further by Y.~Eliashberg \cite{el20}.

Usually, the attention is restricted to tight contact structures. These contact structures are characterized by the absence of so-called overtwisted discs. This is because a theorem of Y.~Eliashberg implies, among other things, that the classification of contact structures on a fixed closed $3$-manifold up to diffeomorphism  which contain an overtwisted disc  up to isotopy  coincides with the classification of plane fields up to homotopy. Nevertheless, as is shown in this paper, there are rigidity phenomena for overtwisted contact structures. 

What is important in Eliashberg's result is that one needs to  control an overtwisted disc. In particular, if one considers the isotopy problem for a pair of  Legendrian or transverse links, then Eliashberg's theorem can be applied effectively when the complement of the links is overtwisted. 

This is not always the case. As K.~Dymara noted in \cite{dym}, it may happen that a Legendrian or transverse knot has a tight complement. In other words it intersects all overtwisted discs.  A knot with this property is called non-loose. Also, it may happen that two Legendrian knots have the same classical invariants (the Thurston-Bennequin invariant and the rotation number), each knot has an overtwisted complement but the complement of the union of the two knots is tight. In these situations one cannot apply Eliashberg's theorem directly to construct for example Legendrian isotopies between two non-loose Legendrian knots with the same classical invariants. 

In her preprint \cite{dym2} K.~Dymara gives several examples of non-loose Legendrian knots in $S^3$.  J.~Etnyre constructed more examples in \cite{etnyre-ex}. He also found examples of pairs of non-loose Legendrian knots whose complements are not diffeomorphic. 

Non-loose Legendrian knots are interesting because one can obtain interesting tight contact structures  from surgeries on non-loose Legendrian knots. Furthermore, as shown in this paper, non-loose unknots can be used to obtain non-trivial information about overtwisted contact structures.

The coarse classification (i.e. up to diffeomorphism) of non-loose unknots in $S^3$ is due to Y.~Eliash\-berg and M.~Fraser in \cite{elfr-long} and independently to J.~Etnyre (see \cite{etnyre-ot}).  It turns out that on $S^3$ there is a unique contact structure which admits a non-loose unknot, we denote it by $\xi_{-1}$. As a plane field  this contact structure is homotopic to the unique negative tight contact structure. Moreover, the Thurston-Bennequin number of a non-loose unknot has to be positive.  Non-loose unknots with Thurston-Bennequin invariant one are called minimal. 

A coarse classification of certain rationally null-homologous knots in lens spaces was obtained by H.~Geiges and S.~Onaran \cite{go}. Also, the problem list \cite{etnyreng} states several questions concerning Legendrian knots in overtwisted contact structures and the structure of the contact mapping class group (in particular the items  31,32, 38(4) and 40).
  
One main result of this paper is the classification of non-loose Legendrian unknots in $S^3$ up to Legendrian isotopy.  
\begin{non-loose class*} 
Let $K\subset S^3$ be a non-loose Legendrian unknot with $\mathrm{tb}(K)=n>0$ and $\mathrm{rot}(K)=n-1$. Then $K$ is isotopic to exactly one of two standard unknots with the same classical invariants.
\end{non-loose class*}

The most important case is the study of the minimal case, i.e. $n=1$. The arguments used to prove the theorem in this case rely on 
\begin{itemize}
\item Eliashberg's classification theorem \cite{elot}, and
\item a study of the characteristic foliation on families of spheres using the techniques developed in the beautiful paper \cite{gi-bif} by E.~Giroux.  
\end{itemize}

\figref{b:geography} summarizes the classification of Legendrian unknots in $(S^3,\xi_{-1})$ up to Legendrian isotopy. This is the combination of \thmref{t:non-loose class}, \propref{p:loose class} with Proposition 4.11  from \cite{elfr-long}. It turns out that Legendrian unknots (loose or non-loose) with non-negative Thurston-Bennequin number are never classified by their classical invariants. 

\begin{figure}[htb]
\begin{center}
\includegraphics[scale=0.6]{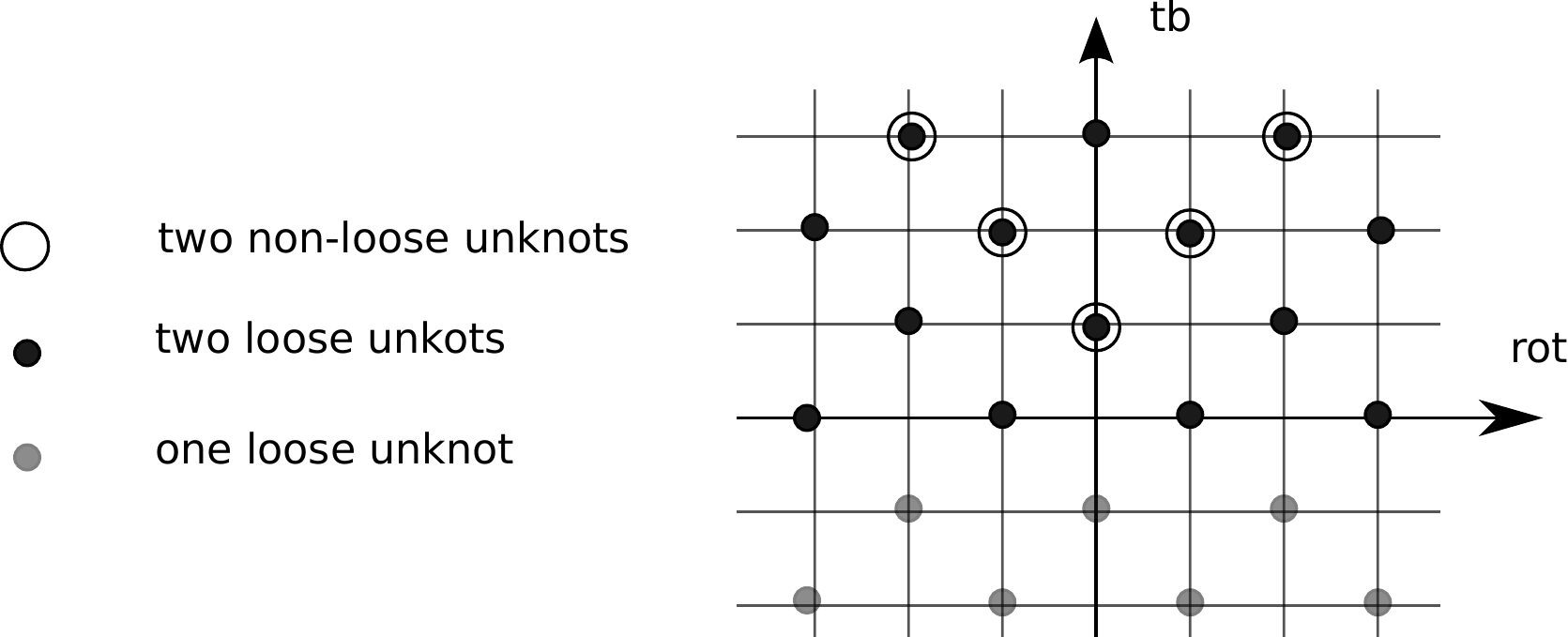}
\caption{Summary of the classification of Legendrian unknots in $(S^3,\xi_{-1})$ up to Legendrian isotopy.}
\label{b:geography}
\end{center}
\end{figure}

For all other contact structures on $S^3$, two Legendrian unknots with the same classical invariants are Legendrian isotopic. This is stated as  \propref{p:other contstr unknot Leg simple} in the case of overtwisted contact structures. The tight case is the subject of \cite{elfr, elfr-long}. J.~Etnyre has shown that transverse unknots are classified up to isotopy by their self-linking number (Cor. 2.3 in \cite{etnyre-ot}). In particular, there are no non-loose transverse unknots in $S^3$.  

There are several  applications of the classification of minimal non-loose Legendrian unknots up to isotopy:  First, we prove that there is a pair of boundaries of overtwisted discs in $(S^3,\xi_{-1})$ which are not Legendrian isotopic to each other. More precisely, it is not always possible to find a contact isotopy in a contact manifold which moves an overtwisted ball into another such ball. 

Another application concerns the contact mapping class group $\pi_0\left(\mathrm{Cont_+(S^3,\xi)}\right)$ of contact structures on $S^3$ which preserve the orientation of the contact structure.  The following theorem is attributed to Y.~Chekanov in \cite{elfr} without any indication of proof. 
\begin{chekanov*}[Chekanov]
Let $\xi$ be an overtwisted contact structure on $S^3$ and denote by $\mathrm{Diff}_+(S^3,\xi)$ the group of diffeomorphisms of $S^3$ which preserve $\xi$ and an orientation of this plane field. 
Then
\begin{align*}%
\pi_0\left(\mathrm{Diff}_+(S^3,\xi)\right) & =\left\{\begin{array}{rl} \Z_2\oplus \Z_2 & \textrm{ if } \xi\simeq \xi_{-1} \\
\Z_2 & \textrm{otherwise.}
\end{array} \right.
\end{align*}
\end{chekanov*}
The additional $\Z_2$-summand in the case $\xi\simeq\xi_{-1}$ is determined by considering the action of contactomorphisms on isotopy classes of oriented minimal non-loose Legendrian unknots. In \cite{dym} K.~Dymara introduced an invariant which detects the other $\Z_2$-factor for all overtwisted contact structures on $S^3$. Finally,  recall that the group of diffeomorphisms preserving the oriented tight contact structure on $S^3$ is connected according to \cite{el20}. 

\thmref{t:chekanov} can be applied to obtain the classification of Legendrian links (loose and non-loose) in overtwisted contact structures once the classification up to contact diffeomorphism is known.  The simplest case is of course the unknot mentioned above.

An important aspect of the classification problem of non-loose unknots in $S^3$, that is left open in this article, is the following problem.

\begin{prob*}
Let $K,K'$ be Legendrian unknots in $(S^3,\xi_{-1})$ with $\tb(K)=\tb(K')=n>0$ and the same rotation number which are not isotopic to each other and both loose or both non-loose. 

If $L$ is yet another Legendrian knot with the same properties, determine whether $L$ is isotopic to $K$ or to $K'$. 
\end{prob*}
More problems can be found in \cite{etnyre-ot} and \cite{bo}. Etnyre's article also contains proofs of most of the previously known results on non-loose knots. In \cite{bo}, K.~Baker and S.~Onaran propose to quantify the degree of non-looseness of Legendrian/transverse knots and suggest a number of problems related to their invariants. Another reference is \cite{non-loose-ex} where the authors give an example of a pair of non-loose transverse knots with the same classical invariants which are not isotopic.

\bigskip


{\bf Organization of the paper:} \secref{s:prelim} contains definitions and examples of non-loose knots.  In \secref{s:class} we first review the coarse classification using the approach taken in \cite{elfr-long} because the normal form for the characteristic foliation on a disc bounding a non-loose unknot is used  in the next sections. After the construction of examples of contact diffeomorphisms preserving a particular non-loose unknot, we prove that every minimal non-loose unknot is isotopic to a standard example $K\subset (S^3,\xi_{-1})$. Then we show that  $K$ and $\overline{K}$ are not isotopic as oriented knots (here we $\overline{K}$ is $K$ with the orientation reversed).  Finally, in \secref{s:appl} we discuss Chekanov's theorem and classify Legendrian unknots in overtwisted contact structures on $S^3$. 
\medskip 

{\bf Acknowledgments:} The first time I heard about non-loose knots was in discussions with J.~Etnyre during which  he developed a surgery based proof of the coarse classification of non-loose unknots, and I thank him for his explanations. Others, in particular Y.~Eliashberg, H.~Geiges, R.~Gompf, P.~Massot and  S.~Onaran, have discussed the problem with me, and I want to thank all of them.  It is a great pleasure for me to thank the Mittag-Leffler Institute (and the organizers T.~Ekholm, Y.~Eliashberg)  for organizing a program on symplectic topology in the fall of 2015 where a portion of this manuscript was prepared. 

The referees provided valuable suggestions which helped to improve the presentation. In addition, one of the referees proposed that one should try to classify loose Legendrian unknots in $S^3$. This classification is the content of \propref{p:loose class}. I want to thank both referees for their thorough and  thoughtful comments.

\section{Preliminaries} \label{s:prelim}

In this section we review standard definitions, we review Giroux's description of contact structures in terms of families of characteristic foliations. We also discuss basic facts  about non-loose knots and some examples. Finally, in \secref{s:homotopy} we discuss a topological  invariant introduced by K.~Dymara in \cite{dym} which can be used to show that the group of diffeomorphisms preserving an oriented overtwisted contact structure on $S^3$ is not connected. 

\subsection{Tight and overtwisted contact structures} \label{ss:tight ot}

\begin{defn} \label{d:cont str}
A {\em  contact structure} is a smooth plane field $\xi$ on a $3$-manifold which is locally defined by a $1$-form $\alpha$ such that $\alpha\wedge d\alpha$ never vanishes. 
\end{defn}
Manifolds admitting contact structures are necessarily oriented, and we will only consider cooriented contact structures which are positive (i.e. $\alpha\wedge d\alpha>0$). Throughout this paper, we assume that $\xi$ is oriented as a plane field. 

An important fact which we will use is Gray's theorem:
\begin{thm} \label{t:gray}
Let $\xi_t$ be a family of contact structures on a closed $3$-manifold. Then there is an isotopy $\psi_t$ such that $\psi_{t*}(\xi_0)=\xi_t$. 
\end{thm}
Parametric and relative versions of this theorem also hold. One refinement which comes out of the proof of \thmref{t:gray} by the Moser method is the following: If $K$ is a Legendrian knot which is tangent to $\xi_t$ for all $t$, then $\psi_t$ can be chosen to preserve $K$ (but not pointwise, of course). The details can be found in Section 2.2 of \cite{geiges}, for example.

\begin{defn} \label{d:ot}
An embedded disc $D\lra M$ in a contact manifold $(M,\xi)$ is {\em overtwisted} if $T_pD=\xi(p)$ for all points $p\in\partial D$. A contact structure is {\em tight} if there is no overtwisted disc, otherwise it is {\em overtwisted}. 
\end{defn}

The following two theorems of Eliashberg highlight the fundamental difference between overtwisted contact structures (which are very flexible) and tight ones (for which more rigidity phenomena appear).

\begin{thm}[Eliashberg \cite{el20}] \label{t:tight on ball}
Let $\xi$ be the germ of a tight contact structure along $\partial B^3$. The space of tight contact structures on the  ball which coincide with $\xi$ near $\partial B^3$ is weakly contractible. 
\end{thm}
In particular, the relative homotopy type of the plane field on $(B^3,\partial B^3)$ is completely determined be the boundary data and the requirement that $\xi$ is a tight contact structure. 

Now let $\xi_\Delta$ be a contact structure defined near a disc $\Delta$ and on a neighborhood of a compact set $N$ so that $\Delta$ is an overtwisted disc for $\xi_\Delta$. We will use the following notation:
\begin{align*}
\mathrm{Cont}(M,N,\xi_\Delta) & = \{ \textrm{contact structures equal to }\xi_\Delta\textrm{ near  }\Delta \cup N\} \\
\mathrm{Distr}(M,N,\xi_\Delta) & = \{ \textrm{plane fields which coincide with }\xi_\Delta\textrm{ near  }\Delta\cup N\}.
\end{align*}
Both spaces carry the $C^k$-topology with $k\ge 1$. 
The proof of the following theorem can be found in \cite{elot}.  It is discussed also in \cite{gi-bour} and \cite{geiges}. Y.~Huang proved a non-parametric version using bypasses in \cite{huang-ot}.

\begin{thm}[Eliashberg \cite{elot}] \label{t:eliashclass}
Let $\xi_\Delta$ be a contact structure as above and $N\subset M$ a compact set in the complement of $\Delta$ such that $M\setminus  N$ is connected. Then 
$$
\mathrm{Cont}(M,N,\xi_\Delta) \hookrightarrow \mathrm{Distr}(M,N,\xi_\Delta) 
$$
is a weak homotopy equivalence. 
\end{thm}
The statement can be enhanced even further: Not only does a relative version with respect to compact subsets $N$ of $M$ hold, but also a relative version with respect to compact subsets of the parameter space. More precisely, if $\zeta_s, s\in S$, is a family of plane fields in $\mathrm{Distr}(M,N,\xi_\Delta)$ with compact parameter space $S$ and $S'\subset S$ is a compact set  such that $\zeta_s\in\mathrm{Cont}(M,N,\xi_\Delta)$ for $s\in S'$,  then one does not have to change $\zeta_s$ for $s\in S'$ when applying \thmref{t:eliashclass}.

Given two disjoint overtwisted discs $\Delta_1$ and $\Delta_2$ one can combine the two relative versions and obtain a family of contact structure from a family $\zeta_s,s\in[0,1],$ 
such that $\Delta_1$, respectively $\Delta_2$, are overtwisted discs of $\zeta_s$ when $s\in[0,2/3]$ respectively $[1/3,1]$.

Also, assume that for a family of plane fields $\zeta_s, s\in[0,1],$ there is a family of  overtwisted discs $\Delta_s$ contained in the region where $\zeta_s$ is a positive contact structure and an isotopy $\varphi_s$ such that $\Delta_s=\varphi_s(\Delta_0)$.  Then one can apply  \thmref{t:eliashclass} to $\left(\varphi_{s}^{-1}\right)_*\left(\zeta_s\right)$ to obtain a deformation of $\zeta_s$ to a family of contact structures $\xi_s$ such that $\Delta_s$ is an overtwisted disc of $\xi_s$. 


Theorem \ref{t:eliashclass} can be used to produce families of overtwisted contact structures on $S^3$ with prescribed homotopy type as families of plane fields. The following consequence of \thmref{t:eliashclass} which can be found \cite{dym}.

\begin{lem}  \label{l:families}
Let $\xi$ be a contact structure on a collar $N$ of the boundary of $B^3$ which is overtwisted. For every homotopy class of families of plane fields $\zeta_s, s\in K$, where $K$ is a compact manifold, which coincides with $\xi$ on $N$ there is a family of contact structures $\xi_s$ in the same homotopy class.
\end{lem}

\begin{proof}
This is a direct application of \thmref{t:eliashclass}: By assumption, there is always a fixed overtwisted disc in $N$, so every family of plane fields  $\zeta_s, s\in K$, as in the lemma on $B^3$ can be homotoped to a family of contact structures $\xi_s, s\in K$. 
\end{proof}
Before we discuss Legendrian and transverse knots, let us note one consequence of \thmref{t:tight on ball}. The arcs $\gamma_0,\gamma_1$ appearing in the second part of the proposition will be open segments of Legendrian knots in later applications.

\begin{prop}[V.~Colin, \cite{col-gluing}] \label{p:contact squeezing}
Let $B_0,B_1\subset M$ be two closed balls in a contact manifold such that $B_0\subset\ring{B}_1$ and  the contact structure is tight on an open neighborhood $V$ of $B_1$. Then there is a contact isotopy of $M$ with support in $V$ which moves $B_1$ into the interior of $B_0$. 

If there is a pair of unknotted open Legendrian arcs $\gamma_0,\gamma_1$ which are transverse to $\partial B_0$ and $\partial B_1$ each arc intersects $\partial B_0$ and $\partial B_1$ exactly once, then one can choose the isotopy so that it preserves $\gamma_0,\gamma_1$.
\end{prop}

\begin{defn} \label{d:Leg knot}
A {\em Legendrian knot} in a contact manifold $(M,\xi)$ is an embedding $S^1\lra M$ so that the tangent space of the image is contained in $\xi$. A knot $K$ in $(M,\xi)$ is {\em transverse} if its tangent space is transverse to $\xi(p)$ at every point $p\in K$, usually a transverse knot is oriented so that the orientation coincides with the coorientation of $\xi$. 
\end{defn}

In addition to their isotopy type as embedded curves, null-homologous Legendrian or transverse knots have the following classical invariants.

\begin{defn} \label{d:tb rot}
Let $\Sigma$ be a connected Seifert surface for a Legendrian knot $K$. The {\em Thurston-Bennequin invariant} $\tb(K)$ of $K$ is the algebraic intersection number of the push off of $K$ along a vector field transverse to $\xi$ with $\Sigma$. The {\em rotation number} of $K$ is the number of full turns the positive tangent vector of $K$ makes compared to an oriented framing of $\xi\eing{\Sigma}$ ($\Sigma$ is a surface with boundary) as one moves along the oriented curve $K$.

The {\em self-linking number} $\mathrm{sl}(K)$ of a null-homologous transverse knot $K$ is the algebraic intersection number of a Seifert surface $\Sigma$ with $K'$ where $K'$ is obtained from $K$ by pushing $K$ away from itself using a vector field tangent to $\xi$ which vanishes nowhere along $\Sigma$.   
\end{defn}
In general, these invariants depend on the choice of $[\Sigma]\in H_2(M,K;\Z)$, but if $M=S^3$, then the Thurston-Bennequin invariant, the rotation number and the self-linking number are independent of the Seifert surface $\Sigma$. Note that $\tb$ is also independent of the orientation of the knot while the rotation number changes sign when the orientation of $K$ is reversed.

For an oriented Legendrian knot $K$ one obtains a transverse knot by pushing $K$ away from itself using a vector field $N$ tangent to $\xi$ so that along the knot the vector field $N$ followed by the tangent direction of $K$ is an oriented basis of $\xi$. The resulting knot will be denoted by $K^+$ and according to \cite{bennequin} the classical invariants of $K$ and $K^+$ are related as follows
\begin{equation} \label{e:sl tb}
\mathrm{sl}(K^+)=\mathrm{tb}(K)-\mathrm{rot}(K). 
\end{equation} 

\begin{thm}[Eliashberg \cite{el20}] \label{t:tbineq}
Let $(M,\xi)$ be a contact $3$-manifold. Then $\xi$ is tight if and only if 
\begin{equation} \label{e:tbineq}
\tb(K) + |\mathrm{rot}(K)|\le-\chi(\Sigma)
\end{equation}
for every null-homologous Legendrian knot $K$ and Seifert surface $\Sigma$. This is equivalent to 
\begin{equation} \label{e:tbineq t}
\mathrm{sl}(K)\le-\chi(\Sigma) 
\end{equation}
for every transverse knot with Seifert surface $\Sigma$.
\end{thm}
Finally, we note that a Legendrian knot can be stabilized to another Legendrian knot  in the same smooth knot type such that the Thurston-Bennequin invariant decreases by $1$ and the rotation number has changes by $\pm 1$. Depending on the sign the stabilization is said to be positive or negative (a negative stabilization appears in \figref{b:isotopy} on p.~\pageref{b:isotopy}). The stabilized knot is well-defined up to Legendrian isotopy and if two Legendrian knots are isotopic, then the same is true for their positive (or negative) stabilizations. 

\subsection{Surfaces in contact manifolds} \label{s:non-convex}

Let $\Sigma$ be an oriented surface embedded in a contact manifold $(M,\xi)$. If $\partial\Sigma\neq\emptyset$, then  we require that the boundary is tangent to $\xi$. The following terminology introduced in \cite{gi-conv} for contact structures will be used for general plane fields (in this context we will often use the notation $\zeta$ instead of $\xi)$.

\begin{defn} \label{d:char fol voc}
The {\em characteristic foliation} $\xi(\Sigma)$ on $\Sigma$ is defined by the singular line field $T\Sigma\cap \xi$. The characteristic foliation is oriented so that the orientation of $\xi(\Sigma)$ followed by the coorientation of $\xi$ is the orientation of $\Sigma$. 

A point $p\in\Sigma$ is called {\em singular} if $T_p\Sigma=\xi(p)$, when $\xi$ and $\Sigma$ are oriented, a singular point is {\em positive} if the orientations coincide, otherwise it is {\em negative}.
\end{defn} 
The characteristic foliation on a closed surface $\Sigma$ determines the germ of $\xi$ along $\Sigma$ up to isotopy preserving $\xi(\Sigma)$. 

For singular points there is a well-defined notion of positive/negative divergence. Positive singularities have positive divergence while the divergence at negative singularities is negative.  The orientation convention turns positive/negative elliptic singularities into sources/sinks. 

Because of the contact condition, the characteristic foliation near an isolated singular point has a particular form. For example, the index of such a critical point can only be $\pm 1$ or $0$ (a proof can be found in \cite{unique}).  

Generically, the singular line field $\xi(\Sigma)$ on $\Sigma$ defines a Morse-Smale foliation, i.e. 
\begin{itemize}
\item all singularities are non-degenerate,
\item all closed orbits are hyperbolic, and
\item there are no connections between hyperbolic singularities. 
\end{itemize}

Giroux \cite{gi-conv} has proved that these properties together imply that $\Sigma$ is there is a contact vector field which is transverse to $\Sigma$. 

\begin{defn} \label{d:convex}
An embedded surface with Legendrian boundary in a contact manifold is {\em convex} if there is a contact vector field transverse to $\Sigma$. 
\end{defn}
Of course, the characteristic foliation does not have to be Morse-Smale to be convex. For example, connections between hyperbolic singularities of the same sign never prevent convexity.

Giroux has found a topological property of the characteristic foliation on $\Sigma$ which determines whether $\Sigma$ is convex or not.     
\begin{defn} \label{d:dividing}
A collection $\Gamma$ of curves and properly embedded arcs on a surface $\Sigma$ with $\partial \Gamma\subset\partial\Sigma$ {\em divides} a singular foliation if 
\begin{itemize}
\item[(i)] $\Gamma$ is transverse to $\xi(\Sigma)$ and $\Gamma$ decomposes $\Sigma$  into two subsurfaces $\Sigma_+,\Sigma_-$ (not necessarily connected),  and
\item[(ii)] there is a vector field  $X$ directing $\xi(\Sigma)$ and an area form $\omega$ on $\Sigma$ such that $L_{X}\omega>0$, respectively $L_X\omega<0$, on the interior of $\Sigma_+$, respectively $\Sigma_-$, and $X$ is pointing out of $\Sigma_+$. 
\end{itemize}
\end{defn}

\begin{thm}[Giroux \cite{gi-conv}] \label{t:char convexity} 
An oriented compact surface $\Sigma\subset (M,\xi)$ with Legendrian boundary is convex if and only if $\xi(\Sigma)$ admits a dividing set. The dividing set is unique up to isotopy through multicurves transverse to the characteristic foliation. 

A convex surface in a contact manifold has a neighborhood where $\xi$ is tight if and only if either
\begin{itemize}
\item $\Sigma\simeq S^2$ and $\xi(\Sigma)$ has a connected dividing set, or
\item $\Sigma\not\simeq S^2$ and no component of a dividing set of $\xi(\Sigma)$ bounds a disc. 
\end{itemize}
\end{thm}
If one of these two conditions in the theorem are violated for a convex surface, we will say that this surface contains an obvious overtwisted disc. 

Because of the Poincar{\'e}-Bendixson theorem, it is easier to describe the necessary and sufficient conditions on $\xi(\Sigma)$ for $\Sigma$ to being convex when $\Sigma\simeq S^2$. Recall that $\xi(\Sigma)$ is oriented, so we may think of the leaves of $\xi(\Sigma)$ as flow lines of a vector field.

\begin{lem} \label{l:convex spheres}
An embedded sphere $S^2$ in a contact manifold $(M,\xi)$ whose characteristic foliation has only isolated singularities is convex if and only if 
\begin{itemize}
\item each closed leaf of $\xi(S^2)$ is hyperbolic, and
\item there is no connection from a negative singular point to a positive singular point. 
\end{itemize}
\end{lem}

\begin{proof}
If $\Sigma$ is convex and $\Gamma$ is a dividing set, then all positive singularities are contained in $\Sigma_+$ while the negative singularities lie in $\Sigma_-$ Also, a closed leaf is hyperbolic if and only if there is a couple $(X,\omega)$ as in \defref{d:dividing} on a neighborhood of the closed leaf. Thus, in order to be convex, (i) and (ii) have to be satisfied.

For the opposite direction, recall that by the Poincar{\'e}-Bendixson theorem all limit sets of leaves of $\xi(\Sigma)$ are either closed leaves or cycles formed by finitely many leaves of $\xi(\Sigma)$ which connect singular points. Therefore, if a characteristic foliation on $S^2$ satisfies (i) and (ii) then one can construct a triple $(X,\omega,\Gamma)$ as in the proof of Proposition 2.6 of \cite{gi-conv}.    
\end{proof}
There are effective methods to manipulate the characteristic foliation on a surface, in particular when the surface is convex (see \cite{gi-conv} and Theorem 3.7 in \cite{honda}).  The following lemma can also be found in \cite{gi-conv} and does not require convexity.

\begin{lem} \label{l:elim}
Let $\Sigma\subset (M,\xi)$ be an embedded surface such that the characteristic foliation has a leaf $\gamma$ connecting two singular points with the same sign such that the index of one of the singularities is $1$ while the other has index $-1$. 

Then $\Sigma$ can be isotoped on a neighborhood of $\gamma$ such that these two singular points  disappear from the characteristic foliation on the isotoped surface and $\gamma$ remains Legendrian while no new singular points appear. 
\end{lem}
Conversely, one can also create a pair of singularities as in \lemref{l:elim}.  A more precise/complicated version of this is \lemref{l:parametric elim} below. 

There is another possible modification which is particularly relevant when considering surfaces with Legendrian boundary. The modification shown in \figref{b:split-half} preserves the boundary of a surface (represented by the thickened horizontal line), it is explained in Lemma~2.2 of \cite{elfr-long}. Note that if the hyperbolic point is positive, respectively negative, then its unstable leaves, respectively its stable leaves, have to be part of the boundary.  

\begin{figure}[htb]
\begin{center}
\includegraphics[scale=0.6]{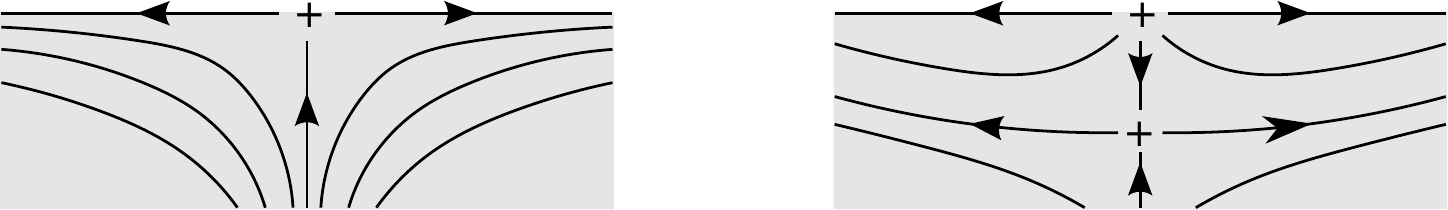}
\end{center}
\caption{Splitting hyperbolic points on a Legendrian boundary of a surface}\label{b:split-half}
\end{figure}
Finally, we note that the rotation number of a Legendrian knot $K$ with Seifert surface $\Sigma$ can be computed from $\xi(\Sigma)$ assuming all singular points of $\xi(\Sigma)$ are isolated. Let $e_\pm$, respectively $h_\pm$, be the number of singular points of index $+1$, respectively $-1$, whose sign is $\pm$ in the interior of $\Sigma$. If the singularities on the boundary are all either hyperbolic or all elliptic, then according to \cite{etnyre-overview}
\begin{equation} \label{e:rot number}
\mathrm{rot}(K)=e_+-e_--h_++h_-.
\end{equation}

\subsection{Tomographie}

If $\xi$ is a contact structure on $\Sigma\times[-1,1]$, then one can study the family $\xi(\Sigma\times\{t\}), t\in[-1,1]$ of characteristic foliations. This approach is pursued by E.~Giroux in \cite{gi-bif}, an exposition of some sapects of this can be found in \cite{geiges}. We recall some results from \cite{gi-bif} in this section. An important application of Giroux's study is a proof of Bennequin's theorem.
\begin{thm}[Bennequin] \label{t:bennequin}
A contact structure on $S^2\times [-1,1]$ such that every sphere of the product decomposition is convex and has connected dividing set is tight. 
\end{thm}

\subsubsection{Movies of characteristic foliations} \label{s:tomographie}

Let $\Sigma\times[-\delta,\delta], \delta>0,$ be the closure of a tubular neighborhood of an oriented surface in a contact manifold $(M,\xi)$. We assume that $\Sigma$ is compact and  $\partial{\Sigma}\times\{t\}$ is Legendrian for all $t\in[-\delta,\delta]$. 

\begin{defn}
The {\em movie} of $\xi$ on $\Sigma\times[-\delta,\delta]$ is the family of characteristic foliations of $\xi$ on $\Sigma\times\{t\}=\Sigma_t, t\in[-\delta,\delta]$. 
\end{defn} 
According to \cite{gi-bif} (see also \cite{geiges}) the movie determines the contact structure up to isotopy. Conversely, not every movie of singular foliation arises as movie of a contact structure: If $\xi$ is a contact structure and $\xi(\Sigma_0)$ violates one of the conditions (i),(ii) of \lemref{l:convex spheres}, then this has consequences for $\xi(\Sigma_t)$ for $t$ close to $0$. The following definitions and lemmas are due to E.~Giroux \cite{gi-bif}. 

 
\begin{defn}   \label{d:retrogradient}
A {\em retrogradient connection} is a leaf connecting a positive singular point with a negative one such that the leaf is oriented towards the positive singularity. A retrogradient connection is {\em non-degenerate} if both endpoints are non-degenerate singular points.   
\end{defn}

The following two lemmas are key for the understanding of non-convex spheres in contact manifolds:

\begin{lem}[Lemme de croisement, 2.14 in \cite{gi-bif}]\label{l:rgc}
Let $\xi$ be an oriented positive contact structure on $\Sigma\times(-1,1)$ such that there is a non-degenerate retrogradient connection in $\xi(\Sigma_0)$. For $t>0$, respectively for $t<0$, the corresponding stable leaf in $\xi(\Sigma_0)$ lies over, respectively under, the corresponding unstable leaf.  
\end{lem}
Here and in the next lemma, the words over and under refer to the coorientation of $\xi$.  

\begin{lem}[Lemme de naissance-mort, 2.12 in \cite{gi-bif}]\label{l:birth}
Let $\gamma$ be a closed leaf of the characteristic foliation on $\xi(\Sigma_0)$ which is attractive on one side and repelling on the other. If the repelling side lies above $\gamma$, then for $t>0$ small enough $\xi(\Sigma_t)$ has two non-degenerate closed orbits in a neighborhood of $\gamma$ while this neighborhood contains no closed leaves for $t<0$. If the repelling side is below $\gamma$, then it is the other way round.     
\end{lem}

These lemmas provide criteria to determine when a movie coming from a plane field cannot be the movie of a contact structure. 

The elimination lemma (\lemref{l:elim}) admits a generalization where the surface $\Sigma\times\{0\}$ is part of a movie. It takes into account how the elimination affects the characteristic foliation on nearby surfaces of the movie. We state a simplified version of Lemma~2.15 \cite{gi-bif} which suffices for the applications in this paper. 

The setup for the lemma is as follows: Let $\xi$ be a contact structure on $\Sigma\times[-1,1]$ such that $\xi(\Sigma_0)$ has an elliptic point $a_0$ and a hyperbolic point $b_0$ of the same sign (we assume positive)  and a separatix $\gamma_0$ of $b_0$ connects $b_0$ to $a_0$. This is a stable configuration, i.e. there is $0<\delta<1$ and a smooth family  $a_t,b_t$ of singular points connected by a leaf $\gamma_t$ of $\xi(\Sigma_t), t\in[-\delta,\delta]$.  Let $U$ be a neighborhood of $\cup_{t\in[-\delta,\delta]}\gamma_t$. We abbreviate $U_t=U\cap \Sigma_t$ and use $G_t$ to denote the closure of the union of closed segments of leaves of $\xi(\Sigma_t)$ both of whose endpoints lie in $U_t$.

\begin{lem}[Lemme d'elimination, Lemma 2.15 in \cite{gi-bif}] \label{l:parametric elim}
In this situation, there is a deformation of $\xi$ through contact structures with support in $U$ such that after the deformation $U_t$ contains no singular points of the characteristic foliation for $t\in(-\delta,\delta)$. 

Moreover, let $\Lambda\subset(-1,1)$ be compact and assume that for all $t\in\Lambda$ there is a defining $1$-form for the characteristic foliation on $\Sigma_t$ whose differential does not vanish on $G_t$. 

Then the deformation can be chosen such that every surface $\Sigma_t,t\in\Lambda$ which
\begin{itemize}
\item was convex before the deformation, and
\item admits a dividing set not intersecting $U_t$
\end{itemize} 
remains convex after the deformation. 
\end{lem}
Recall that by Gray's theorem a deformation of $\xi$ through contact structures keeping the product structure on $\Sigma\times[-1,1]$ fixed can be viewed as a deformation of the product structure  keeping $\xi$ fixed and vice versa. 

A simple application of \lemref{l:parametric elim} shows that a single retrogradient connection appearing in a movie of a contact structure on $S^2\times[-1,1]$ which is not obviously overtwisted (i.e. no overtwisted disc is present on one of the leaves) is tight. 

\begin{lem}  \label{l:only one}
Let $\xi$ be a contact structure on $N\simeq S^2\times[-1,1]$ (with a generic product decomposition) which is tight near the boundary such that the characteristic foliation on $S^2_t$ is convex except when $t=0$ and the characteristic foliation on $S^2_0$ has exactly one retrogradient connection between hyperbolic points.  

Then $\xi$ is isotopic to a family of contact structures on $S^2\times [-1,1]$ such that all spheres $S^2_t$ are convex and $\xi$ is tight. 
\end{lem}

\begin{proof}
Because we assume that the product decomposition is generic,  all singular points of $\xi(\Sigma_0)$ are non-degenerate.  We pick $\varepsilon_0>0$ so that this is the case for $\xi(\Sigma_t)$ for $|t|\le\varepsilon_0$.

Let $\Gamma_t^-$ be the graph whose vertices are negative singularities and whose edges are unstable leaves of negative singularities.  For $t\neq 0$ this is a connected tree because $\xi$ is tight on $S^2\times[-1,0)$ and $S^2\times(0,1]$. The graph $\Gamma^-_0$ consists of two connected components, one of these components is not closed (the $\omega$-limit set of the edge $\gamma$ forming the retrogradient saddle-saddle connection is a positive hyperbolic singularity $p_+$ and does not lie in $\Gamma_0^-$). 

By \lemref{l:rgc} both unstable leaves of $p_+$ have their $\omega$-limit set in the closed component of $\Gamma_0^-$, otherwise $\xi(S_t^2)$ would contain obvious overtwisted discs for $t>0$ or $t<0$.

Apply \lemref{l:parametric elim} to eliminate all positive hyperbolic singularities on $S_t^2$ for $t$ close to $0$. Since for $t\neq 0$ both stable leaves of positive hyperbolic singularities come from different elliptic singularities, one can choose $U_t$ so that $G_t=U_t$ for $\varepsilon/2\le|t|\le\varepsilon$ so that the condition on the existence of $1$ forms is fulfilled and $U_t$ is disjoint from a dividing set. In the same way, we eliminate all negative hyperbolic singularities. 

We can thus arrange that for the deformed contact structure
\begin{itemize}
\item there are no hyperbolic points left on $\Sigma_t$ for $|t|\le \varepsilon/2$, and 
\item $\Sigma_t$ is convex for $\varepsilon/2\le|t|\le\varepsilon$.
\end{itemize}
The resulting contact structure is convex for all $t$. Therefore it is tight by \thmref{t:bennequin}. The same is true for the original contact structure $\xi$ by Gray's theorem. 
\end{proof}

\begin{rem} \label{r:more than one in only one}
The above proof works, if $\xi(S_0^2)$ has more than one retrogradient connection as long as every connected component of the graph obtained from $\Gamma_0^-$ by removing these connections contains at least as many negative elliptic points as negative hyperbolic points. In particular, if $\xi(S_0^2)$ has exactly two retrogradient connections, then $\xi$ can either be deformed using the elimination lemma so that the resulting characteristic foliation on a disc in $S_0$ is homeomorphic to the one shown in the middle part of \figref{b:n=1} (on p.~\pageref{b:n=1}), or the retrogradient connections can be eliminated completely. 
\end{rem}

\subsubsection{Movies for contact structures on  $S^2\times[-1,1]$} 

Let $\xi_{\sigma}, \sigma\in [0,1],$  be a family of contact structures on $S^2\times[-1,1]$ where the product decomposition and the contact structure vary smoothly but are constant near the boundary of $S^2\times[-1,1]$.

We make the following genericity assumptions on the movie of $\xi_{\sigma}$ on $S^2\times [-1,1]$ with respect to product decomposition associated to the parameter value $\sigma$:  
\begin{itemize}
\item For all parameter values $\sigma$ and all $t\in[-1,1]$, the characteristic foliation on $\xi_{\sigma}(S_t^2), t\in[-1,1],$ has only finitely many singular points.
\item The parameter values where $\xi_{\sigma}(S_t^2)$ has a simply degenerate closed orbit (i.e. the first derivative of the holonomy along the close leaf vanishes, the second derivative does not)  is a union of finitely many hypersurfaces in $[-1,1]\times [0,1]$ which intersect transversely.
\item The parameter values where $\xi_{\sigma}(S_t^2)$ has a doubly degenerate closed orbit (i.e. the first and the second derivative of the holonomy along the close leaf vanishes, the third derivative does not)  is a union of finitely many submanifolds of codimension $2$ of $[-1,1]\times[0,1]$. 
\end{itemize}

By \lemref{l:rgc} and \lemref{l:birth} the submanifolds of $[-1,1]\times [0,1]$ associated to connections pointing from negative singularities to positive ones (i.e. retrogradient connections) or to degenerate closed leaves are never tangent to the foliation of $[-1,1]\times [0,1]$ given be the first factor. In particular, if the stable leaf of a positive singularity coincides with the unstable leaf of a negative singularity, then this coincidence is not degenerate.   


In Section 2.I of \cite{gi-bif}, E.~Giroux gives a beautiful proof of Bennequin's theorem based in particular on \lemref{l:rgc} 
and \lemref{l:birth}. 

The main step is his proof of the following claim: For a generic family of contact structures on $S^2\times[-1,1]$ configurations like the one shown in  \figref{b:giroux-imp} do not occur. In this figure, the $t$-coordinate corresponds to the horizontal direction, the vertical direction  represents the parameter $\sigma$ of the family. The curved lines intersecting once transversely in the box represent the locus where $\xi_{\sigma}(S_t^2)$ is not convex (by \lemref{l:rgc} and \lemref{l:birth} these lines are transverse to the horizontal direction).  The numbers on the complement of the non-convexity locus denote the number of connected components of the dividing set.

\begin{figure}[htb]
\begin{center}
\includegraphics[scale=0.6]{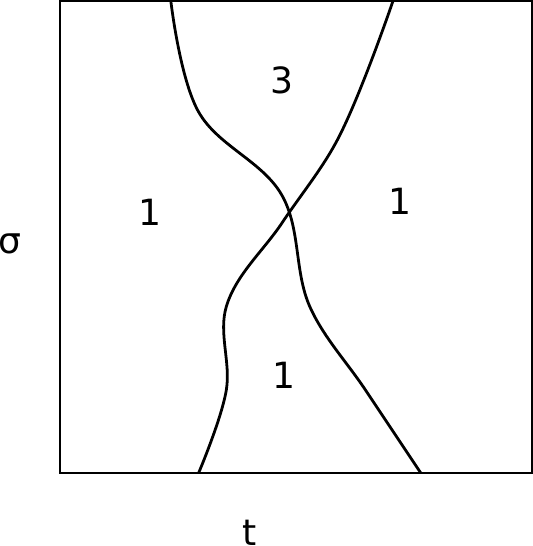}
\end{center}
\caption{Forbidden configuration according to Giroux's proof of Bennequin's theorem}\label{b:giroux-imp}
\end{figure}

Let $\xi$ be an overtwisted contact structure on $N=S^2\times[-1,1]$ which coincides with a given tight contact structure $\xi_0$  near the boundary.  We fix a Riemannian metric on $N$, and we will use the $C^2$-topology on the space $\mathrm{Cont}_\partial(N,\xi_0)$ of contact structures which coincide with $\xi_0$ on a neighborhood of the boundary. We define  $T_\pm : \mathrm{Cont}_\partial(N,\xi_0) \lra [-1,1]$ as
\begin{align}
\begin{split} \label{e:Tpm}
T_-(\xi) & =\sup\{t\in[-1,1] \textrm{ such that }\xi \textrm{ is tight on } S^2\times[-1,t)\} \\
T_+(\xi) & =\inf\{t\in[-1,1] \textrm{ such that }\xi_\sigma\textrm{ is tight on } S^2\times(t,1]\}.
\end{split} 
\end{align}
Since $\xi$ is tight near the boundary, these functions take values in $(-1,1)$.
\begin{lem} \label{l:Tpm cont}
$T_+$ and $T_-$ are continuous. 
\end{lem}

\begin{proof}
Let $\xi\in\mathrm{Cont}_\partial(N,\xi_0)$. For all $\varepsilon>0$ the set $S^2\times[-1,T_-(\xi)+\varepsilon)$ contains a transverse knot violating the Thurston-Bennequin inequality \eqref{e:tbineq t}. Therefore, there is a neighborhood of $\xi$ such that $T_-(\xi')<T_-(\xi)+\varepsilon$ for all $\xi'$ in that neighborhood. 

For all plane fields $\xi'$ on $N$ there is a unique defining form $\alpha'$ coorienting $\xi'$ such that $\|\alpha'\|\equiv 1$. Let $\alpha$ define $\xi$ and fix  a neighborhood $V$ of $\xi$ in the space of plane fields which coincide with $\xi_0$ near $\partial S^2\times[-1,1]$ such that for all $\xi'$ in $V$ the forms $\alpha_t=t\alpha'+(1-t)\alpha, t\in[0,1],$ are contact.

The isotopy $\varphi_t$, obtained from the proof of Gray's theorem (such that $\varphi^*_1\alpha'$  is a multiple of $\alpha$)  depends continuously on $\xi'$. In particular, for all $\varepsilon>0$ there is a neighborhood $U$ of $\xi$ in $\mathrm{Cont}_\partial(N,\xi_0)$ such that for all $\xi'$ in $U$ the  image of $S^2\times[-1,T_-(\xi)-\varepsilon]$  under the contactomorphism $\varphi_1^{-1} : (N,\xi') \lra (N,\xi)$  is contained in $S^2\times[-1,T_-(\xi))$ where $\xi$ is tight. Therefore, $T_-(\xi')\ge T_-(\xi)-\varepsilon$ for $\xi'$ in $U$. 

This implies $T_-(\xi')\in[T_-(\xi)-\varepsilon,T_-(\xi)+\varepsilon]$ for all $\xi'$ in the neighborhood $U\cap V$ of $\xi$. The proof for $T_+$ is analogous.   
\end{proof}

The previous lemma is true for contact structures on $\Sigma\times[-1,1]$ for all closed oriented surfaces $\Sigma$. In contrast, the next lemma follows from Giroux's proof of Bennequin's theorem (p.~\cite{gi-bif}) and is false when $S^2$ is replaced by an oriented surface of higher genus.  
\begin{lem} \label{l:T-<T+}
Let $\xi$ be an overtwisted contact structure on $S^2\times [-1,1]$ such that $\xi=\xi_0$ is tight on a neighborhood of $\partial N$. Then $T_-(\xi)\le T_+(\xi)$.
\end{lem}

\begin{proof}
Assume first that $\xi$ is generic in the sense that all singular points are isolated and retrogradient connections occur only between non-degenerate singular points of $S^2\times\{t\}$ for all $t$.

The proof of Bennequin's theorem in \cite{gi-bif} shows that if the movie of a generic contact structure on $S^2\times[-1,1]$ contains no obvious Legendrian knot (i.e. a loop formed  by leaves of the characteristic foliation on a sphere $S^2\times\{t\}, t\in(-1,1)$) which violates the Thurston-Bennequin inequality, then the contact structure is tight.

Then, by definition, $S^2\times\{T_+(\xi)\}$ contains an obvious Legendrian knot violating the Thurston-Bennequin inequality. Thus, for all $t>T_+(\xi)$, the contact structure  $\xi$ on $S^2\times[-1,t)$ is overtwisted, i.e. $T_-(\xi)\le T_+(\xi)$. 

By continuity, the claim follows for all contact structures in $\mathrm{Cont}_\partial(N,\xi_0)$. 
\end{proof}

\subsection{Non-loose knots} \label{s:non-loose}
We introduce the main subject of this paper,  non-loose knots,  and discuss some of their properties.  Recall the following theorem which follows from \thmref{t:eliashclass}.
\begin{thm}[Dymara, \cite{dym}]\label{t:loose triv}
Let $\xi$ be an overtwisted contact structure on $S^3$ and $K_0,K_1$ Legendrian knots which are smoothly isotopic, $\tb(K_0)=\tb(K_1)$ and  $\mathrm{rot}(K_1)=\mathrm{rot}(K_1)$. If there is an overtwisted disc $D$ in the complement of $K_0\cup K_1$, then there is a Legendrian isotopy between $K_0$ and $K_1$, i.e. a family of Legendrian knots $K_t,t\in[0,1]$, interpolating between $K_0$ and $K_1$. 
\end{thm}

It is therefore more interesting to study knots (or pairs of knots) which do not satisfy the conditions stated in \thmref{t:loose triv}.

\begin{defn} \label{d:non-loose}
A Legendrian or transverse knot $K$ in an overtwisted contact manifold $(M,\xi)$ is {\em loose} if $\xi$ is overtwisted on the complement of $K$. Otherwise, it is {\em non-loose}.
\end{defn}

The analogous terminology is used for Legendrian and transverse links. In \cite{elfr} non-loose knots are called exceptional. 

It is relatively easy to show that every overtwisted contact structure $\xi$ on a closed manifold $M$ admits a non-loose transverse link. For example, the binding $K$ of an open book of $M$ carrying $\xi$ is a non-loose transverse knot. Also, it is well-known that every closed contact manifold $(M,\xi)$ can be obtained by Legendrian surgery on a link in the standard tight contact structure in $S^3$. The surgery curves induce a non-loose Legendrian link in $(M,\xi)$. 

The following modified Thurston-Bennequin inequality is one of the few general results on non-loose knots. 

\begin{thm}[{\'S}wi{\k{a}}tkowski \cite{etnyre-ot}] \label{t:swiatkowski}
Let $K\subset M$ be a null-homologous Legendrian non-loose knot which is non-loose. Then 
\begin{equation} \label{e:swiatkowski}
-|\mathrm{tb}(K)|+|\mathrm{rot}(K)|\le-\chi(\Sigma)
\end{equation}
for every embedded surface $\Sigma$ with $\partial\Sigma=K$.
\end{thm}
Unlike the Thurston-Bennequin inequality \eqref{e:tbineq} for Legendrian knots in tight contact manifolds, \eqref{e:swiatkowski} does not yield an upper bound for the Thurston-Bennequin invariant of a non-loose knot. 

\begin{lem}[Etnyre, \cite{etnyre-ot}]\label{l:non-loose stable}
Let $K\subset (M,\xi)$ be a Legendrian knot in an overtwisted contact structure. If a stabilization of $K$ is non-loose, then so is $K$.
\end{lem}

\begin{proof}
We argue by contradiction. Assume that $K$ is loose and let $K'$ be a stabilization of $K$ which is non-loose. The complement of $K$ contains an overtwisted disc $D$. But every stabilization of a Legendrian knot is isotopic to a knot contained in an arbitrarily small neighborhood of $K$. Thus, we can assume that $K'$ is contained in $M\setminus D$ contradicting the tightness of $M\setminus K'$.  
\end{proof}
Finally, there is a further restriction on the Thurston-Bennequin invariant of non-loose unknots \cite{dym2}.

\begin{lem} \label{l:tb>0}
Let $K$ be a non-loose Legendrian unknot in an overtwisted contact manifold $(M,\xi)$. Then $\tb(K)>0$. 
\end{lem}

\begin{proof}
Since $K$ is an unknot, it bounds a disc $D$. Obviously $\tb(K)\neq 0$, since otherwise we could choose $D$ to be a convex overtwisted disc by Proposition 3.1 in \cite{honda}, and the complement of $K$ would not be tight. Also, if $\tb(K)<0$ we can choose $D$ to be convex, i.e. there is a tubular neighborhood $D\times[-\eps,\eps]$ such that the contact structure on this neighborhood is tight (because $\xi$ is tight on the complement of $K$, the dividing set of $D$ does not have a closed component).  Since the contact structure on $M\setminus D=M\setminus (D\times \{0\})$ is also tight, this tubular neighborhood cannot contain an overtwisted disc. Rounding the edges of $D\times[-\eps,\eps]$ we find a closed ball $B$ such that 
\begin{itemize}
\item the contact structure on $B$ is tight,
\item the contact structure on $M\setminus B$ is tight, and
\item the boundary of $B$ is a convex surface.  
\end{itemize}
According to \cite{col-connect} this implies that $\xi$ is tight contradicting the assumption that $\xi$ is overtwisted.
\end{proof}
This lemma motivates the following terminology. 
\begin{defn} \label{d:minimal}
A non-loose unknot $K$ with $\mathrm{tb}(K)=1$ is called {\em minimal}.
\end{defn}

\subsection{Examples of non-loose knots} \label{ss:fronts}

In this section we discuss examples of non-loose Legendrian knots in $S^3$. We will focus on one overtwisted contact structure which we will denote by $\xi_{-1}$ later (sometimes we will refer to it as the {\em standard overtwisted contact structure} on $S^3$). 

According to R.~Lutz \cite{lutz} $S^1$-invariant contact structures on $S^1$-principal bundles $\pi : M\lra\Sigma$ over closed oriented surfaces are classified up to $S^1$-equivariant isotopy by 
the projection $\Gamma$ of
\begin{equation} \label{e:Gamma}
\{ p\in M \, \vert \, \xi(p) \textrm{ is tangent to the fiber through }p\} 
\end{equation}
to $\Sigma$. Because of the contact condition, this is an embedded submanifold of dimension $1$ in the base surface. When $\Gamma$ is a simple closed curve in $\Sigma=S^2$ and Euler number $1$, then we obtain $\xi_{-1}$. 

Let $K$ be the fiber over a point of $p\in\Gamma$ and $p\neq q\in\Gamma$. A convenient way to establish the universal tightness of $\xi_{-1}$ on the complement of $K$ relies on techniques developed by E.~Giroux in \cite{gi-conv, gi-bif}: One studies the characteristic foliation of $\xi_{-1}$ on $T_t^2=\pi^{-1}(S^1_t)$ where $S_t^1,t\in(0,1)$, is a family of nested circles in $S^2$ such that the circles $S_t^1$ converge to $p$, respectively $q$, as $t\to 0$, respectively $t\to 1$. We choose the circles in such a way that they are transverse to $\Gamma$ and the intersection of $\Gamma$ with $S_t^1$ consists of two points for all $t\in(0,1)$. The fibers above these points form pairs of Legendrian curves and each of these  curves is either a non-singular leaf or a line of singularities of $\xi_{-1}(T^2_t)$. One of these lines is repulsive while the other is attractive. For all $\varepsilon>0$ the sheet (this is our translation of the french word {\em feuille} used in \cite{gi-bif}) of the movie of characteristic foliations on $T^2_t, t\in[\varepsilon,1),$ becomes an annulus after we add the Legendrian curve $\pi^{-1}(q)$ to $T^2_t, t\in(0,1)$ (adding this curve to $\cup_{t\in[\varepsilon,1)}T_t^2$ one obtains a solid torus). 

According to \cite{gi-bif} (Section 4.B) up to isotopy there is a unique contact structure on the solid torus 
which has a movie with the properties described above. This contact structure is universally tight for all $\varepsilon\in(0,1)$, i.e. $\xi_{-1}$ is universally tight on the complement of $K$. 

Alternatively, one can also show that $\xi_{-1}$ is universally tight on $S^3\setminus K$  by embedding the universal cover of $S^3\setminus K$ into $\R^3$ with the standard contact structure. This is the approach taken by K.~Dymara in \cite{dym, dym2}.

From the above description of $\xi_{-1}$ and $K$ one can read of $\mathrm{tb}(K)=1$. Then $\mathrm{rot}(K)=0$  by \eqref{e:sl tb}. Thus, $K$ is a minimal non-loose unknot.

In order to describe more examples of non-loose knots we use front projections. Such projections are often used to represent Legendrian knots in the standard tight contact structure on $\R^3$, many facts commonly known in that context generalize to the situation considered here with minor modifications. We refer the reader to \cite{etnyre-overview} for information about fron projections of Legendrian knots in the standard contact structure on $\R^3$.

A convenient way to describe the standard overtwisted contact structure $\xi_{-1}$ on $S^3$ is to identify the complement of the Hopf link with $T^2\times(0,1)$. The $3$-sphere is the result of the following operation: Compactify $T^2\times(0,1)$ to obtain $T^2\times[0,1]$ and then collapse the first, 
respectively second, 
factor in $S^1\times S^1\times\{i\}$ to a point when $i=0$, respectively $i=1$. The image of the boundary components of $T^2\times [0,1]$ under the quotient map to $S^3$ is a Hopf link $H$ in $S^3$.  

Front projections of Legendrian knots were studied by K.~Dymara in \cite{dym2} and the following examples are from that paper. However, our conventions are different from  those used in \cite{dym2}. The restriction of $\xi_{-1}$ to $T^2\times(0,1)$ is defined by 
\begin{equation} \label{e:ot form} 
\alpha =\sin\left(\frac{3\pi}{2}z\right)dx+\cos\left(\frac{3\pi}{2}z\right)dy.
\end{equation}
When one replaces the factor $3$ by $1$ in this formula, one obtains the tight contact structure. $\xi_{-1}$ is the contact structure  containing the minimal non-loose unknot described above. The fibers of the Hopf fibration are linear curves of slope $1$ 
inside the tori $T^2\times\{t\}, t\in(0,1),$ and the components of $H$. 

Along the Hopf link $H$ the contact planes of $\xi_{-1}$ are tangent to the discs formed by the images of 
$\{*\}\times S^1\times[1/3,1]$ and $S^1\times\{*\}\times [0,2/3]$.  All discs in these two families are overtwisted, two of them are shown in \figref{b:front1}. The Hopf link $H$ is transverse to $\xi_{-1}$ and non-loose.    

\begin{center}
\begin{figure}[htb]
\begin{center}
\includegraphics[scale=0.4]{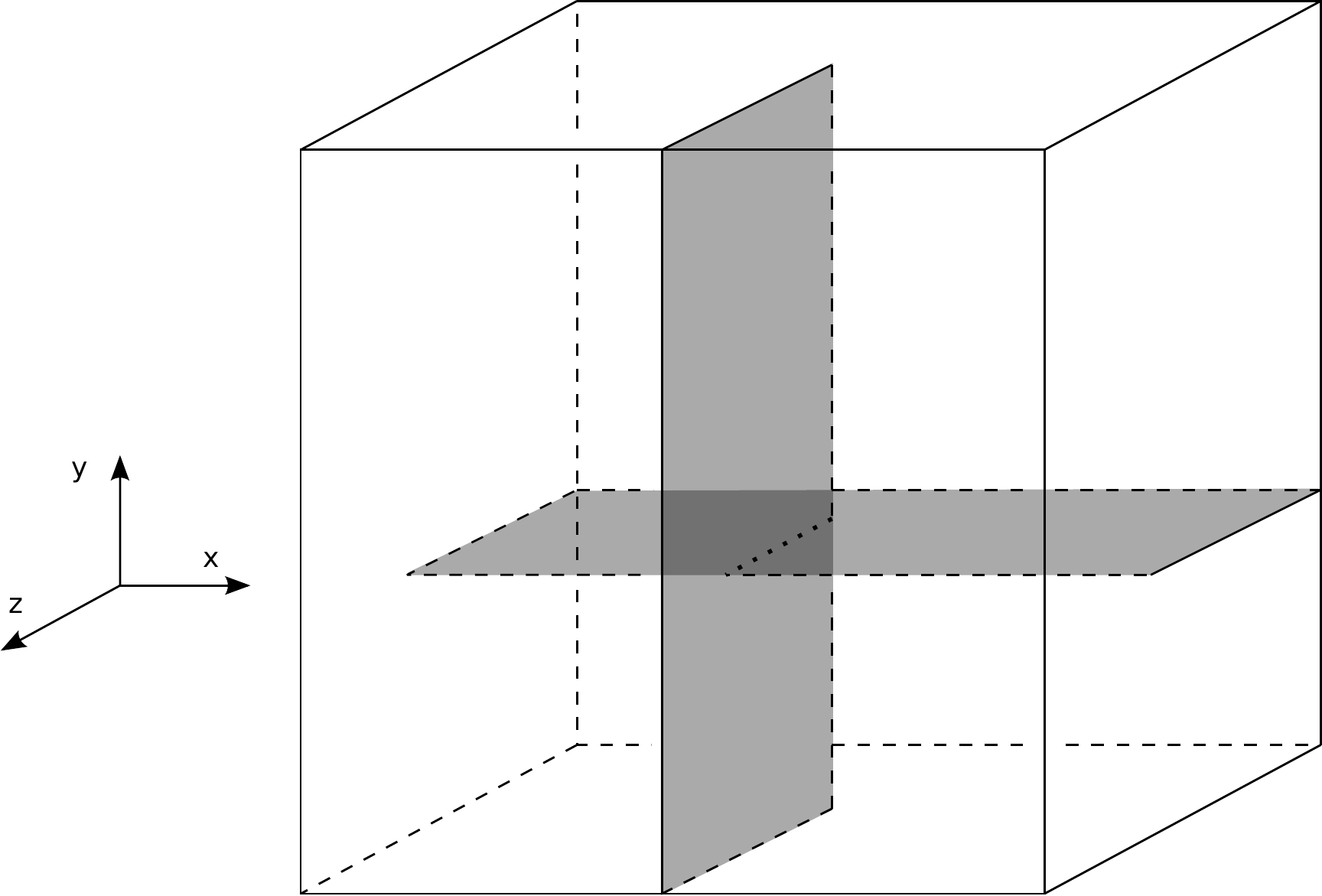}
\end{center}
\caption{Two overtwisted discs in $(S^3,\xi_{-1})$}
\label{b:front1}
\end{figure}
\end{center}

For a Legendrian curve in $S^3$ which is disjoint from the Hopf link $H$ we can consider its image under the front projection $T^2\times(0,1)\lra T^2$. Contrary to the case of the standard tight contact structure on $S^3$, a vector in the tangent space of $T^2$ does not have a unique preimage which is tangent to $\xi_{-1}$. More precisely, only lines whose slope is not negative 
 with respect to coordinates $x,y$ on  $T^2$ have a unique preimage. 
 Because of this, one has to indicate which strand passes above the other strand at a double point of a front projection. 

Let $K$ be a generic Legendrian knot in $(S^3,\xi_{-1})$ which is disjoint from $H$. The front projection of $K$ is a closed curve in $T^2$ whose homotopy class is denoted by $(a,b)\in\Z^2\simeq \pi_1(T^2)$. The only singularities of this curve are transverse double points and cusps.

Front projections of isotopic Legendrian knots are related via isotopies, Reidemeister moves (similar to the ones described in \cite{dym2}) and modifications of the front projection corresponding to the Legendrian knot crossing a component of $H$. When a segment of a Legendrian knot passing through the image of $\{z=1\}$ the diagram changes as indicated in \figref{b:pass-z=0}, when the segment passes through $\{z=1\}$ a similar (horizontal) modification of the front appears. 

\begin{center}
\begin{figure}[htb]
\includegraphics[scale=0.6]{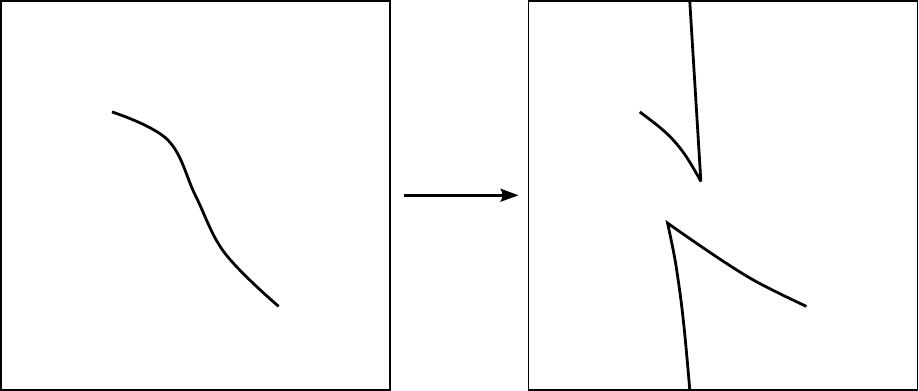}
\caption{A Legendrian segment passing through $\{z=1\}$}
\label{b:pass-z=0}
\end{figure}
\end{center}

According to \cite{dym2} the Thurston-Bennequin invariant of $K$ is
\begin{equation} \label{e:front tb}
\tb(K) = \mathrm{cross}_+(K)-\mathrm{cross}_-(K)-\frac{1}{2}\mathrm{cusp}(K)+a\cdot b. 
\end{equation}
Here $\mathrm{cross}_+$ respectively $\mathrm{cross}_-$ denote the number of positive respectively negative crossings and $\mathrm{cusp}(K)$ denotes the number of cusps of the front projection. 

As explained in \cite{dym2} the rotation number of an oriented Legendrian knot can be read of from a generic front diagram as follows:
\begin{equation} \label{e:front rot}
\mathrm{rot}(K) = \frac{1}{2}\left(\mathrm{cusp}_+(K)-\mathrm{cusp}_-(K)\right) + a - b.  
\end{equation}
Here $\mathrm{cusp}_+$, respectively $\mathrm{cusp}_-$, is the number of positive, respectively negative cusps.  We will need this formula to determine the rotation number for Legendrian knots which have projections without cusps, so we only indicate what positive/negative cusps are: A cusp is positive if the tangent space of the knot crosses $\pm \partial_z$ positively with respect to the orientation of  $\xi$ given by $\alpha$.

The non-loose Legendrian knot described at the beginning of this section is diffeomorphic to the knot $K=\{(t,t,1/2)\,|\,t\in[0,1]\}$. 
There are more non-loose Legendrian knots in $(S^3,\xi_{-1})$ which can be easily described in terms of front projections.

For coprime integers $m,n$ consider a linear curve in $T^2$ representing the homology class $(m,n)\in\Z^2\simeq H_1(T^2;\Z)$. If $mn\ge 0$
\begin{align*}
\tb(K_{m,n}) &=  mn 
& \mathrm{rot}(K_{m,n})&=m-n.  
\end{align*}  
When $n=1$ and $m\ge 1$ (or vice versa) we obtain non-loose Legendrian unknots with positive Thurston-Bennequin invariant such that the rotation number is $\pm(\tb(K_{m,n})-1)$.  

For $m=0,n=1$ or $m=1,n=0$ the curve $K_{m,n}$ lifts to  a Legendrian unknot which bounds overtwisted discs, namely one of the two discs appearing in \figref{b:front1}. Of course neither of these unknots is non-loose. However, according to Proposition 4.9 of \cite{dym2} their union is a non-loose Hopf link.

Finally, we show that a non-loose unknot $K_{m,n}$ with $m=1$ or $n=1$ and $mn>0$ can be stabilized to a Legendrian knot isotopic to $K=K_{1,1}$ (with one of its two possible orientations). In \figref{b:isotopy} we illustrate this for the case $m=1,n=2$, the general case is similar.

\begin{center}
\begin{figure}[htb]
\includegraphics[scale=0.6]{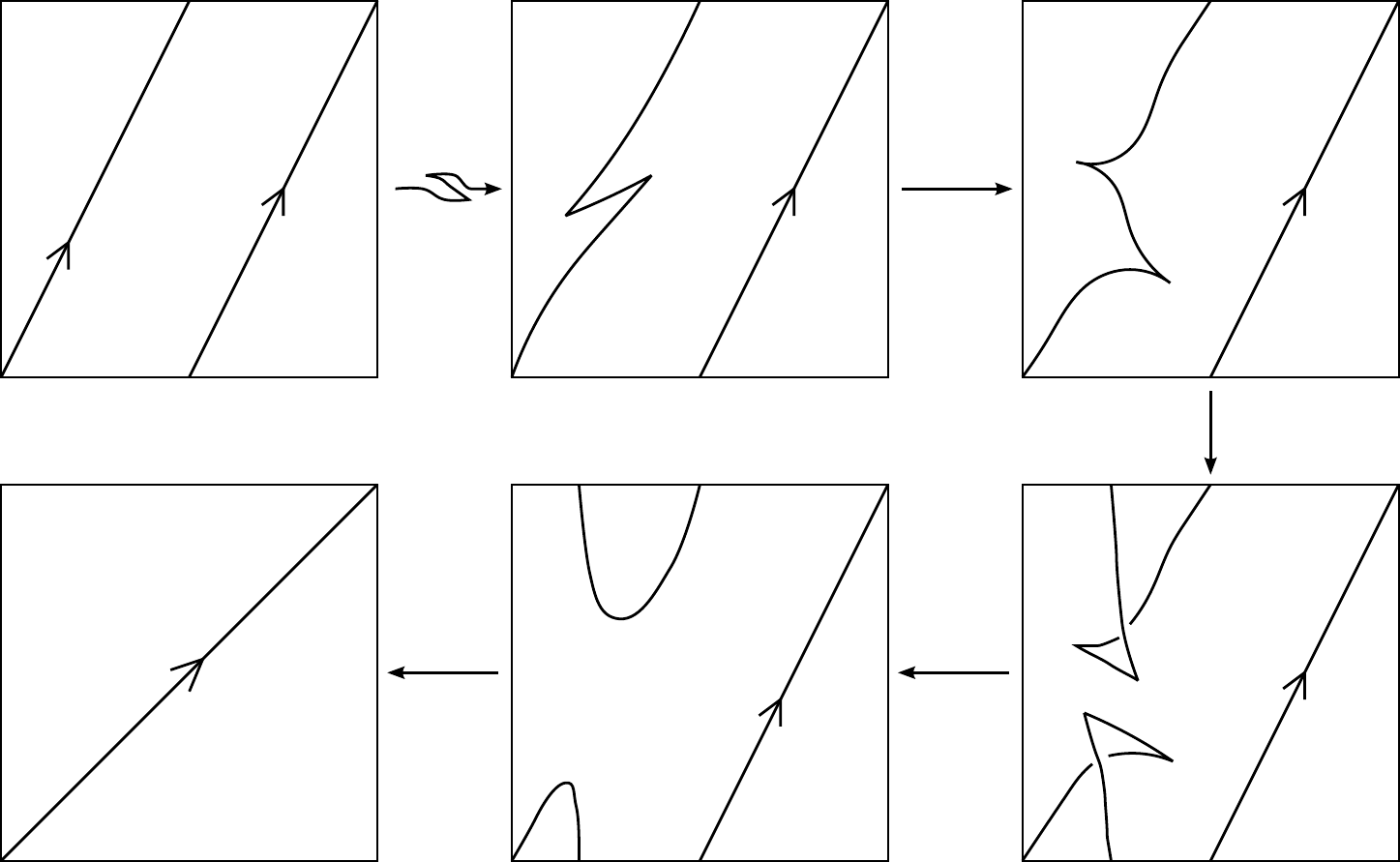}
\caption{Stabilizing $K_{1,2}$ one obtains $K_{1,1}$}
\label{b:isotopy}
\end{figure}
\end{center}

The first arrow indicates a negative stabilization, the third downward arrow indicates the move depicted in \figref{b:pass-z=0} and the fourth arrow indicates two Reidemeister moves. The last and the second arrow  correspond to isotopies.  

Finally, we note an interesting contact diffeomorphism of $\xi_{-1}$. 
\begin{ex} \label{ex:orrev}
Consider the diffeomorphism $\psi$ whose restriction  on $T^2\times(0,1)$ (the complement of the Hopf link $H$) is 
\begin{align*}
\psi\eing{S^3\setminus H} : T^2\times(0,1)=(\R^2/\Z^2)\times(0,1) & \lra T^2\times(0,1) \\
(x,y,z) & \longmapsto (-y,-x,1-z). 
\end{align*}
This map extends to a contact diffeomorphism $\psi$ of $S^3$ which preserves the orientation of $\xi$ and maps $K=K_{1,1}$ to itself but reverses the orientation of this knot. 
\end{ex}

\subsection{Homotopy theory of plane fields on $S^3$} \label{s:homotopy}

In this section we discuss homotopical properties of families of contact structures. This leads to the definition of a $\Z_2$-valued invariant $\mathrm{d}$ for contactomorphisms on the $3$-sphere.

Recall that the tangent bundle of $S^3$ is trivial. One can use a Riemannian metric and a trivialization of $TS^3$ to associate to each oriented plane field $\xi$ on $S^3$ a Gau{\ss} map 
\begin{align*}
S^3 &\lra S^2 \\
p & \lmt \textrm{positive unit vector orthogonal to } \xi.
\end{align*}
By the Thom-Pontrjagin construction homotopy classes of maps from manifolds to $m$-spheres are in one-to-one correspondence with framed submanifolds of codimension $m$ up to framed cobordisms (cf. \cite{milnor}). Usually, the Thom-Pontrjagin theorem is stated for manifolds without boundary, the case when maps are fixed on the boundary of a manifold was considered e.g. by J.~Etnyre \cite{etnyre-ot} or Y.~Huang \cite{huang} in the context of plane fields on $3$-manifolds.

This is useful since homotopy classes of plane fields on an oriented $3$-manifold $M$  correspond to homotopy classes of maps to $S^2$, when $M=S^3$ these homotopy classes correspond to elements of $\pi_3(S^2)$. The homotopy groups of the $2$-sphere that matter for us are well-known (see for example Chapter 4 of \cite{hatcher} or Chapter 5 of \cite{dfn} for a discussion based on the Thom-Pontrjagin construction): 
\begin{align*}
\pi_3(S^2)&\simeq\Z & \pi_4(S^2)&\simeq \Z_2 & \pi_4(S^3)&\simeq\Z_2 .
\end{align*}
A generator of $\pi_3(S^2)$ is the Hopf map 
\begin{align*}
S^3\subset\mathbb{C}^2 & \lra\mathbb{CP}^1\simeq S^2 \\
(z_0,z_1) &\lra [z_0:z_1],  
\end{align*}
a generator of $\pi_4(S^3)$ is the suspension of the Hopf map and a generator of  $\pi_4(S^2)$ is represented by the composition of the Hopf map with its suspension.

Recall that $M=S^3\simeq \mathrm{SU}(2)$ and consider the $\mathrm{SU}(2)$-invariant framing to define the Gau{\ss} map. Homotopy classes of plane fields on $S^3$ are distinguished by an integer $h(\xi)\in\Z$, the Hopf invariant of its Gau{\ss} map defined in \cite{hopf}. We recall one definition of the Hopf invariant using the Thom-Pontrjagin construction. 

Let $f: S^3\lra S^2$ be smooth and $p\in S^2$ be a regular value. Fix an oriented basis of $T_p S^2$. Then $f^{-1}(p)$ is a submanifold of codimension $2$ in $S^3$ with normal bundle $\left(f\eing{f^{-1}(p)}\right)^*(T_pS^2)$. This normal bundle is framed using the preimage of the basis of $T_pS^2$ under $f$. The Hopf invariant $H(f)$ is the linking number of $f^{-1}(p)$ and a push-off of $f^{-1}(p)$ in the direction of the first component of the framing. This number is independent of choices other than the homotopy class of the framing and determines $f$ up to homotopy. 

\begin{ex}
The standard contact structure $TS^3\cap iTS^3$ on $S^3\subset\mathbb{C}^2$ is $\mathrm{SU}(2)$-invariant, just like the framing. Therefore, the Gau{\ss} map is constant  and the Hopf invariant vanishes in this case.
\end{ex}

A construction which is frequently used to change the homotopy class as plane field of a contact structure is the Lutz-twist. In order to describe it, recall that a transverse knot $K$ in a contact manifold has a tubular neighborhood $N(K)\simeq S^1\times D^2(\rho)$ such that the contact structure on $N(K)$ is isomorphic to the contact structure defined by $dz+r^2 d\vartheta$. Here $D^2(\rho)$ is a disc of radius $\rho$ and $(r,\vartheta)$ are polar coordinates on $D^2(\rho)$.

For two smooth functions $f,g$ on $[0,\rho]$ such that 
\begin{align*}
f(0) & =-1  & g(r) &=-r^2 \textrm{ near } r=0 \\
f(r) & = 1  & g(r)&=r^2 \textrm{ near } r=\rho 
\end{align*}
and $fg'-g'f>0$ we can define the contact structure
$$
\xi'=\left\{\begin{array}{ll} \ker(f(r)dz+g(r)d\vartheta) & \textrm{ on }S^1\times D^2(\rho) \\
\xi & \textrm{ on the complement of }N(K).
\end{array}\right.
$$
\begin{defn} \label{d:Lutz twist}
We say that $\xi'$ is the result of a {\em $\pi$-Lutz twist} along $K$. 
\end{defn}
The contact structure $\xi'$ is well-defined up to isotopy. When one reverses the orientation of $K$, the resulting knot is positively transverse to $\xi'$. 

In general, $\xi'$ is not homotopic to $\xi$ even as a plane field. In  Chapter 4.3 of \cite{geiges} it is explained how to determine the difference between the homotopy classes of $\xi$ and $\xi'$. What is relevant for us is that applying a $\pi$-Lutz twist along a null-homologous transverse knot $K$ in $S^3$ changes the Hopf invariant by $\mathrm{sl}(K)$.  The usual Lutz twist (or $2\pi$-Lutz twist) corresponds two consecutive $\pi$-Lutz twists applied to the same knot, it does not change the homotopy type of the plane field.

\begin{rem} \label{r:obvious ot}
The contact structure obtained by a $\pi$-Lutz twist along a transverse knot is always overtwisted since the discs $\{z\}\times D^2(\rho_0)$ are overtwisted when $(f(\rho_0),g(\rho_0))$ lies on the positive part of the $x$-axis ($\rho_0$ is uniquely determined because of $fg'-f'g>0$). 
\end{rem}

\begin{ex} \label{ex:Lutz}
Let $\xi=TS^3\cap iTS^3$ be the standard contact structure on $S^3\subset\mathbb{C}^2$. This contact structure is $\mathrm{SU}(2)$-invariant. The fibers of the Hopf fibration are transverse curves, each of them has self-linking number $-1$. After a $\pi$-Lutz twist along $k$ distinct fibers, one obtains a contact structure with Hopf invariant $(k-2)k$ (c.f. p.47 in \cite{elfr-long}). 

For $k=1$ we obtain the contact structure which we are interested in most in this article. 
\end{ex}

When the plane field varies continuously the same is true for the Gau{\ss} map. Therefore,  homotopy classes of $1$-parameter families $\xi_t, t\in[0,1],$ of coorientable plane fields on $S^3$ are in one-to-one correspondence with  the set of homotopy classes of maps $g: S^3\times [0,1]\lra S^2$ (with fixed boundary conditions).  This set is isomorphic to $\pi_4(S^2)\simeq \Z_2$  if $\xi_0=\xi_1$ and the framing on $S^3$ is chosen so that $g(\cdot,0)=g(\cdot,1)$  is constant. This is the situation we consider from now on. 

Following \S 23.4 in \cite{dfn} we  review how to determine to which element of $\pi_4(S^2)$ a given map $g : S^4\lra S^2$ corresponds. 


We assume that $g$ is smooth and pick a regular value $p$. Then  $\Sigma=g^{-1}(p)\subset S^4$ is a framed submanifold of codimension $2$.  For each oriented simple closed curve $\gamma$ in $\Sigma$ consider its normal vector field in $\Sigma$. Then $\gamma$ can be viewed as framed submanifold of $S^4$ of codimension $3$, so it represents an element $\phi(\gamma)$ in $\pi_4(S^3)\simeq\Z_2$. This element depends only on the homology class $\gamma\in H_1(\Sigma;\Z_2)$ and we have defined a map 
\begin{equation} \label{PHI}
\phi : H_1(\Sigma;\Z_2)  \lra \Z_2. 
\end{equation}
Given two simple closed curves $\gamma,\gamma'$ in $\Sigma$ which intersect transversely we can replace $\gamma\cup\gamma'$ by a collection of simple closed curves representing the homology class $\gamma+\gamma'$ by smoothing the intersection points. One then has
$$
\phi_g(\gamma+\gamma') = \phi_g(\gamma)+\phi_g(\gamma') + \gamma\cdot\gamma' \quad \mod 2
$$
where $\gamma\cdot\gamma'$ is the intersection pairing. 

Thus, $\phi_g : H_1(\Sigma;\Z_2) \lra \Z_2$ is a non-degenerate quadratic form.  Recall that the Arf invariant of a $\Z_2$-valued quadratic form on $H_1(\Sigma;\Z_2)$ takes values in $\Z_2$ and is non-trivial if and only if more than one half of the elements $\gamma\in H_1(\Sigma;\Z_2)$ have $\phi_g(\gamma)=1$. One can show that the Arf invariant of $\phi_g$ depends only on the homotopy class of $g : S^4 \lra S^2$.

K.~Dymara \cite{dym} used $\pi_4(S^2)\simeq\Z_2$ to define a continuous group homomorphism 
$$
\mathrm{d} : \mathrm{Diff}_+(S^3,\xi) \lra \Z_2
$$
on the group of orientation preserving contactomorphisms. This homomorphism can be used to show that the group of coorientation preserving contactomorphisms $\mathrm{Diff}_+(S^3,\xi)$ is not connected when $\xi$ is overtwisted.

We recall the definition from \cite{dym}. For $\psi\in\mathrm{Diff}_+(S^3,\xi)$ choose a family of diffeomorphisms $\psi_t$ such that $\psi_0=\mathrm{id}$ and $\psi_1=\psi$. Such a family exists because $\psi$ is orientation preserving and the group of orientation preserving diffeomorphisms of $S^3$ is connected according to Cerf's theorem \cite{cerf}.

Consider the loop $\psi_{t*}(\xi)$  in the space of oriented plane fields on $S^3$. (This is even a loop in the space of contact structures on $S^3$, but we ignore this fact.) Since the Euler class of $\xi$ in $H^2(S^3;\Z)=0$ vanishes, we can pick a trivialization of $\xi$ and extend it to a framing of $S^3$. 

Applying the Gau{\ss} map we obtain a loop in the space of maps from $S^3$ to $S^2$ based at the constant map. This loop  can be viewed as map $S^4\lra S^2$. Now define 
\begin{align*}
\mathrm{d} : \mathrm{Diff}_+(S^3,\xi) &  \lra \pi_4(S^2)\simeq \Z_2 =\{0,1\} \\
\psi &\longmapsto \left[\left(\psi_{t*}(\xi)\right)_{t\in[0,1]}\right].
\end{align*}
It is proved in \cite{dym} that this is a well-defined homomorphism on the group of connected components of $\mathrm{Diff}_+(S^3,\xi)$.  The proof in \cite{dym} that $\mathrm{d}$ is well-defined contains a minor gap since it is assumed that every contact structure  on $S^3$ is contactomorphic to a contact structure which is invariant under the standard $S^1$-action on $S^3$. However, the proof idea still works. In order to see this, one constructs a contact structure with Hopf invariant $k$ using $\pi$-Lutz twists along a collection of knots transverse to the standard contact structure  with self linking number $-1$ which are invariant under the contact diffeomorphism $(w,z)\longmapsto(w,-z)$ of $S^3\subset\mathbb{C}^2$.  

The following diagram summarizes the various groups involved in the construction of the invariant $\mathrm{d}$.
$$
\xymatrix{
\pi_1(\mathrm{Diff}(S^3,\mathrm{or}))=\Z_2  \ar[r] &  
\pi_1(\mathrm{Cont}(S^3,\xi)) \ar[r]_\delta \ar[d]  &   \pi_0(\mathrm{Diff}_+(S^3,\xi)) \ar@/_1pc/[l] \\
\mbox{} &
\pi_1(\mathrm{Distr}(S^3,\xi)) \ar[r]^{\textrm{Gau{\ss}-map}}_{\textrm{+  collapse}} & 
\pi_4(S^2)=\Z_2 
}
$$
Here $\mathrm{Distr}(S^3,\xi)$, respectively $\mathrm{Cont}(S^3,\xi)$,  denotes the plane fields homotopic to $\xi$, respectively the contact structures isotopic to $\xi$. The group of orientation preserving diffeomorphisms of $S^3$ is  $\mathrm{Diff}(S^3,\mathrm{or})$. It is homotopy equivalent to $\mathrm{SO}(4)$ by \cite{smale conj}. 

The upper line is part of the long exact sequence of the fibration $\mathrm{Diff}(S^3,\mathrm{or}) \lra \mathrm{Cont}(S^3,\xi)$ given by $\varphi\longmapsto \varphi_*(\xi)$, the vertical arrow is induced by the inclusion and the map  $\pi_0(\mathrm{Diff}_+(S^3,\xi))\lra\pi_1(\mathrm{Cont}(S^3,\xi))$ is the section of $\delta$ chosen above. 
 
\begin{rem} \label{r:dym surj}
When $\xi$ is overtwisted, it follows from \lemref{l:families}  that there is a family of contact structures $\xi_t,t\in[0,1],$ with $\xi=\xi_0=\xi_1$  representing a non-trivial loop in the space of pane fields on $S^3$. By Gray's theorem there is a family $\psi_t$ of diffeomorphisms of $S^3$ such that $\psi_{t*}\xi=\xi_t$. Then  $\mathrm{d}(\psi_1)=1$, i.e. $\mathrm{d}$ is surjective. 

If one chooses the homotopy $\xi_t$ so that it is constant on a closed ball $B$ containing an overtwisted disc $\Delta$ such that $S^3\setminus B$ is still overtwisted one can achieve that the support of $\psi_1$ is contained in the complement of $B$.
\end{rem}

\section{The classification of minimal non-loose unknots in $(S^3,\xi)$} \label{s:class}

In this section we prove a basic result of this paper:  We classify minimal non-loose Legendrian unknots in $S^3$ up to Legendrian isotopy. For this, we first review the classification of non-loose unknots up to contact diffeomorphism from \cite{elfr-long, etnyre-ot} where  it is shown that every non-loose unknot $K$ with $\tb(K)=1$ in $S^3$ carrying an overtwisted contact structure is diffeomorphic to the example $K$ discussed at the beginning of \secref{ss:fronts}. Using
\begin{itemize} 
\item the coarse classification,
\item some particular contact diffeomorphisms constructed in \secref{ss:contiso}, and
\item Eliashberg's theorem on overtwisted contact structures
\end{itemize}
we prove in  \secref{s:non or} that $K$ is isotopic to any other minimal non-loose unknot. However, this isotopy does not respect orientations of these knots in general. The proof that $K$ and $\overline{K}$ are not isotopic as oriented Legendrian can be found in \secref{s:or} (here $K$ is a minimal non-loose unknot and $\overline{K}$ is the dame knot with the reversed orientation). 

Except for the coarse classification we will consider only non-loose Legendrian unknots $K$ with $\tb(K)=1$ in this section. Non-loose Legendrian unknots with higher Thurston-Bennequin number will be classified in \secref{ss:Leg class}.

\subsection{The coarse classification} \label{ss:coarse}
In this section we recall the coarse classification of non-loose unknots in $S^3$, i.e. the classification up to contactomorphism from \cite{elfr-long} and \cite{etnyre-ot}.  The following proposition summarizes the information which is needed to put a Seifert disc of a non-loose Legendrian unknot in a standard form. This implies that $(S^3,\xi,K)$ is diffeomorphic to $(S^3,\xi_{-1},K_{1,1})$ described in \secref{ss:fronts} and that this diffeomorphism can be chosen to preserve fixed orientation of the contact structures.  

The proposition is a consequence of the classification of tight contact structures on the solid torus (from \cite{gi-bif,honda}), applied to the complement of a tubular neighborhood of a non-loose unknot in $S^3$ \cite{elfr-long}

\begin{prop} \label{p:n-1} 
Let $\xi$ be a contact structure on $S^3$ and $K\subset S^3$ a non-loose unknot with $\tb(K)=n>0$. Then $\mathrm{rot}(K)=\pm(n-1)$. 
\end{prop}

From this we now deduce the coarse classification of non-loose unknots in $S^3$. The definition of $K_{m,n}$ can be found in \secref{ss:fronts}.

\begin{thm} \label{t:coarse unknot general}
Let $\xi$ be an overtwisted contact structure on $S^3$ and $K$ a non-loose Legendrian unknot with $\tb(K)=n>0$ and $\mathrm{rot}(K)=\pm(n-1)$. Then $(S^3,\xi,K)$ is diffeomorphic to $(S^3,\xi_{-1},K_{\pm 1,\pm n})$. 
\end{thm}

 We give the proof of this because it yields a normal form for the characteristic foliation on the Seifert surface of a non-loose unknot which will be used later. 

\begin{proof}
Let $\xi$ be an overtwisted contact structure on $S^3$ and $K$ a non-loose oriented Legendrian unknot. We assume $\tb(K)=n>0$ because of \lemref{l:tb>0}.  Choose an oriented spanning disc for $K$. We will simplify the characteristic foliation on $D$ to bring it in a standard form. For this we assume that $D$ is generic so that the singular points of $\xi(D)$ are either elliptic or hyperbolic. 

We first consider the singular points of $\xi(D)$ along $\partial D$ and simplify the characteristic foliation as follows. 
\begin{itemize}
\item Positive hyperbolic points on $\partial D$ whose unstable leaves lie on $\partial D$ are replaced by positive elliptic points (cf. \figref{b:split-half} on p.~\pageref{b:split-half}).
\item Negative hyperbolic points on $\partial D$ whose stable leaves lie on $\partial D$ are replaced by negative elliptic points. 
\end{itemize}

After this, a negative (positive) elliptic singularity is connected  to singularities which are positive (negative) elliptic  or  negative (positive) hyperbolic. The second case can be further simplified using the elimination lemma which provides a deformation of $D$ canceling a hyperbolic point with an elliptic singularity of the same sign while keeping the boundary of $D$ fixed  throughout the deformation.

From now on we assume that there are no canceling pairs of singularities left on $\partial D$. But then, if there is an elliptic singularity left, all singular points on $\partial D$ are elliptic, they alternate between negative and positive along $\partial D$, and the Thurston-Bennequin invariant of $K$ is negative (this can be seen by a direct computation in a model). If $\partial D$ contains no singular points, then $\tb(K)=0$ which is impossible for a non-loose Legendrian unknot in an overtwisted contact manifold. Therefore, $\xi(D)$ has only hyperbolic points on the boundary which alternate between positive and negative and have retrogradient connections between them.

The interior of $D$ can be simplified further. Generically, there are no connections between hyperbolic points such that the connecting separatrix lies in the interior of $D$. Then  every unstable leaf of a negative hyperbolic singularity has a negative elliptic singularity as its $\omega$-limit set since $\xi$ is tight on $M\setminus K$. Therefore, all negative hyperbolic points of $\xi(D)$ in the interior of $D$ can be eliminated and the same is true for the positive hyperbolic points in the interior of $D$. 

We assume $\mathrm{rot}(K)\ge 0$. Because $K$ is non-loose, $\xi(D)$ has no closed leaf in the interior of the disc which could act as sink or source. But only finitely many leaves are stable or unstable leaves of hyperbolic singular points. So $\xi(D)$ must have at least one positive and at least one negative singular point in the interior of $D$ to act as sink or source for infinitely many leaves of $\xi(D)$. According to \eqref{e:rot number} (on p.~\pageref{e:rot number})
$$
\mathrm{rot}(K)=e_+-e_-=n-1. 
$$ 
Therefore, there are at least $n$ positive elliptic singularities. But there cannot be more since on the boundary, there are exactly $n$ singular points whose unstable leaves can come from positive elliptic singularities. If no leaf coming from an elliptic singular point ends at a hyperbolic point, then the basin of this elliptic singularity is either a sphere or it is the entire disc. But this is not the case since the disc $D$ contains  elliptic singular point of both signs.

It follows that $e_+=n$ (and $e_-=1$) and every positive elliptic singular point is connected to a negative hyperbolic point in $\partial D$. All unstable leaves of positive hyperbolic on $\partial D$ end at the same negative elliptic singular point  in the interior. This determines the characteristic foliation on $D$ up to homeomorphism (c.f. the middle part of \figref{b:n=1} for $n=1$).
\end{proof}

\begin{defn} \label{d:p-}
Assume that $S^2\subset(S^3,\xi)$ such that the characteristic foliation has only isolated singularities.  
Let $\Gamma_-^*$ be $\Gamma_-$ with unstable leaves connecting to positive hyperbolic singularities removed and $\Gamma'$ a connected component of $\Gamma_-^*$. By $e_-(\Gamma')$, respectively $h_-(\Gamma')$, we denote the number of singularities of $\xi(S^2)$ in $\Gamma'$ whose index is $+1$, respectively $-1$.  Using $\Gamma_+$ instead of $\Gamma_-$  one defines $e_+(\Gamma'),h_+(\Gamma')$ similarly.
\end{defn}

In general, if $\Gamma_{t,-}$ is a tree for all $t\neq t_0$ close to $t_0$ and some unstable leaves of negative  singularities take part in a retrogradient connection on $S^2\times\{t_0\}$, then for each connected component $\Gamma'$ of $\Gamma_{t_0,-}^*$ 
\begin{equation} \label{e:chi Gamma'}
e_-(\Gamma')-h_-(\Gamma') = 1 - \textrm{number of retrogradient connections starting at }\Gamma'.
\end{equation}
Since $\Gamma_t$ is a tree for $t\neq t_0$, the right hand side of \eqref{e:chi Gamma'} is $\pm 1$ or $0$. 
 
Consider the particular case when $S^2\times\{t_0\}$ contains a non-loose unknot $K$ with $\mathrm{tb}(K)=1$. From the proof of \thmref{t:coarse unknot general} it follows that 
$$
e_-(\Gamma')-h_-(\Gamma')=\pm 1
$$
and that $\Gamma'$ is a tree for all components $\Gamma'$ of $\Gamma_-^*$. There is exactly one component $\Gamma'$ such that $e_-(\Gamma')-h_-(\Gamma')=-1$ and this is the component where both unstable leaves forming the retrogradient connections on $S^2$ start.

The converse of this statement follows immediately from the proof of \thmref{t:coarse unknot general}. 
\begin{cor} \label{c:non-loose exist}
Let $S^2$ be a sphere in a contact manifold $(M,\xi)$ such that the characteristic foliation has 
\begin{itemize}
\item precisely two retrogradient connections,
\item the corresponding unstable leaves start at the same connected component of $\Gamma_-^*$, 
\item $\Gamma_-^*$ has exactly three connected components which are all trees, and
\item $\xi(S^2)$ has no closed leaf. 
\end{itemize}
Then $\xi(S^2)$ contains a minimal unknot which is non-loose in a small tubular neighborhood of the sphere. 
\end{cor}

\begin{rem}
Recall that the positive/negative stabilization of a non-loose Legendrian unknot $K$ with $\tb(K)>1$ results in a non-loose Legendrian unknot $K'$ with $\tb(K')=\tb(K)-1$ and $\mathrm{rot}(K')=\mathrm{rot}(K)\pm 1$.  

By \thmref{t:coarse unknot general} the positive stabilization of $K_{1,n}$ is loose, and we have seen before that the negative stabilization of $K_{1,n}$ is non-loose when $n>1$.  
\end{rem}

\begin{center}
\begin{figure}[htb]
\includegraphics[scale=0.7]{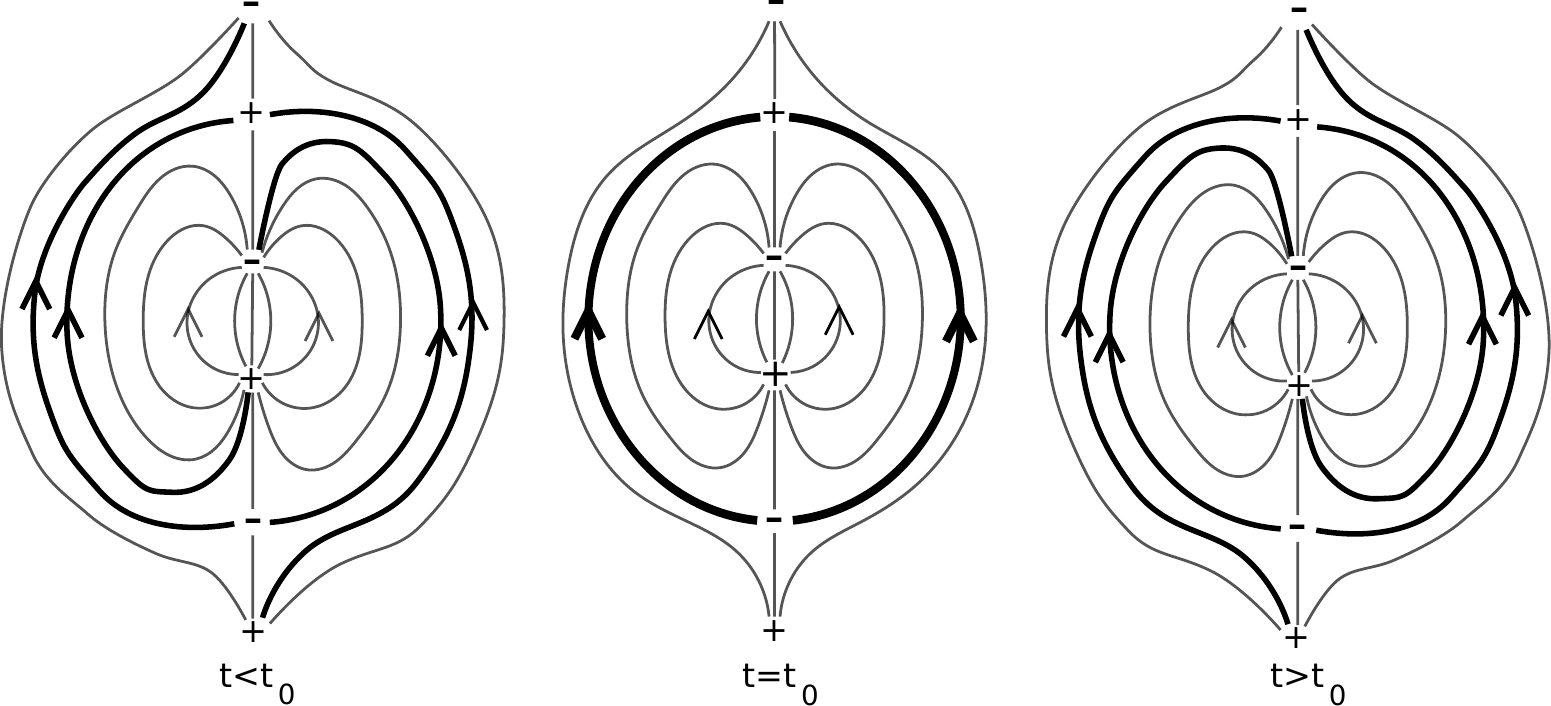}
\caption{Examples of characteristic foliations on spheres.  The thickened lines in the middle figure form a minimal non-loose unknot.}
\label{b:n=1}
\end{figure}
\end{center}

Assume that $S^2\times\{t_0\}$ contains a non-loose unknot and that the characteristic foliation on each disc bounding the knot in  $S^2$ is simplified as in the proof of the previous Theorem, i.e. there are exactly two hyperbolic singularities and four elliptic ones. It is interesting to consider the characteristic foliation on nearby spheres. \figref{b:n=1} shows the characteristic foliation on $S^2\times\{t\}$ for $t<t_0$ (on the left) and $t>t_0$ (on the right) where $t$ is close to $t_0$. \lemref{l:rgc}  describes the behavior of stable/unstable leaves of the two hyperbolic singularities, and the presence of a connection between two hyperbolic points is a codimension $1$ phenomenon. In view of this, the presence of a minimal non-loose unknot of $(S^3,\xi_{-1})$ on an embedded sphere is a codimension two phenomenon. 

Overtwisted discs are not directly visible in \figref{b:n=1} although there is an obvious Legendrian knot violating the Thurston-Bennequin inequality.

In order to see  overtwisted discs appear more explicitly, we will use a  schematic representation of the characteristic foliation on spheres. This representation will be employed later, too.  
Consider a tubular neighborhood $S^2\times(t_0-\varepsilon,t_0+\varepsilon), \varepsilon>0,$ of an embedded sphere $S^2=S^2_{t_0}=S^2\times\{0\}$ in a contact manifold $(M,\xi)$ and assume that
\begin{itemize}
\item $\xi(S^2_{t_0})$ has a retrogradient saddle-saddle connection,
\item the product decomposition on the tubular neighborhood is generic  in the sense that all singular points of $\xi(S^2_{t_0})$ are non-degenerate, and
\item $\xi(S^2_{t_0})$ has no closed leaf.
\end{itemize}  
Note that we do not require that there is only one retrogradient connection for $t=0$ (this case is not really interesting because of \lemref{l:only one}). 

As in the proof of \lemref{l:only one} we consider the graph $\Gamma_{t,-}$ consisting of unstable leaves of negative points of $\xi(S^2_t)$. Solid lines will represent the connected components obtained from $\Gamma_{t,-}$ by removing all stable leaves which do take part in a retrogradient connection at $t=0$. We are only interested in the homotopy type of these components. In all cases under consideration these components will be simply connected, and we represent them by solid lines. 

The unstable leaves which do take part in a retrogradient connection at $t=0$ are dashed. We will always draw the graph before and after  the retrogradient connections occur, and we will use dotted lines to indicate where a stable leaf of $\Gamma_{t,-}, t<0,$ which takes part in the retrogradient connection ends after a retrogradient connection in the future.  

From such a diagram one can read of the number of connected components of $\Gamma_{t,-}$ and their Euler characteristics before and after one or more retrogradient connections occur.  Although $\Gamma_{t,+}$ is not represented, a lot of information about $\Gamma_{t,+}$ can be recovered from the diagrams:  Under our assumptions, $\Gamma_{t,+}$ is homotopy equivalent to $S^2\setminus\Gamma_{t,-}$ for $t\neq 0$. In particular, these diagrams can be used to determine whether  $S^2_t$ contains an overtwisted disc of $\xi$.

Let $S^2_{t_0}\subset(S^2\times[-1,1],\xi)$ be a sphere whose characteristic foliation is homeomorphic to the one shown in \figref{b:n=1}. We want to study nearby spheres $S^2_t$ of the product decomposition.  The union of the dashed and solid lines in the top diagram represents $\Gamma_{t,-}$ for $t<t_0$ while the dotted lines represent the future positions of the unstable leaves in the retrogradient connection. The bottom diagram shows $\Gamma_{t,-}$ for $t>t_0$. The solid straight horizontal line in the middle of each of the four diagrams in \figref{b:non-loose} corresponds to the negative hyperbolic singularity in \figref{b:n=1}, each of the other two solid arcs represents one of the negative elliptic singularities in \figref{b:n=1}.

The situation when two retrogradient connections occur at the same time in a movie is unstable: After a small perturbation of the product structure on $S^2\times[-1,1]$, the retrogradient connections (marked with $A,B$) will still occur, but they will lie on different spheres $S^2\times\{t_A\},S^2\times\{t_B\}$.  The left diagram in the middle row represents $\Gamma_{t,-}$ for $t\in(t_A,t_B)$ for $t_A<t_B$. Finally, the diagram on the right-hand side is analogous for the case $t_B<t_A$. 
In both cases there are obvious overtwisted discs.


\begin{center}
\begin{figure}[htb]
\includegraphics[scale=0.8]{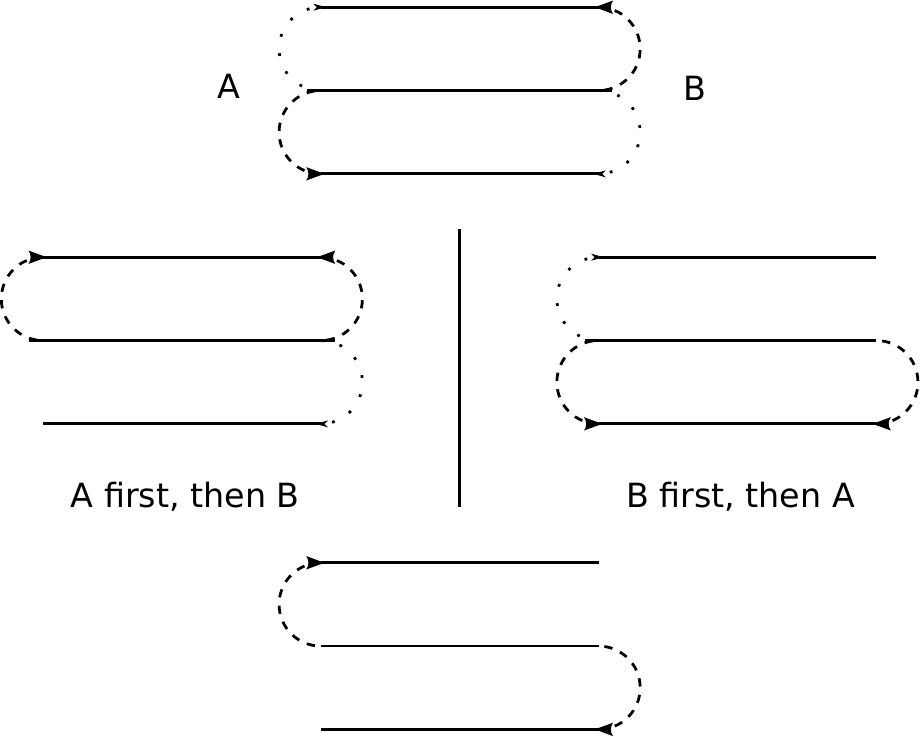}
\caption{Schematic representation of the movie on the spheres in a product neighborhood of a sphere carrying a minimal non-loose Legendrian unknot with two possible resolutions.}
\label{b:non-loose}
\end{figure}
\end{center}

In the following remark we consider minimal non-loose unknots on spheres such that one of the singular points of the characteristic foliation is degenerate. The assumptions made here are very specific and tailored to the application we have in mind.  
\begin{rem}  \label{rem:degenerate sing}
Assume that a minimal non-loose unknot is contained in an embedded sphere $S^2_{t_0}$  containing exactly two retrogradient connections and that all singular points of $\xi(S^2_{t_0})$ are isolated.  Recall that the index of such singularities has to be $\pm 1$ or $0$ and fix a tubular neighborhood $N=S^2\times[t_0-\varepsilon,t_0+\varepsilon]$. We assume that $N$ is so thin that $\xi(S^2_{t_0})$ is Morse-Smale for all $t\neq t_0$. 

Later, it will be relevant  to consider the case when there is a single degenerate negative singularity $x$ in $\Gamma'$ with index zero such that one of its unstable leaves takes part in the retrogradient connection with the positive singularity $z$. Then $e_-(\Gamma')-h_-(\Gamma')=-1$ and $x$ is the endpoint of an unstable leaf $\gamma'$ of a hyperbolic singularity $x'$ in $\Gamma'$.

If one perturbs the product structure on $N$ slightly relative to a neighborhood of $\partial N$, then there is a retrogradient connection on one of the spheres of the perturbed product decomposition involving the same stable leaf of $z$ but it may happen that  this stable leaf is now the unstable leaf of $x'$.   

Of course, the discussion when a positive singular point is degenerate is completely analogous. 
\end{rem}

\subsection{Contact diffeomorphisms preserving non-loose unknots} \label{ss:contiso}
For the proof of our classification results we need  a contact diffeomorphism $$
\psi : (S^3,\xi_{-1})\lra (S^3,\xi_{-1})
$$ which preserves $K$, the orientation of the plane field and has $\mathrm{d}(\psi)\neq 0$. This will be the content of the second example in this section, the first is a preparation for the second. We will use the consequences of the Thom-Pontrjagin construction which were outlined in \secref{s:homotopy}.

 Before explaining the construction, we note that some of the material presented here could be rephrased using the methods developed by Y.~Huang in \cite{huang-ot, huang}.

First, we consider how a particular operation on a given contact structure affects the Hopf invariant. This construction will yield an alternative construction of a non-loose unknot in $S^3$.
  
\begin{ex} \label{ex:homotopy ex}
Start with the standard contact structure $\xi_{st}=TS^3\cap iTS^3$ on $S^3\subset\mathbb{C}^2$. It is $\mathrm{SU}(2)$-invariant, and we use it to orient $S^3$. We choose an $\mathrm{SU}(2)$-invariant framing to define the Gau{\ss} maps (the Gau{\ss} map of $\xi_{st}$ is constant).

Consider the transverse unknot $H=S^1\times\{0\}$  together with a tubular neighborhood  $N(H)\simeq S^1\times D^2$ such that the contact structure on this neighborhood is invariant under translations in the $S^1$-direction and under rotations of the disc. 

We fix a round $2$-sphere $S_0\subset S^3$ which is orthogonal to $\mathbb{C}\times\{0\}\subset \mathbb{C}^2$ such that $S_0\cap N(H)\supset \{1\}\times D^2$ (here $1=(1,0)\subset\mathbb{C}^2$ is a point on $H$) and we orient $S_0$ so that $\xi_{st}(S_0)$ has a positive singular point in the midpoint of the disc $\{1\}\times D^2$. Then $S_0$ is convex and $\xi_{st}(S_0)$ has exactly two singular points (both elliptic), these are the intersection points of $S_0$ with $L$.  

We isotope $\xi_{st}$ to a new contact structure $\xi$ with the following properties: 
\begin{itemize}
\item $\xi$ is $S^1$-invariant on $N(H)\simeq S^1\times D^2 $.
\item The characteristic foliation  of $\xi$ on $\{1\}\times D^2$ is diffeomorphic to the one shown in \figref{b:isolated hyperbolic}. In particular, there is a disc $D_{in}$ such that $\partial D_{in}$ is transverse to $\xi$ and $\xi(D_{in})$ has two positive elliptic singularity, one negative hyperbolic one and no closed leaves. 
\item There is an involution $\rho_1$ of $D_{in}$ which preserves the singular foliation on $D_{in}$ but interchanges the two positive elliptic points.
\end{itemize} 
  
\begin{center}
\begin{figure}[htb]
\includegraphics[scale=0.6]{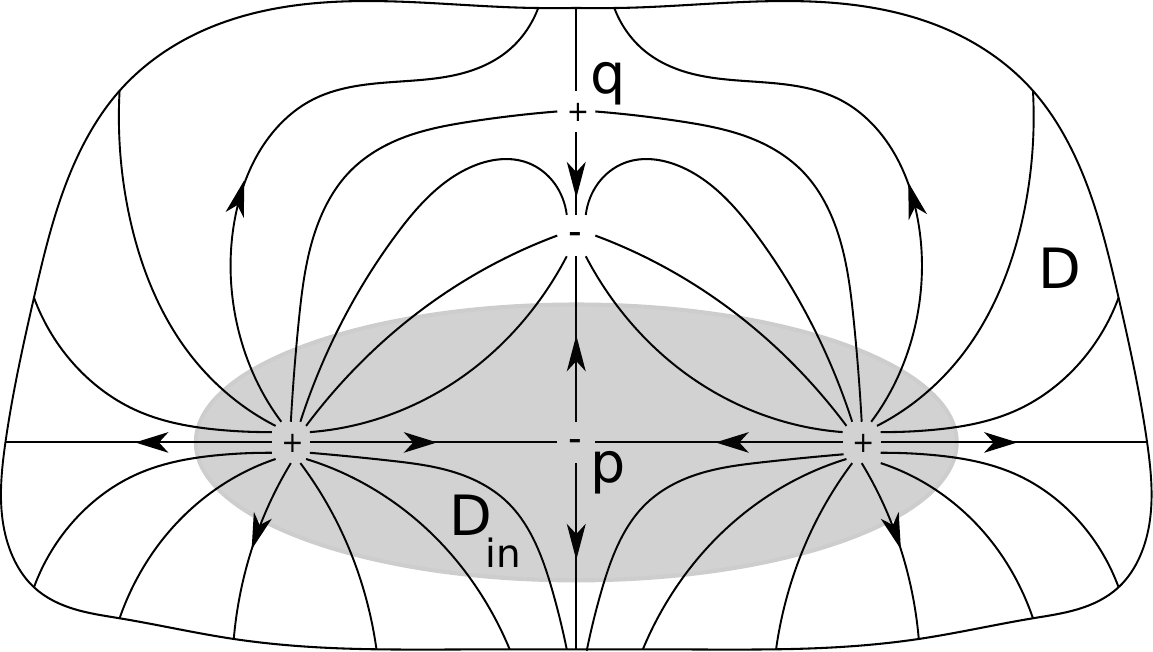}
\caption{Characteristic foliation on $D$ and $D_{in}$}
\label{b:isolated hyperbolic}
\end{figure}
\end{center}

To obtain a new plane field $\zeta_1$ we consider a product neighborhood $S^2\times(-\infty,\infty)$ around the convex sphere $S^2=S^2\times \{0\}$ oriented so that the product orientation is positive and the contact structure is translation invariant. 

Because the boundary of $D_{in}$ is transverse to the characteristic foliation, there is an isotopy $\rho_s, s\in (-1,1),$  of $D_{in}$ interpolating between the identity of $D_{in}$ and the involution $\rho_1$ mentioned above such that
\begin{itemize}
\item $\rho_s=\mathrm{id}$ for $s$ close to $-1$,
\item $\rho_s$ is constant and preserves the characteristic foliation for $s$ close to $1$,
\item $\rho_s$ preserves the characteristic foliation near $\partial D_{in}$ for all $s$,
\item the hyperbolic singularity $p$ is a fixed point of $\rho_s$ for all $s$,
\item the unstable leaves of $p$ are interchanged by $\rho_s$ for $s$ close to $1$, and
\item as $s$ varies from $-1$ to $1$, the unstable leaves of $p$ rotate by a counter clockwise half turn. 
\end{itemize} 
A new plane field $\zeta_1$ is determined (up to homotopy) by the requirement that $\zeta_1=\xi$ outside of $D_{in}^2\times(-1,1)$ and the characteristic foliation of $\zeta_1$  on $D_{in}^2\times\{s\}$ is the image of the characteristic foliation of $\xi$ under $\rho_s$. 

For $k\in\N$ the plane field $\zeta_k$ is defined by applying the above operation $k$ times. For negative $k$ one uses $|k|$-times the isotopy $\rho_s^{-1}$ instead of $\rho_s$. 

When one passes from $\xi$ to $\zeta_k$, the homotopy type of the plane field changes. In order to determine how, we choose a framing of the tangent bundle so that $TD$ is tangent to the span of the first two components of the framing and so that the framing is vertically invariant on the ball $D\times(-1,1)\subset D\times(-\infty,\infty)$ (recall that $\zeta_k=\xi$ outside of this ball). We consider the Gau{\ss} maps $g, g_k$ of $\xi,\zeta_k$. 

Because all singular points of the characteristic foliation on $D$ are non-degenerate, we may assume that $g(-TD^2(p,0))$ is a regular value  of $g$ and $g_k$. By construction, the Thom-Pontrjagin submanifolds of $g$ and $g_k$ coincide, but their framings are different. A direct computation in terms of local coordinates near the negative hyperbolic singularity coordinates shows that the Hopf invariants satisfy   $H(g_k)- H(g)=k$.  
\end{ex}

We continue to study the plane fields from the above example.  

\begin{lem} \label{l:homotopy ex}
For $k>0$, the plane field constructed above cannot be a contact structure. If $k\le 0$, then the Hopf invariant of the resulting plane field is $k$ and it can be chosen to be a contact structure. 
\end{lem}

\begin{proof}
The first part is immediate from \lemref{l:rgc}. We now sketch the argument for the second part.  We now sketch the second part.

By Lemma~2.4 of \cite{gi-bif}, it suffices to consider neighborhoods of the instances $s$ where one or two retrogradient connections occurs (for all other $s$ the characteristic foliation on $D\times\{s\}$ admits a dividing set). Let $s_0$ be such an instance. Since all singular points have non-vanishing divergence, the corresponding characteristic foliation on $D\times\{s_0\}$ is the characteristic foliation of a contact structure by Proposition~II.1.2 of \cite{gi-conv} and there is a germ of a contact structure $\xi$ on $D\times(-1,1)$ near $D\times\{s_0\}$ which induces the singular foliation $\mathcal{F}_{s_0}$ on $D_{s_0}$.  Now \lemref{l:rgc} stated above is slightly weaker than what is actually proved in \cite{gi-conv} (Giroux shows a more analytic condition formulated as equation (\textdaggerdbl) on p.~638 in \cite{gi-bif}). This stronger condition is also satisfied by our movie.

In order to arrange that the characteristic foliation of $\xi$ near $D\times\{s_0\}$ coincides with the singular foliation we have defined, one can choose a foliation tangent to $\xi$ and transverse  to $D\times\{s_0\}$ on a tubular neighbourhood  of $\partial D_{in}\times S^1$   and consider  flows  with compact support in that collar tangent to the chosen  foliation which preserve $D\times\{s_0\}$. By the contact condition, there is a flow such that the characteristic foliation of the pull back of $\xi$ with respect to this flow is diffeomorphic to the singular foliation we have constructed. 
\end{proof}

We now specialize to $k=-1$.  There is exactly one positive hyperbolic singularity $q\in D$ and both stable leaves of $q$ come from $D_{in}$. When $k=-1$, one can choose $\rho_s$ so that for each stable leaf of $q$ there is exactly one coincidence with an unstable leaf of $q$ of $\rho_s^{-1}$ for exactly one $s\in(-1,1)$. 

If we choose this parameter to be the same for both coincidences (say for $s=0$) the union of the pair of stable/unstable leaves on the sphere $S^2\times\{s=0\}$ is a minimal non-loose unknot $K$ by \ref{c:non-loose exist}.

Using this description of a minimal non-loose unknot we next construct an example of a family of contact structures on $S^3$ which is not homotopic to the constant family of plane fields. By Gray's theorem this produces a contact diffeomorphism $\psi$ with $\mathrm{d}(\psi)\neq 0$.  Of course, such a contact diffeomorphism can be produced using Eliashberg's classification theorem. Thus, the main point of the following lemma is that we do not only control $\mathrm{d}(\psi)$ but also  $\psi(K)$ where $K$ is the non-loose unknot from the above example. 

\begin{lem} \label{l:K preserved,dym=1}
Let $K\subset S^3$ be the standard non-loose unknot in $(S^3,\xi_{-1})$. Then there is $\psi\in\mathrm{Diff}_+(S^3,\xi_{-1})$ which 
\begin{itemize} 
\item preserves the orientation of $\xi_{-1}$ and $\mathrm{d}(\psi)=1$,
\item $\psi(K)=K$  and the orientation of $K$ is reversed.
\end{itemize} 
\end{lem}

\begin{proof} \label{ex:homotopy ex family}
We continue to consider the situation from \lemref{l:homotopy ex} for $k=-1$. Our goal   is to obtain a non-trivial loop of contact structures on $S^3$ which preserves $K$ as a set. 

For $\tau\in[0,2\pi]$ let $\psi_\tau : S^3\lra S^3$ be the composition of two rotations: The first is a rotation around a plane orthogonal to $H$  by the angle $\tau$ while the second rotation is a rotation by $-\tau/2$ around the plane containing $H$. The family of contact structures $\psi_{\tau}(\xi_{st})$ is constant because these rotations preserve the complex structure on $\mathbb{C}^2$. 

As above, deform $\xi_{st}$ to a new contact structure so that on the second factor of $N(H)=S^1\times D^2$ the characteristic foliation of $\xi$  is diffeomorphic to the one shown in \figref{b:isolated hyperbolic}. This can be done relative to a neighborhood of $\{0\}\times S^1$ and $\xi$ can be chosen so that it is invariant under $\psi_{2\pi}$. Then the loop of contact structures $\xi(\tau)=\left(\psi_{\tau}\right)_*(\xi),\tau\in[0,2\pi],$ is contractible in the space of plane fields because $\xi$ is homotopic to $\xi_{st}$ and $\xi_{st}$ is preserved by $\psi_\tau$. 

Let $R$ be the unit vector tangent to the fibers of the Hopf fibration. This vector field is positively transverse to $\xi_{st}$ everywhere.  We may assume that $R$ is a component of the framing and that $-R$  is a regular value of the Gau{\ss} map of $\xi$. Since the loop $\xi_{\tau},\tau\in[0,2\pi]$, is null homotopic, the associated Thom-Pontrjagin manifold of $S^3\times S^1$ (carrying the family of plane fields tangent to $S^3$ induced by $\xi_\tau$) is framed cobordant to the empty manifold. In particular, the sum of the Arf invariants of all components of the Thom-Pontrjagin submanifold is zero. 

We now modify the family $\xi_{\tau}$ so that the framing of one component of the Thom-Pontrjagin submanifold changes so the associated Arf invariant changes while we do not modify the framing of other components. 

For this we consider the family of thickened spheres $\psi_\tau(S^2\times (-\delta,\delta))\subset S^3$ and apply to each such family the modification discussed in \exref{ex:homotopy ex}.  We obtain a closed loop of contact structures $\xi_{-1}(\tau)$. The only part of the Thom-Pontrjagin submanifold affected by this construction is the component containing $H$. The framing of this submanifold changes so that the framing makes one full turn (compared to a $\mathrm{SU}(2)$-invariant framing) as one moves along $H$ and also as one moves along $\{p\}\times[0,2\pi]$ for $p\in H$. 

We identify $S^3\times\{0\}$ with $S^3\times\{2\pi\}$ and consider the situation in $S^3\times S^1$. The Thom-Pontrjagin submanifold associated to the constant family $\xi$ is the same as the Thom-Pontrjagin submanifold of $\psi(\tau)$ but the framing of the component containing $H$ has changed in such a way that the Arf invariant also changes. This implies that $\xi_{-1}(\tau)$ is a homotopically non-trivial loop of plane fields on $S^3$. 

Note that by construction, we also obtain a family $K_\tau$ of Legendrian knots (with respect to $\xi_{-1}(\tau)$) so that $K_{2\pi}=K_0$ but the orientation is reversed. Thus, $\psi_{2\pi}$ is a contact diffeomorphism of $\xi_{-1}$ which reverses the orientation of $K$, preserves the orientation of $\xi_{-1}$ and has $\mathrm{d}(\psi_{2\pi})\neq 0\in\pi_4(S^2)$. Thus, $\psi=\psi_{2\pi}$ has the desired properties. 
\end{proof}


\subsection{The classification up to isotopy for non-oriented unknots} \label{s:non or}
In this section we classify minimal  non-oriented non-loose unknots in $S^3$ up to isotopy. This classification makes heavy use of Eliashberg's classification theorem (\thmref{t:eliashclass}) as discussed on p.~\pageref{t:eliashclass}: The overtwisted disc we will use to apply that theorem will vary in a controlled fashion.


In the following we denote the minimal non-loose unknot described at the beginning of \secref{ss:fronts} by $L$ and $\overline{L}$ denotes the same knot with the reversed orientation. 
\begin{thm} \label{t:non or isotopy}
Let $\xi_{-1}$ be the overtwisted contact structure with Hopf invariant $-1$ on $S^3$. Then every minimal non-loose Legendrian unknot $K$ is isotopic to  $L$ or $\overline{L}$. 
\end{thm}

\begin{proof}
Let $K$ be a non-loose unknot. For convenience, we remove two open Darboux balls from the complement of $L\cup K$ such that the characteristic foliation on the boundary of each of these balls is convex. We thus view $L$ as non-loose knot in $S^2\times[-1,1]$. The contact isotopy moving $L$ to $K$ will have support inside this smaller space.  

According to the coarse classification of non-loose unknots there is a contact diffeomorphism $\psi$ with $\psi(L)=K$. In particular, $\psi$ preserves the orientation of $S^3$ and by Cerf's theorem there is an isotopy $\psi_\sigma,\sigma\in[0,1]$, of $S^3$ connecting $\psi$ to the identity. By \lemref{l:K preserved,dym=1} we may assume that $\mathrm{d}(\psi)=0$.  Because the space of $3$-balls in $\R^3$ is simply connected (p.5 in \cite{cerf}), we can choose $\psi_\sigma$ so that it preserves the two balls we have removed from $S^3$. 

In order to construct the contact isotopy moving $L$ to $K$ we will apply Gray's theorem to a family of contact structures on $N=S^2\times[-1,1]$. The parameter space will be $I^2=[0,1]\times[0,1]$. We  now describe 
\begin{itemize}
\item the contact structures $\xi(\sigma,\tau)$ on $N$, and 
\item restrictions on the product decomposition $N\simeq S^2\times[-1,1]$
\end{itemize}
for parameter values $(\sigma,\tau)$ in a neighborhood of $\partial I^2$ and near the boundary of $M$. 
\begin{center} \label{setup table}
\begin{tabular}{|l|l|}
\hline
$t=\pm 1$ & \begin{tabular}{l} $\xi(\sigma,\tau)=\xi_{-1}$ independent of $\sigma$ and $\tau$. \\ The product decomposition near $\partial N$ is \\ independent of $\sigma,\tau$.
\end{tabular}\\
\hline
$\tau=0$ &
\begin{tabular}{l} $\xi(\sigma,\tau)=\xi_{-1}$ independent of $\sigma$. \\
The product decomposition of $N$ is  the image \\of the original 
product decomposition\\ under $\psi_\sigma$. 
\end{tabular}\\
\hline
$\tau=1$ &
\begin{tabular}{l} $\xi(\sigma,\tau)=\left(\psi_{\sigma}\right)_*\xi_{-1}$. \\
The product decomposition of $N$ is  the image \\of the original 
product decomposition\\ under $\psi_\sigma$. 
\end{tabular}\\
\hline
$\sigma=0$ & 
\begin{tabular}{l} $\xi(\sigma,\tau)=\xi_{-1}$ independent of $\tau$. \\
$N\simeq S^2\times[-1,1]$ is such that $L\subset S^2\times\{0\}$ \\
is in normal form (c.f. the proof of \thmref{t:coarse unknot general}).
\end{tabular}\\
\hline
$\sigma=1$ &
\begin{tabular}{l} $\xi(\sigma,\tau)=\xi_{-1}$ independent of $\tau$. \\
$N\simeq \psi_1(S^2\times[-1,1])$ so that $K\subset \psi_1(S^2\times\{0\})$ \\
is in normal form. 
\end{tabular}\\
\hline
\end{tabular}
\end{center}\smallskip

\figref{b:box} illustrates the situation: Each point in the box represents a sphere, horizontal lines represent the manifold $S^2\times[-1,1]$ and some vertices are labeled with coordinates.  On the front face $\{\tau=1\}$, there is an obvious family (parametrized by $\sigma$) of Legendrian knots  interpolating between $K$ and $L$. This is indicated by the curved line on the front face of the box.  However, the contact structure is not constant. The straight lines on the top face, respectively bottom face,  correspond to the knots $L$, respectively $K$. On the back face $\{\tau=0\}$ the contact structure is constant but there is no obvious isotopy moving $K$ to $L$. 

\begin{figure}[htb]
\begin{center}
\includegraphics[scale=0.7]{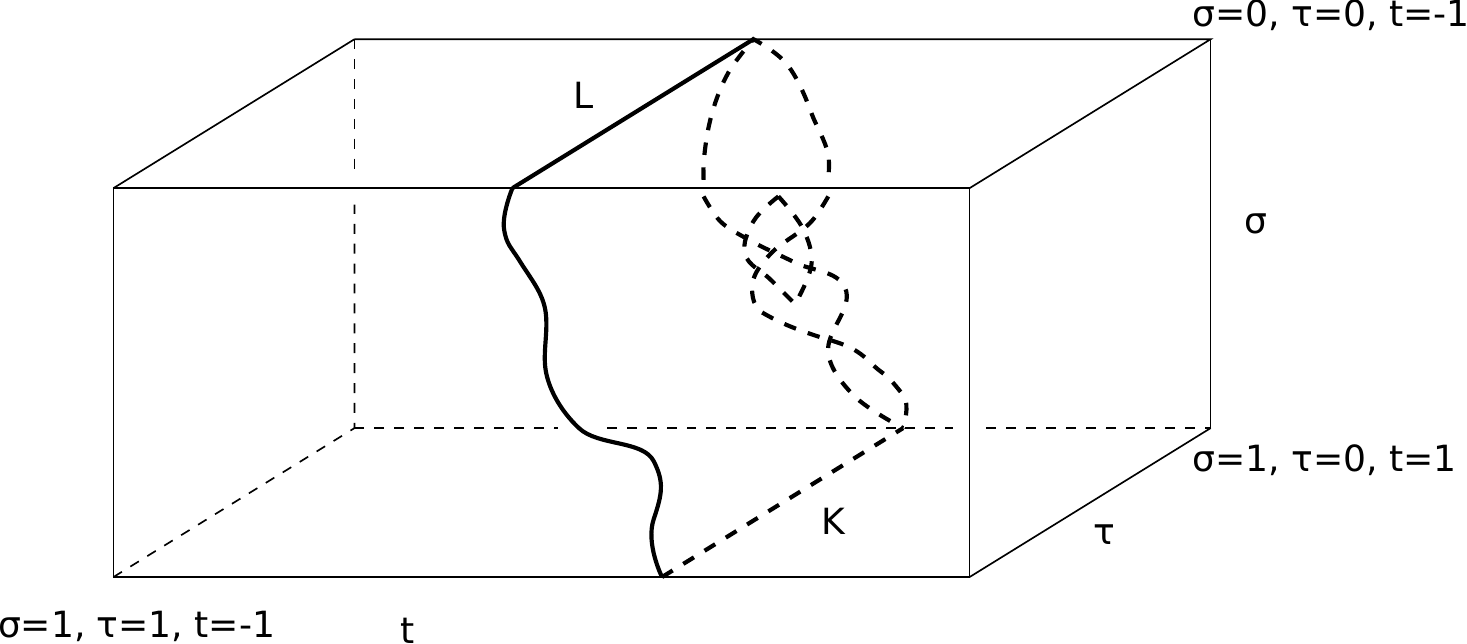}
\end{center}
\caption{The setup for the construction of an isotopy from $L$ to $K$}
\label{b:box}
\end{figure}

For $\tau=1$, the knot $\psi_\sigma(L)$ is contained in a sphere of the product decomposition. One should keep in mind that the product decomposition of $N$ depends on $\sigma$ and $\tau$.  The curved line on the front of the box is
$$
\left\{(T_\pm(\xi(\sigma,1)),\sigma,1)\,|\,\sigma\in I\right\}.
$$
where $S^2(t,\sigma,\tau)=S^2\times\{t\}\times\{(\sigma,\tau)\}$. The union of the left-most, respectively right-most, pieces of the dashed curves on the back side $\{\tau=0\}$ are 
$$\left\{(T_-(\xi(\sigma,1)),\sigma,1)\,|\,\sigma\in I\right\}, \textrm{ respectively } \left\{(T_+(\xi(\sigma,1)),\sigma,1)\,|\,\sigma\in I\right\}.
$$

For $\sigma,\tau\in\partial I^2$ we have constructed  contact structures $\xi(\sigma,\tau)$ on $N$  which are constant near $\partial N$. Our current goal is to extend this family to an $I^2$-family of contact structures. 

In order for this to be possible, we have to be able to extend $\xi(\sigma,\tau)$ as a family of plane fields. But this is obvious since we assumed $\mathrm{d}(\psi_1)=0$, i.e. there is an $I^2$-parametric family $\zeta(\sigma,\tau)$ of plane fields on $N$ which coincides with the contact structures on the boundary $\partial I^2$ (as described in the table above).

In order to ensure that \thmref{t:eliashclass} can be applied we will make a particular choice for $\zeta(\sigma,\tau)$, and we also deform the family $\xi(\sigma,\tau)$ on the boundary. Then we will  argue that $\zeta(\sigma,\tau)$ is homotopic to a family of contact structures $\xi(\sigma,\tau)$ for $(\sigma,\tau)$ in $I^2$ relative to the boundary of the parameter space and such that all plane fields are constant on a neighborhood of $\partial N$.


In the next two steps we modify the product decomposition of $N$ for $\tau=0$ in order to achieve that Theorem~\ref{t:eliashclass} can be applied effectively. 

Recall that $T_-(\xi)$ is defined by the requirement that $\xi$ is tight on $S^2\times[-1,T_-(\xi))$ and overtwisted on $S^2\times[-1,t)$ for all $t>T_-(\xi)$.  $T_+$ is defined similarly requiring tightness of $\xi$ on $S^2\times(T_+(\xi),1]$.  We will write $T_\pm(\sigma,\tau)$ instead of $T_\pm(\xi(\sigma,\tau))$.

{\bf Reduction to the case $T_-(\sigma,0)<T_+(\sigma,1)$ for all $\sigma\not\in\{0,1\}$:} 
Consider the functions $T_\pm(\xi(\sigma,0))$ defined in \eqref{e:Tpm}. By \lemref{l:Tpm cont} they  are continuous and $\xi(\sigma,\tau)$ is overtwisted on every neighborhood of $S^2\times [-1,T_-(\sigma,\tau)]$ and $S^2\times[T_+(\sigma,\tau),1]$. According to \lemref{l:T-<T+} $T_-(\xi(\sigma,0))\le T_+(\xi(\sigma,0))$.

By \lemref{l:birth} there are no closed leaves and by \lemref{l:only one} there must be at least two retrogradient connection and generically there are at most two retrogradient connections on every sphere $S^2(t,\sigma,0)$ if $T_-(\xi(\sigma,0))=T_+(\xi(\sigma,0))$. For a generic product decomposition of $N$ the equality $T_-(\sigma,0)= T_+(\sigma,0)$ happens finitely many instances 
$$
0=\sigma_0<\sigma_1\ldots<\sigma_n=1.
$$

According to \remref{r:more than one in only one} (c.f. p.~\pageref{r:more than one in only one}) the equality $T_-(\sigma_i,0)= T_+(\sigma_i,0)$ implies that  $S^2\times\{T_-(\sigma_i,0)\}$ contains a minimal non-loose unknot $K_{\sigma_i}$. After an isotopy of $N\times I \times\{\tau=1\}$ preserving the hypersurfaces $\{\sigma=\mathrm{const}\}$ to achieve that the non-loose unknots in $S^2(T_+(\xi(\sigma_i,0)),\sigma_i,0)$ and $S^2(T_+(\xi(\sigma_i,1)),\sigma_i,1)$ coincide we can consider layers $[\sigma_i,\sigma_{i+1}]\times I$ in the parameter together with deformed plane fields $\zeta'(\sigma,\tau)$  satisfying boundary conditions along  $\{\sigma=\sigma_i\}$ similar to the ones given in the table above.  The only difference is that along $\{\sigma=\sigma_i\}$ the contact structure is not constant: Let 
$$
\widehat{\psi}_i : (N\times\{(\sigma_i,1)\},\xi(\sigma_i,1)) \lra (N\times\{(\sigma_i,0)\},\xi(\sigma_i,0))
$$ 
which maps the non-loose unknot $\psi_{\sigma_i}(L)\subset N\times \{\sigma_i,1\}$ to the non-loose unknot on $S^2\times\left\{\right(T_-(\sigma_i,0),\sigma_i,0\left)\right\}$ such that $\mathrm{d}(\widehat{\psi}_i)=0$. As above, we then use Cerf's theorem to obtain an isotopy from $\mathrm{id}_N$ to $\widehat{\psi}_i$ and hence a family of contact structures.

Since $\mathrm{d}(\widehat{\psi}_i)=0$ for all $i$, the resulting family of plane fields on 
$$
\left(\partial N\times I^2\right)\cup\left(\bigcup_i N\times\{\sigma_i\}\times I\right)\cup\left(N\times \partial I^2\right)
$$
extends to $N\times I^2$. This reduces the construction of an $I^2$-parametric family of contact structures to the case when $T_-(\sigma,0)<T_+(\sigma,0)$ for all $\sigma\not\in\{0,1\}$. In particular, no sphere $S(t,\sigma,\tau)$ with $\sigma\not\in\{0,1\}$ contains a non-loose unknot. From now on we assume that this was the case from the beginning. 

{\bf Elimination of locally non-loose unknots:} We want to reduce to the case when for all $\sigma\in(0,1)$ there are no spheres $S^2(T_\pm(\sigma,0),\sigma,0)$ containing a Legendrian unknot $K$ such that  \label{elim local non-loose}
\begin{itemize}
\item $\xi(\sigma,0)$ is overtwisted on every tubular neighborhood of that sphere and
\item there is a neighborhood which becomes tight when $K$ is removed,
\item the complement of $S^2(T_\pm(\sigma,0),\sigma,0)$ remains overtwisted when $K$ is removed.
\end{itemize} 
By \remref{r:more than one in only one} the characteristic foliation of $\xi(\sigma,0)$ on $S^2(T_\pm(\sigma,0),\sigma,0)$ has precisely two retrogradient connections (each of these retrogradient connections is indicated by a line in Figure~\ref{b:local-nonloose} below) but now there is an overtwisted disc of $\xi(\sigma,0)$ in the complement of $S^2(T_\pm(\sigma,0),\sigma,0)$. 

By \thmref{t:eliashclass} one can eliminate  the intersection point in the left-hand part of \figref{b:local-nonloose} as indicated in the right-hand side of that figure. 

\begin{figure}[htb]
\begin{center}
\includegraphics[scale=0.7]{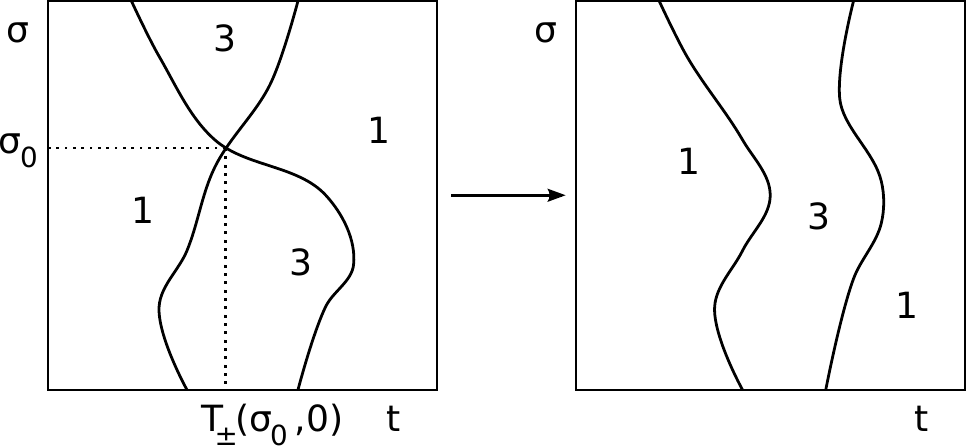}
\end{center}
\caption{Removing a locally (but not globally) non-loose unknot } 
\label{b:local-nonloose}
\end{figure}



{\bf Extension of $\xi(\sigma,\tau)$ to a neighborhood of the boundary of the parameter space
:}  Recall that $\zeta(\sigma,\tau)$ is a contact structure whenever $(\sigma,\tau)\in\partial I^2$. We homotope this family of plane fields relative to $\partial I^2$ to achieve that there is a collar of $\partial I^2$ such that 
there are no simultaneous retrogradient connection on the interior of a collar of $\{\tau=1\}\cup\{\sigma=0 \textrm{ or } \sigma=1\}$.  


For the construction of  this perturbation one uses a slight generalization of \lemref{l:rgc}: We perturb the spheres $S^2\times\{0\}$  so that one  retrogradient saddle-saddle connections occur for $t=0$ while the other such connections occurs for $t>0$ unless $\sigma=0$ or $\sigma=1$. This is indicated in Figure~\ref{b:step0}. The bold line on the right side represents simultaneous retrogradient connection while in the left part, for slightly smaller $\tau$, the two connections occur on different spheres of the product decomposition of $N\times(\sigma,\tau)$.

\begin{figure}[htb]
\begin{center}
\includegraphics[scale=0.7]{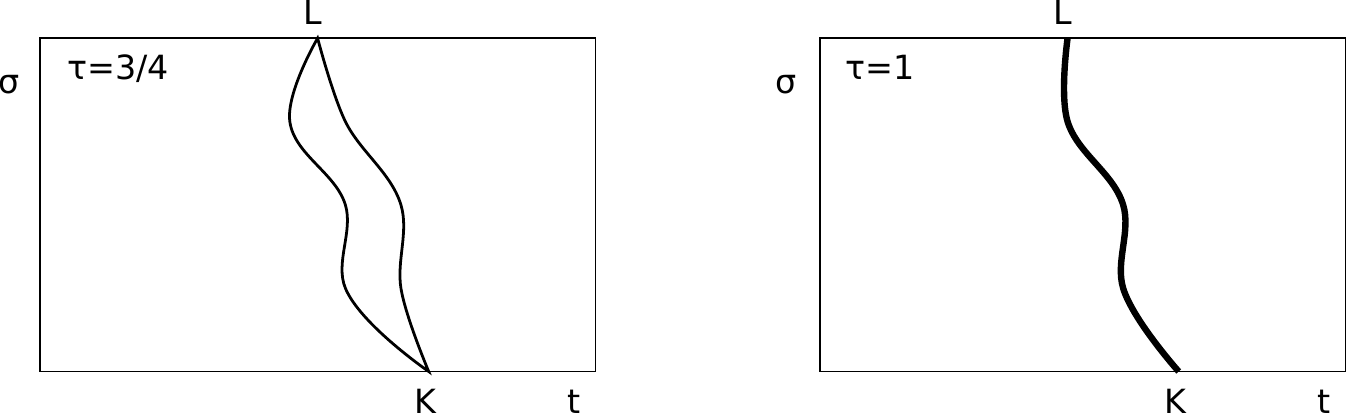}
\end{center}
\caption{The simplification of $\xi(\sigma,\tau)$ near $\{\tau=1\}$}
\label{b:step0}
\end{figure}

From now on we assume that  $\zeta(\sigma,\tau)$ is a contact structure for $3/4\le\tau\le 1$. 

{\bf Construction of plane fields $\zeta''(\sigma,\tau),\sigma,\tau\in I$:} We are now ready to homotope the family of plane fields  $\zeta(\sigma,\tau)$  is contact and overtwisted on parts of $N$ for enough parameter values $(\sigma,\tau)$ so that we will be able to apply \thmref{t:eliashclass}. The notation $T_\pm$ will be used in the same sense as before, the argument of this function will always be an overtwisted  contact structure.

The homotopy consists of two steps. In the first step, we will produce a family of plane fields $\zeta'(\sigma,\tau)$ with the following properties

\begin{itemize}
\item[(1a)] For $\tau\in[0,1/3]$ the plane fields $\zeta'(\sigma,\tau)$ and $\xi(\sigma,0)$ coincide on a  neighborhood of $S^2\times[-1,T_-(\xi(\sigma,0))]$.
\item[(1b)] For $\tau\in[3/4/1]$ there is no change, i.e. $\zeta'(\sigma,\tau)=\zeta(\sigma,\tau)$ is a contact structure.  
\item[(1c)] For $\tau\in[1/2,3/4]$: Let $\varphi_s(\sigma) : N\lra N, s\in[0,1],$ be family of diffeomorphisms of $N$ fixing the boundary such that $\varphi_0=\mathrm{id}_N$ and
$$
\varphi_1(\sigma)\left(S^2\times[T_+(\zeta(\sigma,3/4),1]\right)\textrm{ and }  S^2\times[-1,T_-(\zeta(\sigma,0),1]
$$
are disjoint. Using the isotopy $\varphi_s$ we define $\zeta'(\sigma,\tau)$ for $\tau\in[1/2,3/4]$ as follows: For $\tau\in[1/2,3/4]$ let
$$
\zeta'(\sigma,\tau)=\left(\varphi_s\right)_*(\zeta(\sigma,3/4)) \textrm{ with } s=3-4\tau.
$$
\end{itemize}

The second step is a deformation of the family $\zeta'(\sigma,\tau)$ such that the resulting plane field satisfies $\zeta''(\sigma,\tau)=\zeta'(\sigma,\tau)$ on the regions specified in the previous step, and the following additional requirement.

\begin{itemize}
\item[(2)] For $\tau\in[1/4,1/2]$ we require that $\zeta''(\sigma,\tau)=\zeta'(\sigma,1/2)$ on a neighborhood of $S^2\times[T_+(\zeta'(\sigma,\tau),1]$. 
\end{itemize}

Both homotopies of plane fields exist by standard arguments from homotopy theory (notice that the union of the regions where we impose restrictions on $\zeta'$ and $\zeta'$ retracts onto $\partial(N\times I^2)$. 

The diagram in \figref{b:cond on zeta} indicates where $\zeta''(\sigma_0,\tau)$ is a contact structure for a fixed $\sigma_0$. 

\begin{center}
\begin{figure}[htb] 
\includegraphics[scale=0.9]{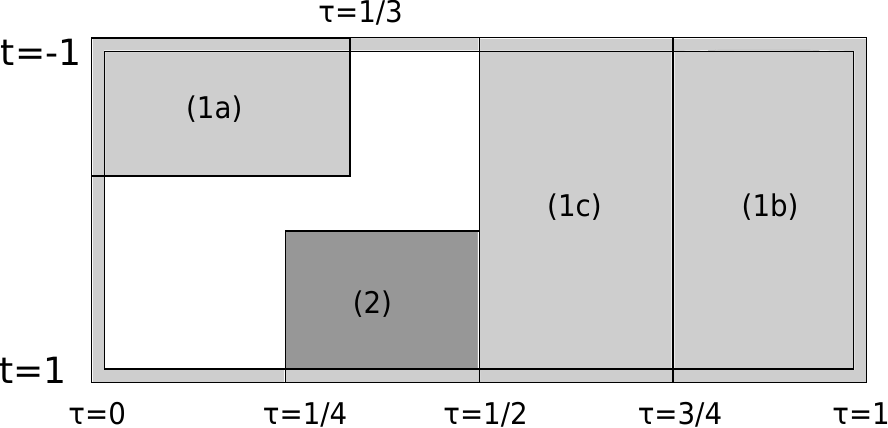}
\caption{Domains in $\{\sigma=\sigma_0\}$ where $\zeta''$ is contact by assumption are shaded. Every line parallel to the $t$-axis represents $N\times\{(\sigma_0,\tau)$.}
\label{b:cond on zeta}
\end{figure}
\end{center}

{\bf Deformation of $\zeta''(\sigma,\tau)$ to a family of contact structures $\xi(\sigma,\tau)$:} By definition  $\zeta''(\sigma.\tau)$ is a contact structure for $\tau\ge 1/2$. When $1/4\le\tau\le1/2$, the plane field $\zeta(\tau,\sigma)$ is an overtwisted contact structure on a collar of $S^2\times\{1\}$ and there is a continuous family of overtwisted discs $\Delta(\sigma,\tau)$. 

Finally, $\zeta''(\sigma,\tau)$ is an overtwisted contact structure on $S^2\times[-1,T_-(\xi(\sigma,0))]$ when $0\le\tau \le 1/3$. Moreover, $\zeta''(\sigma,\tau)$ is independent of $\tau$. In order to apply \thmref{t:eliashclass} it is therefore  enough to construct a collection 
$$
\Delta_1(\sigma_1) \textrm { for }\sigma_1\in(s_0,s_1),\ldots,\Delta_n(\sigma_n) \textrm{ for } \sigma_n\in(s_{n-1},s_n)
$$
of smooth families of overtwisted discs of $\zeta''(\sigma,0)=\xi(\sigma,0)$ such that
\begin{itemize}
\item they are pairwise disjoint in $N\times I\times\{0\}$, 
\item $\Delta_i(\sigma_i)$ is contained in a neighborhood of $S^2\times[-1,T_-(\sigma_i,0)]\times\{\sigma_i,0\}$ where $\zeta''(\sigma_i,0)$ is contact, and
\item $(s_0,s_1)\cup\ldots\cup(s_{n-1},s_n)=(0,1)$.  
\end{itemize}
For this we use $T_-(\sigma,0)<T_+(\sigma,0)$ and the fact that $S^2(T_-(\sigma,0),\sigma,0)$ does not contain a locally non-loose minimal unknot for $\sigma\in(-1,1)$ (c.f. the discussion on p.~\pageref{elim local non-loose}).

By definition, the characteristic foliation of $\zeta''(\sigma,0)=\xi''(\sigma,0)$ on $S^2(T_-(\sigma,0),\sigma,0)$ generically has 
\begin{itemize}
\item a degenerate closed leaf, or 
\item a single retrogradient connection, or
\item a pair of simultaneous retrogradient connections. 
\end{itemize}
In the first two cases we immediately obtain a pair of disjoint overtwisted discs on spheres $S^2(t,\sigma,0)$ for $t=T_-(\sigma,0)+\varepsilon$ for $\varepsilon>0$ small enough. In the third case, because we have eliminated locally non-loose knots on $S^2(t,\sigma,0)$, the same is true. These overtwisted discs correspond to stable configurations of leaves of the characteristic foliation, i.e. they are part of smooth families of overtwisted discs of $\zeta''(s,0)$ lying on $S^2(T_-(s)+\varepsilon,s,0)$ for $s$ close enough to $\sigma$.

Note that $\zeta''(\sigma,0)$ is already a contact structure for $\sigma$ close enough to $0$ or $1$. Finally, since there is always a pair of disjoint overtwisted discs, we have found the desired collection of overtwisted discs.

According to \thmref{t:eliashclass} and the discussion following this statement, $\zeta''$ is homotopic relative to the boundary to a  $I^2$-family of contact structures  extending the $\partial I^2$-family we started with (described in the table on p.~\pageref{setup table}). 

{\bf Conclusion of the proof:}
The family of contact structures $\xi(\sigma,\tau),\sigma,\tau\in I,$ on $N$ is constant near $\partial N$. According to Gray's theorem, there is a family of isotopies $\varphi_{\sigma,\tau}$ such that  $\left(\varphi_{\sigma,\tau}\right)_*(\xi(\sigma,0))=\xi(\sigma,\tau)$. The family of knots  
\begin{equation} \label{e:family of Leg knots} 
\varphi_{\sigma,1}^{-1}(\psi_{\sigma}(L))
\end{equation}
is a family of Legendrian knots (parametrized by $\sigma$) in $(S^2\times[-1,1],\xi=\xi(\sigma,0))$. Because the contact structure is independent of $\tau$ when $\sigma=0$ or $\sigma=1$, it follows that $\varphi_{0,\tau}=\mathrm{id}_M=\varphi_{1,\tau}$. Therefore, the family of Legendrian knots in \eqref{e:family of Leg knots} interpolates between $L$ and $K$. 
\end{proof}

\begin{rem} 
In the previous proof, near $\tau=1$, we perturbed a family of non-loose  Legendrian unknots to obtain families of overtwisted discs. The choice involved here is a retrogradient connection which occurs first in the movie associated to a perturbation of the original family of spheres and the orientation of the knot.

Comparing the effect of the perturbation in the proof above with how a stabilization changes the characteristic foliation on a Seifert surface of the knot  one sees, that the overtwisted discs could also have been obtained by stabilizing the non-loose unknot. Here one has to choose an orientation and the sign of the stabilization. The Legendrian unknots obtained in this way are isotopic to $K_{1,0},K_{0,1}$ (cf. \secref{ss:fronts}).
\end{rem}

\subsection{The classification up to isotopy for oriented unknots} \label{s:or}

All rigidity results in this article build on the following theorem. 
\begin{thm} \label{t:different or, not isotopic}
Let $K$ be a minimal non-loose unknot in $S^3$. Then $K$ and $\overline{K}$ are not isotopic as oriented Legendrian knots. 
\end{thm}

In order to prove \thmref{t:different or, not isotopic} we will use two results. The first is concerned with the space of foliations by spheres of $S^2\times[0,1]$ and is based on Hatcher's theorem \cite{s1s2} on the space of diffeomorphisms of $S^2\times S^1$. The second result is more technical and uses standard theorems on transversality to establish that certain degenerate configurations of characteristic foliations on leaves foliations by spheres occur on topologically tame subsets of the product of the leaf space with the parameter space.

\subsubsection{Foliations on $S^2\times S^1$ and spheres containing a non-loose unknot} \label{ss:hatcher}

We recall Hatcher's theorem on the homotopy type of $\mathrm{Diff}(S^2\times S^1)$ and conclude that the space of foliations on $S^2\times[0,1]$ which coincide with the product foliation near the boundary is weakly contractible. Recall also, that if one leaf of a foliation by orientable surfaces on a $3$-manifold is a sphere, then all leaves are spheres (this famous result is due to G.~Reeb).

\begin{thm}[Hatcher \cite{s1s2}]  \label{t:hatcher}
The map 
\begin{align*}
\mathrm{O}(2)\times\mathrm{O}(3) \times \Omega \mathrm{SO}(3) &\lra \mathrm{Diff}(S^2\times S^1) \\
(A,B,\gamma_t) & \lmt \big( (p,\tau)\lmt (B\circ\gamma_\tau(p),A\tau)\big)
\end{align*}
is a weak homotopy equivalence.
\end{thm} 
Let $\mathrm{Diff}_\partial(S^2\times [0,1])$ be the group of diffeomorphisms of $S^2\times[0,1]$ which coincide with the identity near the boundary and $\mathrm{Diff}_\partial(S^2\times[0,1],\FF)$ the subgroup of those diffeomorphisms which preserve the product foliation. \thmref{t:hatcher} implies that 
$$
\mathrm{Diff}_\partial (S^2\times[0,1],\FF) \lra \mathrm{Diff}_\partial(S^2\times[0,1])
$$
is a weak homotopy equivalence (see item (8) in the appendix of \cite{smale conj} together with \cite{s1s2}). By $\mathcal{FOL}_\partial(S^2\times[0,1])$ we denote the space of foliations on $S^2\times[0,1]$ which coincide with the product foliation near the boundary (we view foliations as plane fields to define the $C^k$-topology on $\mathcal{FOL}_\partial(S^2\times[0,1])$). The map 
\begin{align*}
\mathrm{Diff}_\partial(S^2\times[0,1]) & \lra \mathcal{FOL}_\partial(S^2\times[0,1]) \\
f & \lmt \left(f(S^2\times\{t\})\right)_{t\in[0,1]}
\end{align*}
is a Serre fibration whose fibers are homeomorphic to $\mathrm{Diff}_\partial(S^2\times[0,1],\FF)$. This group can be thought of as $\Omega\mathrm{SO}(3)$ (with the multiplication coming from the group structure on $\mathrm{SO}(3)$). From \thmref{t:hatcher} together with Smale's theorem stating that $\mathrm{O(3)}\lra\mathrm{Diff}(S^2)$ is a weak homotopy equivalence one obtains the following corollary.

\begin{cor} \label{c:fol contractible} 
The space $\mathcal{FOL}_\partial (S^2\times [0,1])$ is weakly contractible. 
\end{cor}

\subsubsection{Degeneracies of characteristic foliations on spheres}
The second ingredient is a result on retrogradient connections present in the characteristic foliation on leaves of foliations or families of such foliations by spheres on $(S^2\times[-1,1],\xi)$.  We assume that all singularities of characteristic foliations are isolated. When we consider families of foliations we still can parametrize each leaf space by $[-1,1]$. 

Let $\FF$ be a foliation from $\mathcal{FOL}_\partial(S^2\times[0,1])$ and assume that the leaf $S_{t_0}(\xi)$ contains a retrogradient connection.  We parametrize a foliated neighborhood of this leaf by $S^2\times(t_0-\varepsilon,t_0+\varepsilon)$ so that the foliation by the first factor coincides with the given foliation and fix a fiber $T$ of  $S^2\times(t_0-\varepsilon,t_0+\varepsilon)\lra S^2=S_{t_0}$.  

Next, we will construct $1$-parameter families of perturbations $\mathcal{F}_x,x\in(-\delta,\delta)$,  with compact support in $S^2\times(t_0-\varepsilon,t_0+\varepsilon)\setminus T$ and $\delta>0$. 

For $p\in\gamma$ choose a compactly supported Legendrian vector field $X$ vanishing near $T$ and outside of $S^2\times[t_0-\varepsilon,t_0+\varepsilon]$ such that $X(p)\neq 0$ and $X$ is positively transverse to the leaves of $\FF$ whenever it does not vanish.

Then the flow $\varphi_x$ of $X$ is well-defined. Let $\FF_x:=\varphi_x(\FF)$. By \lemref{l:rgc} there is still a retrogradient connection close to $\gamma$ on nearby spheres $S_{\Gamma(x)}$ intersecting $T$ in $\Gamma(x)$ if $|x|<\delta$ is small enough. Let
\begin{align*}
\Gamma: (-\delta,\delta) & \lra (t_0-\varepsilon,t_0+\varepsilon)=T \\
x & \longmapsto \Gamma(x).
\end{align*}
The contact property of $\xi$ ensures that  $t_0$ is a regular value of $\Gamma$ (c.f. the proof of \lemref{l:rgc} in \cite{gi-bif}).

We have described a perturbation for a fixed foliation $\FF$ but the construction can be carried out in the same way for finite dimensional families of characteristic foliations. Similar perturbations exist for degenerate singularities and connections between non-degenerate hyperbolic singularities of the same sign. 

This type of perturbation can be applied simultaneously to different retrogradient connections and other degeneracies. Therefore, standard transversality theory (e.g. Theorem 2.7 in Chapter 3.2 of \cite{hirsch}) implies the following proposition:

\begin{prop} \label{p:manifold}
Let $(N=S^2\times [-1,1],\xi)$ be a contact manifold with convex boundary and $\FF_{\sigma,\tau}, \sigma,\tau\in(0,1)$, a $2$-parameter family of foliations on $N$ which is constant near $\partial N$. Assume that the following conditions are satisfied for $(t,\sigma,\tau)$ outside of a compact set in $P=(-1,1)\times(0,1)\times(0,1)$:
\begin{itemize}
\item[(i)] The subset of $P$ consisting of points $(t,\sigma,\tau)$ where a non-degenerate retrogradient connection occurs on $S^2(t,\sigma,\tau)$ is a union of finitely many surfaces. Each retrogradient connection corresponds to one such surface. The analogous statement holds for  degenerate singular points and connections between hyperbolic singular points  which are not retrogradient. A degenerate singular point $p\in S^2(t,\sigma,\tau)$ is supposed to be of birth-death type. In particular, the index of $p$  is zero and there is exactly one leaf $\gamma_p$ of the characteristic foliation which is an unstable, respectively stable, leaf of $p$ if $p$ is negative, respectively positive. 
\item[(ii)] The set of points corresponding to connections of a degenerate singularity $p$ to a hyperbolic singularity such that the connecting leaf is  $\gamma_p$ is a codimension $2$ submanifold. 
\item[(iii)] Any two of the submanifolds from (i),(ii) intersect transversely and each of the submanifolds in (i) is transverse to the intersection of pairs of other submanifolds from (i). In particular, the points $(t,p,\sigma)$ where $S^2(t,\sigma,\tau)$ has three retrogradient connections are isolated.
\end{itemize}
After a $C^r$-small perturbation of $\FF_{\sigma,\tau}$ relative to the complement of a small open neighborhood of the compact set these conditions are satisfied on $P$. 
\end{prop}

The fact that there are only finitely many singular points implies only that there are finitely many leaves of the characteristic foliations taking part in retrogradient connections. It does not mean that the space of instances, where a retrogradient connection occurs, is compact. For example, assume that $\xi(S_{t_0})$ has a closed leaf such that the holonomy is attractive on one side while it is repelling on the other. When a stable leaf of a positive hyperbolic singularity and an unstable leaf of a negative hyperbolic singularity accumulate on the degenerate closed leaf, then the set of leaves of a foliation by spheres containing $S_{t_0}$ as a leaf, which contain a retrogradient connection, is non-compact. 

\subsubsection{Proof that $K$ and $\overline{K}$ are not isotopic.} \label{ss:mainproof}
The proof of \thmref{t:different or, not isotopic} is based on the following idea: We argue by contradiction. From a Legendrian isotopy from $K$ to $\overline{K}$ we construct a $2$-parameter family (parametrized by $(\sigma,\tau)\in[0,1]^2$) of foliations $\FF_{\sigma,\tau}$ by spheres on $S^2\times[-1,1]$ (each leaf space is parametrized by $t\in[-1,1]$) such that 
$$
\mathcal{L}=\left\{ (t,\sigma,\tau)\in[-1,1]\times[0,1]^2 \,\left|\, \begin{array}{l}

 \textrm{the leaf }S^2\times\{t\} \textrm{ of }\FF_{\sigma,\tau} \textrm{ contains } \\ \textrm{ a non-loose unknot}
 \end{array}\right. \right\}
$$
has the following properties:
\begin{itemize}
\item $\LL$ is a piecewise smooth submanifold  of $[-1,1]\times[0,1]^2$ of codimension $2$.  
\item $\LL$ is properly embedded in $[-1,1]\times[0,1]^2$.
\item $\partial\left([-1,1]\times[0,1]^2\right)$ contains exactly one boundary point of $\LL$. 
\end{itemize}
Since compact $1$-manifolds with boundary have an even number of boundary points, this is a contradiction. A slightly weaker claim on $\mathcal{L}$ which would suffice is discussed briefly in the remarks following the proof of \ref{t:different or, not isotopic}.


In the following example we consider one configuration where  the sphere $S^2(t_0,\sigma_0,\tau_0)$  contains three simultaneous retrogradient connection. Our goal is to understand how this can lead to $\mathcal{L}$ being a non-smooth submanifold of $[-1,1]\times[0,1]^2$.  Recall \corref{c:non-loose exist} which will be used frequently  to establish the existence of minimal non-loose unknots. 

\begin{ex} \label{ex:one triple}
The top diagram of \figref{b:triple1} represents three  non-degenerate retrogradient connections on $S^2(t_0,\sigma_0,\tau_0)$ (as explained on p.~\pageref{b:non-loose}).  We assume that $p_0=(t_0,\sigma_0,\tau_0)\in\mathcal{L}$ lies in the interior of $P$. The graph $\Gamma_{t,-}$ is a closed (as subset) tree for $t\in(t_0-\varepsilon,t_0)$ and $t\in(t_0,t_0+\varepsilon)$ for $\varepsilon>0$ small enough.  
We consider which points of 
$$
\mathcal{L}_{\textrm{gen}}=\{(t,\sigma,\tau)\in \mathcal{L}\,|\,\xi(S^2(t,\sigma,\tau)) \textrm{ has exactly two retrogradient connections}\} 
$$ 
can lie in a neighborhood of  $p_0$. 

Points $(t_{double},\sigma,\tau)$ of $\mathcal{L}_{\textrm{gen}}$ which are close enough to $(t_0,\sigma_0,\tau_0)$ have two simultaneous retrogradient connections, and the remaining third bifurcation also occurs in the movie on $N\times\{(\sigma,\tau)\}$, but at a time $t_{single}$ which is different from $t_{double}$.

\begin{center}
\begin{figure}
\includegraphics[scale=0.8]{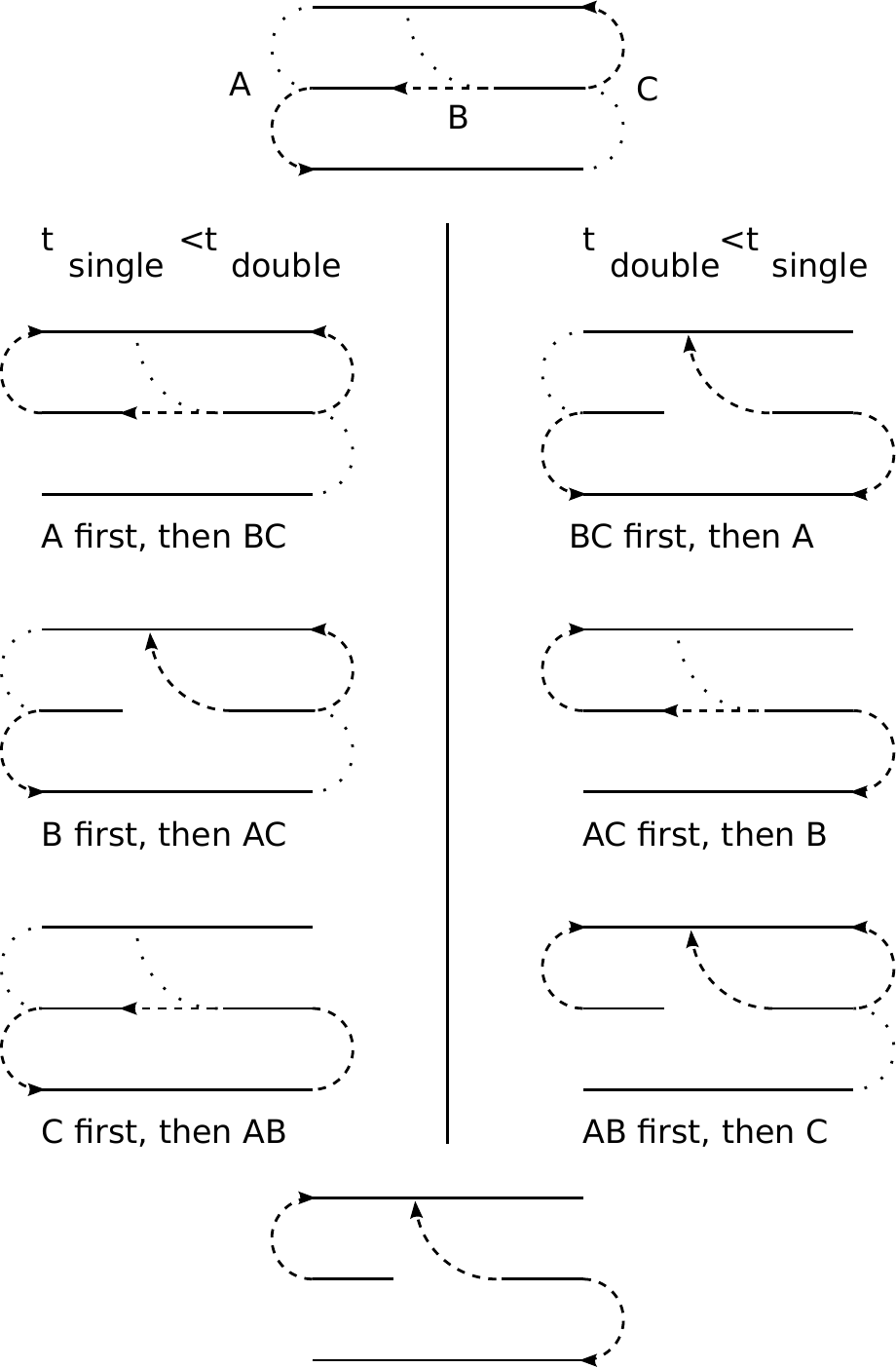}
\caption{Three simultaneous retrogradient connections and possible resolutions with a pair of simultaneous retrogradient connections.}
\label{b:triple1}
\end{figure}
\end{center}

For $(t_{double},\sigma,\tau)\in\mathcal{L}_{\textrm{gen}}$ it is necessary (and also sufficient  by  \corref{c:non-loose exist}), that no overtwisted disc appears on a sphere $S^2_t$ in the movie $N\times\{(\sigma,\tau)\}$ for $t\neq t_{double}$. In \figref{b:triple1} we consider all possibilities: The column on the left schematically shows $\Gamma_{t,-}$ for $t_{single}<t<t_{double}$, when the single bifurcation occurs before the two simultaneous retrogradient connections. $\Gamma_{t,-}$ is the union of solid and dashed arcs while dotted arcs indicate the bifurcation which happens later at $t_{double}$. Note that in all three cases, $\Gamma_{t,-}$ is not a tree, i.e. there is  an obvious overtwisted disc. Hence, none of these resolutions of the threefold retrogradient connection into a single retrogradient connection occurring {\em before} a double retrogradient connection corresponds to points of $\mathcal{L}_{\textrm{gen}}$.

The column on the right shows $\Gamma_{t,-}$ when $t_{double}<t<t_{single}$. Here, there is only the resolution  (marked with {\em AB first, then C}) of the threefold retrogradient connection into a single retrogradient connection occurring {\em after} a double retrogradient connection  leads to $\Gamma_{t,-}$ not being a tree for $t\in(t_{double},t_{single})$. The other two resolutions correspond to points in $\mathcal{L}_{\textrm{gen}}$.

Hence, $p_0$ is a non-smooth point of $\mathcal{L}$, but it has a neighborhood where $\mathcal{L}$ is a topological submanifold of $(-1,1)\times(0,1)^2$. In particular, it is not a boundary point of $\mathcal{L}$.  

\figref{b:ABC2} represents a neighborhood of $p_0$ in $P$. Points on planes marked with $A,B,C$  correspond to spheres where the retrogradient connection $A,B,C$ occurs. The  intersection point of all three planes is $p_0$. The $t$-axis is supposed to be horizontal, and by \lemref{l:rgc} the planes  have to be transverse to this direction. Intersections between two planes (dashed and dotted lines) correspond to two retrogradient connections occurring on the same sphere. The thickened dashed segments are contained in $\mathcal{L}$. 

\begin{center}
\begin{figure}
\includegraphics[scale=0.9]{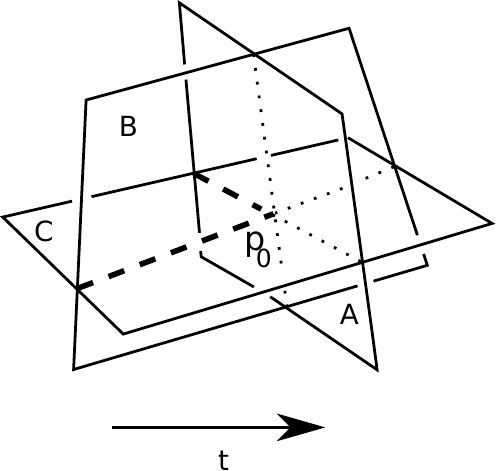}
\caption{A non-smooth point in $\mathcal L$}
\label{b:ABC2}
\end{figure}
\end{center}
\end{ex}

Configurations similar to the one considered in this example will appear later.

The following  preparatory lemma yields a normalization of the characteristic foliation on a sphere containing a non-loose unknot. 
\begin{lem} \label{l:fam almost normal}
Let $K_\sigma\subset S^3,\sigma\in[0,1],$ be a family of non-loose Legendrian unknots with $\mathrm{tb}(K_\sigma)=1$ and $K_1=K_0$. Then there are two balls $B_0,B_1\subset S^3$ and a family of foliations $\FF_\sigma$ of $S^3\setminus(B_0\cup B_1)$ by spheres such that 
\begin{itemize}
\item $\FF_0=\FF_1$, 
\item $K_\sigma$ is contained in a leaf $S_\sigma$ of $\FF_\sigma$, and
\item the characteristic foliation on $S_\sigma$ has exactly two singularities along $K_\sigma$. 
\end{itemize} 
\end{lem}
\begin{proof}
The family of knots $K_\sigma,\sigma\in[0,1]$, misses two small balls $B_0,B_1$ which we assume to be Darboux balls with convex boundary. We choose a sphere $\FF_0$ with the desired properties such that $S_0$ is a leaf carrying $K_0$ such that the characteristic foliation is in normal form.

By Gray's theorem we can choose a contact isotopy $\psi_\sigma$ of $S^3$ with support on the complement of $B_0\cup B_1$ so that $\psi_\sigma(K_0)=K_\sigma$, and we consider $\FF'_\sigma=\psi_\sigma(\FF_0)$. This is a family of foliations with all desired properties except that $\FF_0=\FF_0'\neq \FF'_1$ in general. 

Let $S_1'= \psi_1(S_0)$ and $\alpha,\alpha'$  functions on the universal cover  $\mathrm{pr}: \widetilde{K}_0\lra K_0$ of $K_0=K_1$ such that $\alpha(p)$, respectively $\alpha'(p)$, is the angle between $\xi(\mathrm{pr}(p))$ and $S_0$ respectively $S_1'$. Both functions attain each of the values $\ldots,-\pi,0,\pi,2\pi,\ldots$ exactly once. They are homotopic through functions with this property to strictly monotone functions $\alpha_m,\alpha_m'$.  There is a homotopy between these two functions through strictly monotone functions. Twisting $S_1'$ around $K_0$ accordingly, we obtain an isotopy moving $S_1'$ to $S_1''$ keeping $K_1$ fixed such that $S_1''$ and $S_0$ coincide on a tubular neighborhood of $K_1$ and such that throughout the isotopy the characteristic foliation on the sphere has exactly two singular points along $K_0=K_1$.

As in the proof of the Roussarie-Thurston normal form \cite{rou} for surfaces in $3$-manifolds carrying a Reebless foliation one can isotope $S_1''$ so that the coincides with $S_0$ after the isotopy. From \corref{c:fol contractible} we obtain the desired family of foliations $\FF_\sigma$. 
\end{proof}

Finally, we are in a position to prove \thmref{t:different or, not isotopic}.

\begin{proof}[Proof of \thmref{t:different or, not isotopic}]
Assume that $K_\sigma,\sigma\in[0,1]=I$, is a family of oriented Legendrian knots in $(S^3,\xi)$ such that $K_0=K$ and $K_1=\overline{K}$. This isotopy avoids two points of $S^3$, and we can therefore consider $S^3$ with two small open balls removed. We denote this space by $N\simeq S^2\times[-1,1]$. 

We will consider $N\times I^2$, and as in the proof of \thmref{t:non or isotopy}  the product decomposition of $N$ will vary.  
\begin{itemize}
\item For $\tau=0$ we choose an identification of $N$ with $S^2\times[-1,1]$ such that $S=S^2\times\{0\}$ contains $K$.  We require that the characteristic foliation of  $\xi$ on $S$ is in standard form (c.f. \thmref{t:coarse unknot general}). 
\item For $\sigma=0$ and $\sigma=1$ we consider the same identification of $N$ with $S^2\times[-1,1]$.
\item When $\tau=1$ we choose a family of spheres $S_\sigma$ such that 
\begin{itemize}
\item $K_\sigma\subset S_\sigma$, and 
\item $S_0$ and $S_1$ coincide with the sphere $S^2\times\{0\}$ from the identifications chosen above.
\end{itemize} 
\item Now we extend $S_\sigma$ to a family of smooth foliations by spheres on $N$ such that $S_\sigma$ is a leaf on the foliation $N\times\{(\sigma,\tau=1)\}$. By \lemref{l:fam almost normal} we may assume that the characteristic foliation of $\xi$ on $S_\sigma$ has exactly two singular points along $K_\sigma$.  
\end{itemize}
This fixes the boundary conditions. The vertical dashed line in \figref{b:unique} in the back face $\tau=0$ corresponds to a constant family of Legendrian knots $K$  while the  thickened curve on the front face $\{\tau=1\}$ represents the family $K_\sigma$. 

In the next step we perturb the family of spheres in a particular way on a  neighborhood of  $S^2\times\partial\left([-1,1]\times I^2\right)$. 

On each sphere  $S^2\times[-1,1]\times\{(\sigma,\tau)\}$ with $(\sigma,\tau)\in\partial I^2$ containing the non-loose unknot $K_\sigma$, starting with $\tau=1$ and $\sigma=1/2$, we consider the retrogradient connection $\gamma_{\sigma,\tau}$ whose orientation coincides with the orientation of the Legendrian knot as one moves along $\partial I^2$. 

Because $K_0$ and $K_1$ have different orientations, the retrogradient connection one obtains after returning to $\tau=1,\sigma=1/2$ for the first time is opposite to the one  has started with. 

For each point  $(\sigma,\tau)\in\partial\left( I^2\right)$ we choose a small deformation of the family of spheres as follows:
\begin{itemize}
\item For the point $(\sigma=1/2,t=1)$ the family of spheres is unchanged.  
\item For all other points the deformation is constructed as follows. Fix  a small disc $D_{\sigma,\tau}$ intersecting the retrogradient connection chosen above (but not the other)  such that the characteristic foliation of $\xi$ on $D_{\sigma,\tau}$ has no singular points. To obtain the deformation we push the interior of the small discs slightly into the direction given by the coorientation of the spheres (and extending this deformation to nearby spheres). 

As $(\sigma,\tau)$ approaches $(\sigma=1/2,\tau=1)$, the size of the deformation converges to zero (with respect to every $C^r$-norm, $1\le r\in \Z$), so that we obtain a family of deformations depending smoothly on $(\sigma,\tau)$. 
\end{itemize}
Sufficiently small deformations as above do not introduce new singular points of the characteristic foliation on the deformed spheres and the retrogradient connections which formed the non-loose unknot before the deformation now appear on two different spheres (c.f. the discussion at the end of \secref{ss:coarse}) except for $\sigma=1/2$ and $\tau$ close to $1$.  

We use these deformations to extend the family of spheres from $\partial I^2$ to a neighborhood of $\partial I^2$ in $I^2$. By construction the retrogradient connections on $N\times\{\sigma,\tau\}$ for $(\sigma,\tau)$ in the interior of $I^2$ occur on different spheres except for  $\sigma=1/2$ and $\tau$ close to $1$. These are represented by the thickened dashed lines in \figref{b:unique}. 

\begin{figure}[htb]
\begin{center}
\includegraphics[scale=0.7]{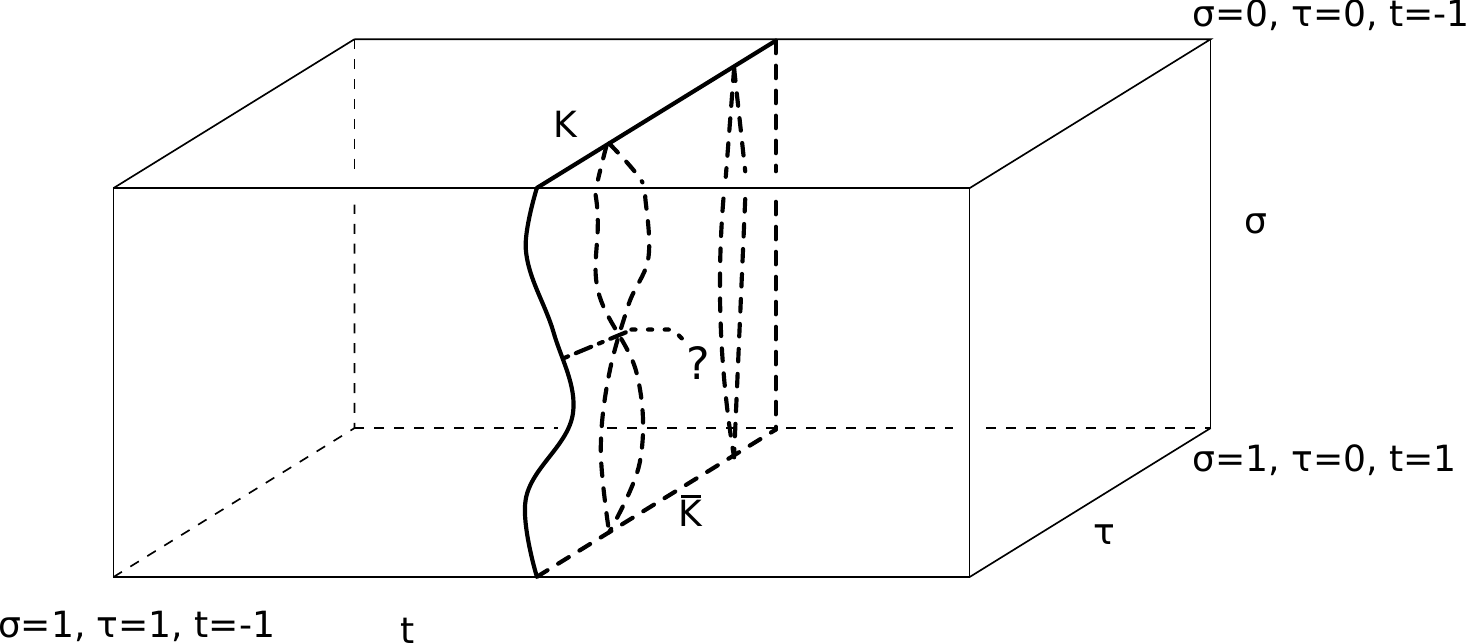}
\end{center}
\caption{Germ of a $1$-manifold obtained from an orientation reversing isotopy of $K$}\label{b:unique}
\end{figure}

According to  Corollary~\ref{c:fol contractible} we can extend the family of foliations by spheres on $N$ we have constructed for parameter values in a neighborhood of $I^2$ to the entire parameter space $I^2$. As usual, the spheres are parametrized by three parameters: $t\in[-1,1]$ and $(\sigma,\tau)\in I^2$, and we assume that the singularities of the characteristic foliations of $\xi$ on these spheres of $\xi$ are isolated. 

We now consider the set $\mathcal{L}$ of those parameter values $(t,\sigma,\tau)$  for which the corresponding sphere contains a non-loose piecewise smooth Legendrian unknot.
\medskip

{\bf Claim:} {\em $\mathcal{L}$ is a piecewise smooth properly embedded submanifold of codimension $2$ whose boundary is contained in $\partial([-1,1]\times I^2)$.}
\medskip

Before proving the claim, note that it implies the theorem since each submanifold component contributes an even number of boundary points. However, there is only one boundary point of $\mathcal{L}$ in $\partial\left([-1,1]\times I^2\right)$.
 
{\bf Proof of the claim:}   By \propref{p:manifold} we may assume that the points in $(-1,1)\times(0,1)^2$ where the corresponding characteristic foliation has exactly two retrogradient connections  is a codimension-$2$ submanifold. 
We denote the set of these generic points in $\mathcal{L}$ by $\mathcal{L}_{\textrm{gen}}$. Points where $\mathcal{L}$  might not be a  smooth submanifold  are 
\begin{enumerate} 
\item points where three retrogradient connections occur simultaneously, 
\item points where  one singularity of the characteristic foliation which is a limit point of a retrogradient connection is degenerate, and 
\item points where two retrogradient connections occur and there is an additional connection between two hyperbolic points.
\end{enumerate}
All these degenerations occur in isolated points of $(-1,1)\times I^2$ by \propref{p:manifold} and more degenerate situations do not occur at all.

In the following we treat the cases (2) and (3). After that we will show that $\mathcal{L}$ is compact and it turns out that all limit configurations of $\mathcal{L}_{gen}$ with three retrogradient connections can be treated in the same way as in \exref{ex:one triple}.

Let us now consider a sphere $S^2(t_0,\sigma_0,\tau_0)$ which contains a minimal non-loose unknot such that one of the singularities of the characteristic foliation, which is a limit point of a retrogradient connection, is degenerate. We will assume that this  singularity is negative. This situation was already discussed in \remref{rem:degenerate sing} on p.~\pageref{rem:degenerate sing}. 

The following considerations deal with a sufficiently small neighborhood of $p_0=(t_0,\sigma_0,\tau_0)$ in the parameter space. By our genericity assumptions, the degenerate singularity has index $0$ and it is of birth-death type. In particular, the set 
$$
\mathcal{S}=\left\{(t,\sigma,\tau)\,\left|\,\xi\left(S^2(t,\sigma,\tau)\right) \textrm{ has a degenerate singularity}\right.\right\}
$$ 
is a submanifold near $p_0$.  According to \remref{rem:degenerate sing}, $\mathcal{S}$ separates two surfaces defined by the presence of a retrogradient connection $A$ and $A'$ which involve the same stable leaf of a hyperbolic positive singularity of $\xi\left(S^2(t,\sigma,\tau)\right)$ for $(t,\sigma,\tau)$ close to $p_0$. While $A'$ defines a surface, $A$ determines a surface with boundary. This boundary is contained in the closure of the surface corresponding to $A'$ and it is also contained in $S$. Moreover, it is smooth. The other retrogradient connection $B$ is not affected by the degeneracy of the singular point and in general position (in a neighborhood of $p_0$) with respect to all subspaces of the parameter space mentioned so far (in particular to the boundary of $A$ where a retrogradient connection is degenerate).

This implies that $\mathcal{L}$ is a piecewise smooth submanifold of $P$ of codimension $2$ near $p_0$. The left-hand side of \figref{b:ABC} depicts the situation we have considered. The shaded plane marked with $S$ consists of those points in $P$  for which the corresponding characteristic foliation has a degenerate singularity. 

\begin{center}
\begin{figure}[htb]
\includegraphics[scale=0.85]{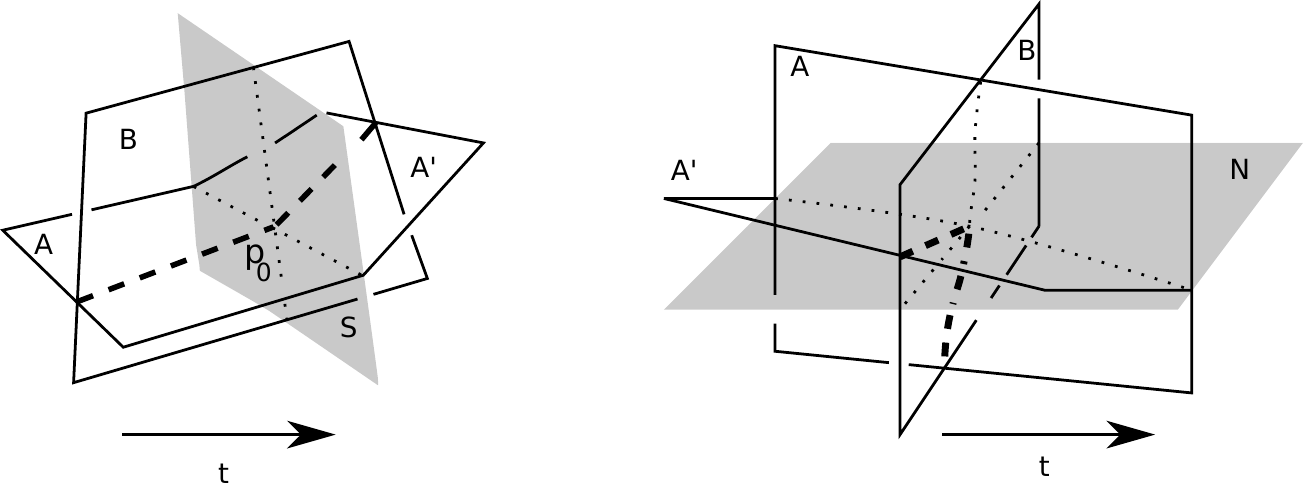}
\caption{Non-smooth points of $\mathcal{L}$}
\label{b:ABC}
\end{figure}
\end{center}

Similar considerations apply in he case when $\xi(S(p_0))$ has exactly two retrogradient connections and another connection $\gamma$ between two hyperbolic singularities $x,y$ of the same sign. We assume that $x,y$ are negative.

We consider the one of the configurations for which one has to appeal to transversality and which leads to a non-smooth point in $\mathcal{L}$. The other cases are analogous. Assume that 
\begin{itemize} 
\item[(i)] an unstable leaf of $x$ is one of the  retrogradient connections,
\item[(ii)] $\gamma$ is a stable leaf of $x$, and 
\item[(iii)] the other retrogradient connection $B$ is connected to $y$ using a path in $\Gamma_-^*$ which does not pass through $x$. 
\end{itemize} 

Let $z$ be the positive singularity at the end of the retrogradient connection which is opposite to $x$ and $U$  a  neighborhood of $p_0\in (-1,1)\times(0,1)^2$. Since all singular points are non-degenerate $x,y$, respectively $\gamma$ is part of a family of hyperbolic singularities respectively unstable leaves of a hyperbolic singularity defined on $U$.  Let $N$ be the local hypersurface such that for parameters values in $U$ the leaf $\gamma$ connects $y$ to $x$.

We choose $U$ so small that $U\setminus N$ has two connected components. In one of these components there are always precisely two retrogradient connections $A,B$. When $A,B$  occur on the same sphere, then there is a non-loose unknot on that sphere (formed by a path containing the two retrogradient connections and paths in the  components   $\Gamma_-'\subset\Gamma_-^*$ and $\Gamma'_+\subset\Gamma_+^*$  which satisfy $e_\pm(\Gamma'_\pm)+1=h_\pm(\Gamma'_\pm)$ (see \defref{d:p-} on p.~\pageref{d:p-}).

By \lemref{l:rgc} there is one additional 2-parametric family of parameter values $p\in U$ such that there is a retrogradient connections $A'$ from $y$ to $z$. The retrogradient connection $A$ (from $x$ to $z$) and $A'$ (from $y$ to $z$) never occur on the same sphere. Because of assumption (iii), a non-loose unknot is now formed when  the retrogradient connections $A'$ and $B$  occur simultaneously (c.f. the right-hand side in \figref{b:ABC}) but not when $A$ and $B$ appear on the same sphere. Note that of (iii) is not satisfied, then $p_0$ is a smooth point of $\mathcal{L}$ defined by the intersection of the surfaces corresponding to $A$ and $B$.

In order to finish the proof of the claim it remains to show that $\LL$ is a {\em compact} piecewise smooth submanifold.  

Let $S_n$ be a sequence of leaves of $\FF_{\sigma_n,\tau_n}$ containing a minimal non-loose Legendrian knot $K_n$. We may assume that 
\begin{itemize}
\item $S_n$ converges to a sphere $S_\infty$ (we denote the corresponding parameter values by $t_\infty,\sigma_\infty,\tau_\infty$) of $\FF_{\sigma_\infty,\tau_\infty}$. 
\item $K_n$ converges to a union $\mathcal{K}$ of leaves of $\xi(S_\infty)$ which is a closed subset of $S_\infty$. 
\end{itemize} 
Since we understand $\mathcal{L}$ near the boundary of $[-1,1]\times I^2$ we are only  interested in the case when the  limit point $(t_\infty,\sigma_\infty,\tau_\infty)$ lies in the interior of $[-1,1]\times I^2$.

$\mathcal{K}$ cannot contain a non-degenerate closed leaf since every sphere close enough to $S_\infty$ contains closed leaf contradicting the fact that $K_n$ is non-loose. For degenerate closed leaves \lemref{l:birth} asserts that for $t>t_\infty$ or $t<t_\infty$ the sphere $S(t,\sigma_\infty,\tau_\infty)$ contains an attractive closed leaf. This prevents the existence of non-loose unknots in spheres close to $S_\infty$. We may therefore assume that $\xi(S_\infty)$ has no closed leaf. 

If $S_\infty$ contains a cycle (made from stable and unstable leaves of hyperbolic singularities) with one sided holonomy, then after an isotopy of $S_\infty$ which moves $S_\infty$ into its complement the characteristic foliation on the isotoped sphere contains an overtwisted disc. This does not require any assumption on the signs  singular points on the cycle and uses only  the Poincar{\'e}-Bendixson theorem: If $\gamma$ is the cycle and $C$ a closed curve transverse curve such that $C$ and $\gamma$ bound an annulus, then a perturbation  as in \figref{b:cycle} makes a closed leaf appear.    

This contradicts   $(t_\infty,\sigma_\infty,\tau_\infty)$ being the limit of parameter values corresponding to spheres containing non-loose unknots. Therefore, $S_\infty$ does not contain any limit cycles with one sided holonomy and all leaves of $\xi(S_\infty)$ have singular points as limit sets.

\begin{center}
\begin{figure}
\includegraphics[scale=0.6]{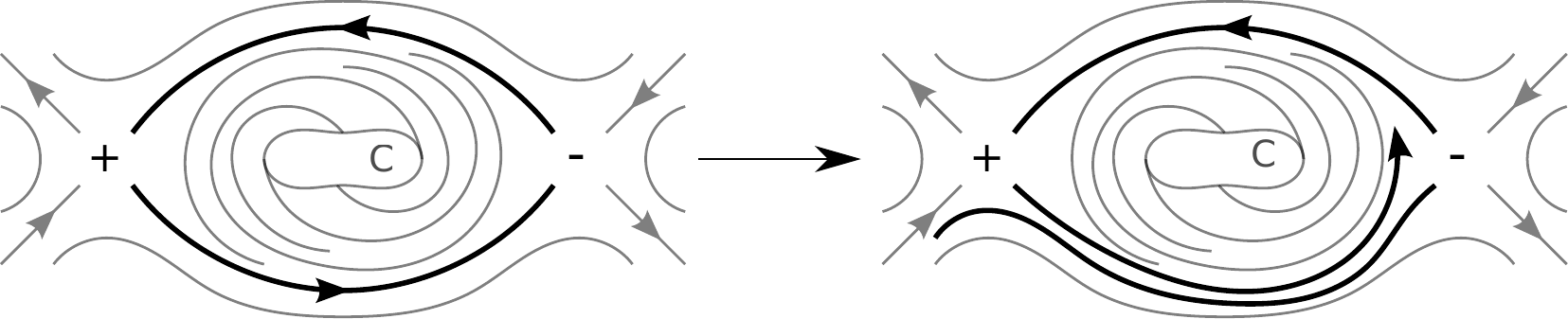}
\caption{Perturbations of cycles with one-sided holonomy result in overtwisted discs.}
\label{b:cycle}
\end{figure}
\end{center}

Let $\Gamma_-^*$ be the graph consisting of negative singularities of the characteristic foliation on $S_\infty$ and their unstable leaves which do not take part in retrogradient connections. If a connected component of $\Gamma_-^*$ is not a tree, then $N\setminus S_\infty$ contains overtwisted discs. This is a contradiction to the fact that there are non-loose unknot arbitrarily close to $S_\infty$. 

By construction, $S_\infty$ contains at least two retrogradient connections and the complement $N\setminus S_\infty$ is tight. Assume that $S_\infty$ contains exactly two retrogradient connections. Then $\Gamma_-^*$ has three connected components. If each of these components takes part in at most one retrogradient connection, then one can eliminate all hyperbolic singularities on $S_\infty$ using \lemref{l:parametric elim}:

Let $S_\infty\times[\delta,\delta], \delta>0,$ be a tubular neighborhood of $S_\infty$ such that the characteristic foliation on each boundary sphere of this neighborhood has only two elliptic singularities. Then one can deform the contact structure on $N\times\{\sigma_\infty,\tau_\infty\}$ so  that 
\begin{itemize}
\item nothing changes away from $N\times(-\delta,\delta)$, and
\item the characteristic foliation on $S_\infty\times\{s\}$ for $s\in(-\delta,\delta)$ never has hyperbolic points after the deformation. 
\end{itemize}
Then $S_\infty$ has a tight neighborhood and $\xi$ itself would be tight. Therefore, there is one component of $\Gamma_-$ which takes part in both retrogradient connections and the same is true for $\Gamma_+$ (consisting of positive singular points and their stable leaves). The union of these two components contains a non-loose Legendrian unknot whose Thurston-Bennequin invariant is one.    

Assume now that the limit sphere $S^2(t_\infty,\sigma_\infty,\tau_\infty)$ contains three retrogradient connections. Since points of $\mathcal{L}_{\mathrm{gen}}$ lie arbitrarily close to $(t_\infty,\sigma_\infty,\tau_\infty)$, this configuration is obtained from a pair of simultaneous retrogradient connections $A,B$  as in the top part of \figref{b:n=1} (on p.~\pageref{b:n=1})
by adding a third retrogradient connection $C$ such that the associated negative hyperbolic singularities lies on one component of $\Gamma_-(t_\infty)$ with the unstable leaves which take part in $A,B$ removed.



Some configurations obtained in this way are less interesting than others: For example, the case when the new retrogradient connection is trivial in the sense that the bifurcation does not affect which components of $\Gamma_-^*$ are connected directly through unstable leaves in $\Gamma_-\setminus\Gamma_-^*$ (as in the left-hand  part of \figref{b:trivial config})  corresponds to smooth points of $\mathcal{L}$. We also do not have to consider the case when the new retrogradient connection (i.e. the one different from $A,B$)  leads to obvious overtwisted discs which are not affected by $A$ or $B$ (as in the right-hand part of \figref{b:trivial config}). Finally, note that \lemref{l:rgc} restricts the possible bifurcations. 

\begin{center}
\begin{figure}[th]
\includegraphics[scale=0.8]{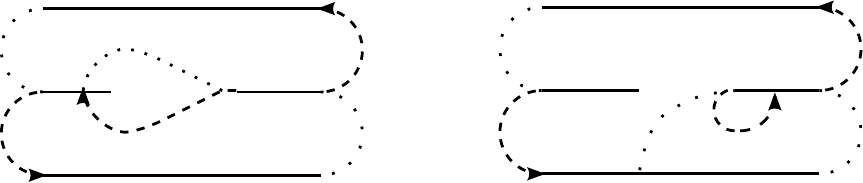}
\caption{Simple examples of configurations with three retrogradient connections.}
\label{b:trivial config}
\end{figure}
\end{center}

Up to combinatorial equivalence (rotating a diagram by $180^\circ$) all remaining possibilities are listed in \figref{b:triple in closure of L} on p.~\pageref{b:triple in closure of L}. Two of these (left column in the middle and right column at the bottom) cannot appear as limit points $(t_\infty,\sigma_\infty,\tau_\infty)$ of $\mathcal{L}_{\mathrm{gen}}$ because there are obvious overtwisted discs on $S^2\times \{t\}\times\{\sigma_\infty,\tau_\infty)\}$ for $t>t_\infty$ or $t<t_\infty$. The top left configuration was analyzed in \exref{ex:one triple}, the remaining three configurations can be treated in the same fashion. 

\begin{center}
\begin{figure}[ht]
\includegraphics[scale=0.8]{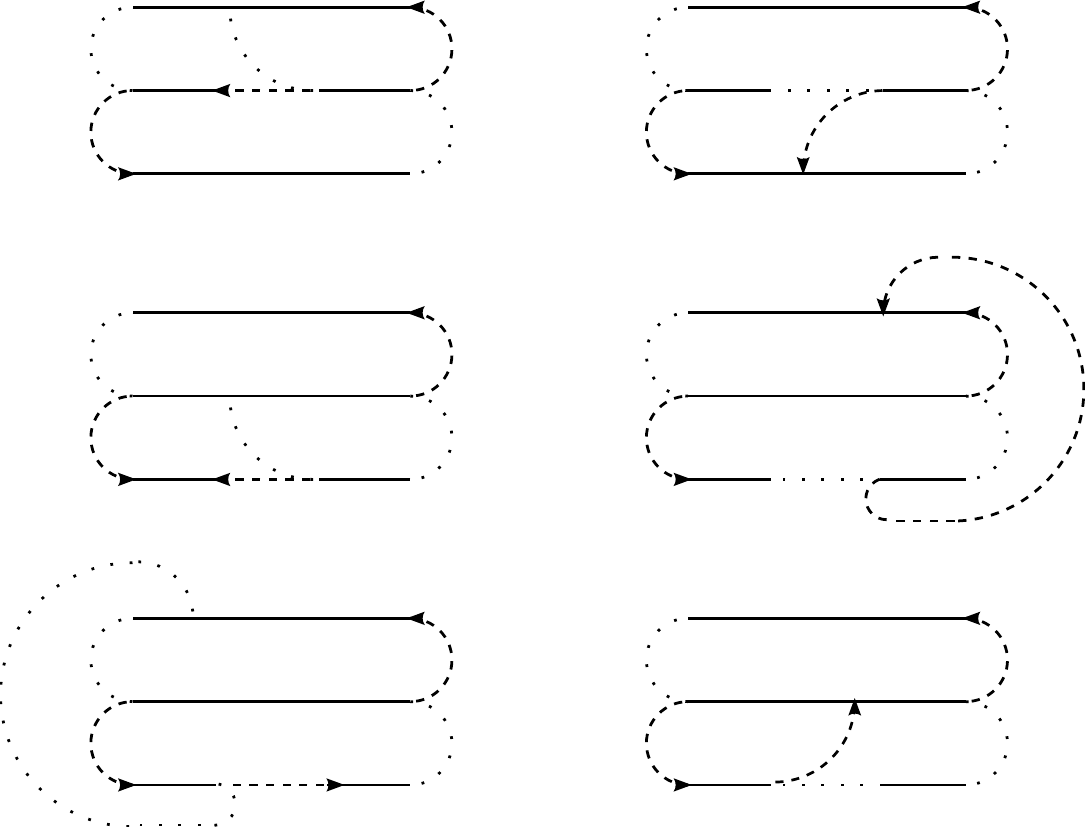}
\caption{Remaining configurations with three simultaneous retrogradient connections.}
\label{b:triple in closure of L}
\end{figure}
\end{center}


We have already treated the case when the characteristic foliation on $S^2(t_\infty,\sigma_\infty,\tau_\infty)$ contains a degenerate singularity or on additional non-retrogradient connection between hyperbolic singular points. Therefore,  $\mathcal{L}$ is a piecewise smooth submanifold of codimension $2$ in $[-1,1]\times I^2$ which is compact and properly embedded.
\end{proof}

Let us discuss two aspects of the above proof:
\begin{enumerate}
\item All configurations with three retrogradient connections in $\overline{\mathcal{L}}_{gen}$   have the following property: $\Gamma_-^*$ consists of four connected components and one of them is obtained by splitting one of the components of the analogous graph for a minimal non-loose unknot with two retrogradient connections into two pieces. Using the normal form from \thmref{t:coarse unknot general} for a minimal non-loose unknot  one sees that one of the components $\Gamma'$ of $\Gamma_-^*$  satisfies  $e_-(\Gamma')=h_-(\Gamma')$. 

Since this component is a tree, one can eliminate $\Gamma'$  completely by \lemref{l:parametric elim}  without creating closed leaves. In particular, one of the three retrogradient connections on $S_\infty$ disappears and 
 no new triples of simultaneous retrogradient connection appear (use Lemma 2.15 in \cite{gi-bif}).  A similar argument can be used in the above proof for degeneracies of type (3). 

This  argument only shows that $\mathcal{L}$ is piecewise smooth after an additional operation on the foliation by spheres. Of course, this would suffice for our purposes. 
\item We did not show that there is a {\em continuous} (in the Hausdorff topology) family of non-loose unknots $K(x)$ with $x\in\mathcal{L}$.
\end{enumerate}

The argument used in the previous proof can be used to show the following, very similar, result.  In the refereeing process it was pointed out to the author that the claim of the following theorem goes back to Y.~Chekanov. Unfortunately, there seems to be no printed reference.

Before we state the theorem, we fix the setup. Let $K$ be a non-loose unknot in $(S^3,\xi_{-1})$ with some orientation. We denote the unknot obtained as positive stabilization of $K$ by $\Delta$, $K$ with reversed orientation is $\overline{K}$ and the positive stabilization of $\overline{K}$ is $\widehat{\Delta}$. Note that $\Delta,\widehat{\Delta}$ are  unknots with  vanishing Thurston-Bennequin invariant and $\mathrm{rot}(\Delta)=\mathrm{rot}(\widehat{\Delta})=1$. 

\begin{thm}[Chekanov] \label{t:not transitive}
$\widehat{\Delta}$ is not isotopic to $\Delta$. 
\end{thm}

\begin{proof}
Assume that there is an isotopy $\varphi_\sigma$ of $S^3$ which moves $\Delta$ to $\widehat{\Delta}$. As in the previous proof $\varphi_\sigma(\Delta),\sigma\in I=[0,1],$ misses two balls which we assume to be so small that they are contained in Darboux domains and have convex boundary. The complement of these balls is $N\simeq S^2\times[-1,1]$. We reconsider the setting of the proof of \thmref{t:different or, not isotopic}. On $N\times I^2$  we consider the restriction of $\xi_{-1}$ to $N$ on all $N\times\{\sigma,\tau\}$. We fix a family of foliations $\FF_{\sigma,\tau},\sigma,\tau\in I$ by spheres on $N$ so that
\begin{itemize}
\item $\FF_{\sigma,\tau}$ is constant near $\partial N$, 
\item a sphere of $\FF_{\sigma=0,\tau=1}$ contains $\Delta$,
\item $\FF_{\sigma,1}=\varphi_\sigma(\FF_{0,1})$, 
\item $\FF_{0,\tau}, \FF_{1,\tau}$ is independent of $\tau$,
\item on $N\times\{\sigma,\tau=0\}$ we pick a family of foliations interpolating between $\FF_{\sigma=0,0}$ and $\FF_{\sigma=1,0}$ (already fixed)  as follows: a leaf of $\FF_{1/2,0}$ contains a non-loose unknot $K$. To obtain $\FF_{\sigma,0}$ for $\sigma\le 1/2$ and $\sigma\ge 1/2$ first deform $\FF_{1/2,0}$ to make overtwisted discs appear which are isotopic to $\Delta$, respectively $\widehat{\Delta}$, for $\sigma<1/2$, respectively $\sigma>1/2$, and use the isotopy to extend the given family of foliations to a family of foliations for $(\sigma,\tau)\in\partial( I^2)$ such that precisely one leaf of one foliation contains a non-loose unknot (namely $K$ on a leaf of $\FF_{1/2,0}$).
\end{itemize}
As in the proof of \thmref{t:different or, not isotopic} one shows that such a configuration is not possible contradicting the assumption that $\Delta$ is isotopic to $\widehat{\Delta}$. 
\end{proof}
The above theorem can be rephrased:  The identity component of $\mathrm{Diff}_+(S^3,\xi_{-1})$ does not act transitively on the set of boundaries of overtwisted discs. 

Finally, we note the following corollary of the proof of \thmref{t:non or isotopy}. This corollary, or a statement equivalent to it, must have been known to Y.~Chekanov (see \thmref{t:chekanov} below). 

\begin{cor} \label{c:part of chekanov}
Let $\psi\in\mathrm{Diff}_+(S^3,\xi_{-1})$ such that $\mathrm{d}(\psi)=0$ and $\psi(K)$ is isotopic to $K$ as oriented Legendrian knot. Then $\psi$ is isotopic to the identity through contact diffeomorphisms. 
\end{cor}
\begin{proof}
We may assume that $\psi$ fixes two disjoint Darboux balls and that $\psi(K)$ is isotopic to $K$ in the complement $N$  of these balls. There is a smooth isotopy $\psi_\sigma,\sigma\in I$, with $\psi_1=\psi$ and $\psi_0=\mathrm{id}$ which does not move the two Darboux balls. 

Now apply the construction of the proof of \thmref{t:non or isotopy}. For this we fix  contact structures $\xi(\sigma,\tau)$ and product decompositions of $N$ for $(\sigma,\tau)\in \partial I^2$ as in the proof of \thmref{t:non or isotopy}. Since $\psi(K)$ is isotopic to $K$ through Legendrian knots   in $N\times I\times\{\tau=0\}$. Therefore we can perturb the product decompositions of $N$ near $\{\tau=0\}$ and $\{\tau=1\}$ so that $T_+(\sigma,\tau)\neq T_-(\sigma,\tau)$ for $t\not\in\{0,1\}$ or $\sigma\not\in\{0,1\}$. (Recall that the assumption that $K$ is isotopic to $\overline{K}$ led points $(\sigma,\tau)$ in near $\partial I^2$ where $T_+(\sigma,\tau)=T_-(\sigma,\tau)$ in the proof of \thmref{t:different or, not isotopic}.)  

Since $\mathrm{d}(\psi)=0$, the proof of \thmref{t:non or isotopy} yields a deformation of the smooth isotopy $\psi_\sigma$ from $\mathrm{id}$ to $\psi$ into a contact isotopy relative to $\psi_0$ and $\psi_1$. 
\end{proof}

\section{Applications} \label{s:appl}

We give the proof of a theorem of Y.~Chekanov about the contact mapping class group of overtwisted contact structures on $S^3$. The proof of that theorem also requires further results on the action of the contactomorphism group on boundaries of overtwisted discs, these will be discussed first. 

We conclude with the classification of Legendrain unknots in overtwisted conact structures on $S^3$. So far, we were mostly concerned with minimal non-loose unknots.

\subsection{The action of contactomorphisms on the set of boundaries of overtwisted discs}
It is an easy corollary of Eliashberg's classification result that the group of all contactomorphisms of an overtwisted contact structure $\xi$  acts transitively on the set $\mathcal{U}_0(\xi)$ of Legendrian unknots with vanishing Thurston-Bennequin invariant and rotation number one. In this section we will show that the connected component of the identity of $\mathrm{Diff}_+(S^3,\xi_{-1})$ does act transitively on $\mathcal{U}_0(\xi_{k})$ when $k\neq -1$.


\begin{lem} \label{l:ot disc part of Lutz}
Let $(M,\xi)$ be a contact manifold and $\Delta\in\mathcal{U}_0$. Then there is a Lutz tube along a transverse unknot with self-linking number $-1$ containing $\Delta$ such that the Lutz tube is contained in a ball whose boundary is convex and has a tight neighborhood. If the $\pi$-Lutz twist is undone, then the contact structure  becomes tight on the ball. 
\end{lem}
$\pi$-Lutz twists as in this lemma will be called simple. 
\begin{proof}
First of all, note that one can choose a convex disc $D$ bounding $\Delta$ such the dividing set on $D$ is connected. In order to see this consider any convex disc $D$ bounding $\Delta$, fix an overtwisted disc $D'$ in the complement of $D$  and construct a contact structure with the desired properties using \thmref{t:eliashclass}. For this one uses assumption that $\mathrm{rot}(\Delta)=1$. 

Now consider a plane field $\zeta$ on $M$ which is homotopic to $\xi$ satisfying the following conditions.
\begin{itemize}
\item $\zeta=\xi$ near $\Delta$.
\item $\zeta=\xi$ near $D'$.
\item $\zeta$ is a contact structure on a closed ball $B^3$ such that
\begin{itemize}
\item $B^3$ is disjoint from $D'$ and contains $D$ in its interior, 
\item the boundary of $B^3$ is convex with respect to $\zeta$ and the dividing set on $\partial B^3$ is connected, and
\item $\zeta\eing{B^3}$ is obtained from the tight contact structure determined by $\zeta(\partial B^3)$ by a single $\pi$-Lutz twist along a transverse unknot and $D$ is one of the obvious overtwisted discs obtained from a  $\pi$-Lutz twist.  
\end{itemize}
\item $\zeta$ is homotopic to $\xi$ as a plane field.
\end{itemize}
According to \thmref{t:eliashclass} $\zeta$ is homotopic to a contact structure and $\zeta$ is isotopic to $\xi$ relative to a neighborhood of $\Delta$ (one uses the overtwisted disc $D'$ for the application of \thmref{t:eliashclass}), so we can assume that $\zeta$ itself is a contact structure. The ball with the desired properties is obtained from $B^3$ and the contact structure $\zeta$ using the fact that $\xi$ and $\zeta$ are isotopic contact structures. 
\end{proof}
The last lemma allows us to establish a contact topological (rather than homotopy theoretic) criterion, which distinguishes $\xi_{-1}$ from those positive overtwisted contact structures on $S^3$ which are not diffeomorphic to $\xi_{-1}$.

\begin{prop} \label{p:xi1 detect}
Let $\xi$ be an overtwisted contact structure on $S^3$. Then $\xi$ is isotopic to $\xi_{-1}$ if and only if there is a closed ball $B^3\subset M$
\begin{itemize}
\item which contains a simple Lutz tube, 
\item $\partial B^3$ has a tight neighborhood, and $S^3\setminus B^3$ is tight. 
\end{itemize} 
\end{prop}

\begin{proof}
First, note that according to the  description of $\xi_{-1}$ outlined in \exref{ex:Lutz}  this contact structure is obtained from the tight contact structure on $S^3$ by a single Lutz twist along a transverse unknot with self-linking number $-1$. Thus, there is a ball containing that Lutz twist whose complement is tight.

Conversely, if there is ball with tight complement and convex boundary who contains a simple Lutz twist in its interior, then undoing the Lutz twist on $B^3$ we obtain a tight contact structure on $B^3$. By Colin's gluing theorem \cite{col-connect} undoing the Lutz twist yields a tight contact structure on $S^3$. In other words, $\xi$ is obtained from the tight contact structure on $S^3$ by a single $\pi$-Lutz twist. Then $\xi$ is homotopic to (and hence isotopic) to $\xi_{-1}$. 
\end{proof}

This  has the following consequence.

\begin{thm} \label{t:three ot discs}
Let $\xi$ be an overtwisted contact structure on $S^3$ which is not isotopic to $\xi_{-1}$. Then for every pair $\Delta_1,\Delta_2\in\mathcal{U}_0(\xi)$ there is an overtwisted disc $\Delta$ which is disjoint from $\Delta_1\cup\Delta_2$.  
\end{thm}

\begin{proof}
According to \lemref{l:ot disc part of Lutz} we may assume that $\Delta_1$ is obtained by a $\pi$-Lutz twist along a transverse unknot in a tight ball $B^3$ with convex boundary. By \propref{p:xi1 detect}, the complement of this ball is overtwisted. 

We will use a certain procedure to attempt to isotope $\Delta_2$ relative to  $\Delta_1$ so that the result lies in the union of $B^3$ with a neighborhood $V$ of $B^3$ such that $\overline{V}\setminus \ring{B}^3$ is tight. According to \cite{col-connect} the complement of $V$ (and not only the complement of $B^3$) contains an overtwisted disc. The process described below either works or not. 

If this procedure works, then there is an overtwisted disc in $S^3\setminus V$, i.e. in the complement of a link which is Legendrian isotopic to  $\Delta_1\cup \Delta_2$. This implies the result in this case. 

It will turn out that if the procedure does not work, then this is due to the presence of an overtwisted disc in $S^3\setminus V$ which is disjoint from a link which is Legendrian isotopic to  $\Delta_1\cup\Delta_2$. Again, the result follows.  

Let $\widehat{D}$ be a disc bounding $\Delta_2$. Without loss of generality we assume that $\Delta_2$ and $\widehat{D}$ are transverse to $\partial B^3$. There is a disc $D\subset\widehat{D}\setminus \ring{B}^3$ such that
\begin{itemize}
\item the interior of $D$ is disjoint from $\Delta_2$, $\xi|_{\delta_2}=T\widehat{D}|_{\delta_2}$, and
\item $\partial D$ is the union of an arc $\delta_2\subset\Delta_2$ with an arc $\delta_B\subset D\cap\partial B^3$.
\end{itemize}
After an isotopy of $D$ relative to $\delta_2$ we obtain a disc $D'$ with piecewise smooth Legendrian boundary  such that $\delta_2'=\partial D'\setminus\delta_2$ is contained in a tubular neighborhood $V'$ of $\partial B^3$ where $\xi$ is tight. (Recall that $\Delta_1$ is obtained by a $\pi$-Lutz twist inside a tight ball $B^3$. The boundary of $B$ still has a tight neighborhood after the $\pi$-Lutz twist.) 

We may also assume that $D'$ is convex since we may stabilize $\delta_2'$. If the dividing set of $D'$ contains a closed component, then by \thmref{t:char convexity} every neighborhood of $D'$ contains an overtwisted disc which is disjoint from $\Delta_2$.  If $D'$ has a tight neighborhood, then by \propref{p:contact squeezing} there is a contact isotopy supported in a small neighborhood of $D'$ which moves a ball $B'$ containing $D'$ into  $V=B^3\cup V'$ and preserves $\Delta_2\setminus B'$ as a set. 


This process can be iterated. After finitely many steps we either found an overtwisted disc in the complement of a Legendrian link which is isotopic to $\Delta_1\cup \Delta_2$, or we have isotoped $\Delta_2$ into the union of $B^3$ with a tight tubular  neighborhood of $\partial B^3$. 
\end{proof}



\subsection{Chekanov's theorem on $\pi_0(\mathrm{Diff}_+(S^3,\xi))$ for overtwisted contact structures $\xi$} \label{s:chekanov}

The following theorem is stated in Remark 4.15 of \cite{elfr}. It is attributed to Y.~Chekanov without an indication of a proof.

\begin{thm}[Chekanov] \label{t:chekanov}
Let $\xi$ be an overtwisted oriented contact structure on $S^3$ and $\mathrm{Diff}_+(S^3,\xi)$ the group of diffeomorphisms of $S^3$ which preserve $\xi$ and its orientation. 
Then
\begin{align*}
\pi_0\left(\mathrm{Diff}_+(S^3,\xi)\right) & \simeq \left\{\begin{array}{rl} \Z_2\oplus \Z_2 & \textrm{ if } \xi\simeq \xi_{-1} \\
\Z_2 & \textrm{otherwise.}
\end{array} \right.
\end{align*}
\end{thm} 
\begin{proof}
According to \cite{dym} the homomorphism 
$$
\mathrm{d} : \mathrm{Diff}_+(S^3,\xi)  \lra \Z_2
$$
which was described at the end of \secref{s:homotopy} is onto for all overtwisted contact structures on $S^3$. 

We first assume $\xi\simeq \xi_{-1}$. Let  $K\subset S^3$ be a non-loose Legendrian unknot with $\mathrm{tb}(K)=1$. According to \thmref{t:non or isotopy} and \thmref{t:different or, not isotopic} there is a well-defined group homomorphism
\begin{align} \label{e:kappa}
\begin{split}
\kappa : \mathrm{Diff}_+(S^3,\xi) &  \lra \Z_2\\
\psi & \longmapsto\left\{ \begin{array}{rl} 0 & \textrm{ if } \psi(K) \textrm{ is Legendrian isotopic to } K \\
1 &   \textrm{ if } \psi(K) \textrm{ is Legendrian isotopic to } \overline{K}.
\end{array} \right.
\end{split}
\end{align}
By \lemref{l:K preserved,dym=1}, there is an orientation preserving contact diffeomorphism $\psi$  such that $\mathrm{d}(\psi)\neq 0$ which maps $K$ to itself but reverses its orientation. Moreover, when one applies the proof of surjectivity of $\mathrm{d}$ to the ball surrounding one of the overtwisted discs shown in \figref{b:front1} (as in indicated in \remref{r:dym surj}), then the resulting contact diffeomorphism $\varphi$  satisfies $\mathrm{d}(\varphi)=1$ and $\varphi$ preserves the overtwisted disc up to isotopy. 

As we have shown above, a contact diffeomorphism of $\xi_{-1}$ which reverses to orientation of the non-loose unknot interchanges (up to contact isotopy) the two overtwisted discs in \figref{b:front1}. This is true because the positive stabilization of the standard non-loose Legendrian unknot yields one of the overtwisted discs in \figref{b:front1} while the negative stabilization (i.e. the positive stabilization when the orientation of $K$ is reversed) yields the other overtwisted disc. Since $\varphi$ preserves one of the two overtwisted discs in \figref{b:front1} and these two overtwisted discs are not isotopic, $\varphi(K)$ is isotopic to $K$ as oriented knot. Thus, $(\mathrm{d},\kappa)$ is a surjective homomorphism
$$
(\mathrm{d},\kappa) : \mathrm{Diff}_+(S^3,\xi_{-1}) \lra \Z_2\oplus\Z_2. 
$$
Every orientation preserving contact diffeomorphism  $\psi$ in the kernel of this map is isotopic to the identity inside of $\mathrm{Diff}_+(S^3,\xi
_{-1})$ by \corref{c:part of chekanov}. This finishes the proof of the theorem in the case when $\xi\simeq\xi_{-1}$. 

We now consider the case $\xi\not\simeq\xi_{-1}$. Let $\psi$ be a contact diffeomorphism of $(S^3,\xi)$ which preserves the orientation of $\xi$ and $\mathrm{d}(\psi)=0$ and choose an overtwisted disc $D$. By \thmref{t:three ot discs}  the complement of $\partial D\cup \psi(\partial D)$ is overtwisted. Therefore, we may assume that $\psi$ preserves an overtwisted disc and since $\mathrm{d}(\psi)$ is trivial we can apply \thmref{t:eliashclass} to obtain an isotopy connecting $\psi$ to the identity. 
\end{proof}
The following consequence is immediate. We denote the connected component of the space of contact structure on $S^3$ which contains a fixed contact structure $\xi$ by $\mathrm{Cont}(S^3,\xi)$.
\begin{cor}
Let $\xi$ be an overtwisted contact structure on $S^3$. Then 
$$
\pi_1\left(\mathrm{Cont}(S^3),\xi\right)\simeq  \left\{\begin{array}{rl} \Z_2\oplus \Z_2 & \textrm{ if } \xi\simeq \xi_{-1} \\
\Z_2 & \textrm{otherwise.} \end{array}\right.
$$
\end{cor}
\begin{proof}
This follows from the long exact sequence of the fibration
$$
\mathrm{Diff_+(S^3,\xi)} \longrightarrow \mathrm{Diff}(S^3,\mathrm{or}) \lra \mathrm{Cont(S^3,\xi)}.
$$
One can choose $\xi$ to be invariant under rotations around one complex plane in $\mathbb{C}^2\supset S^3$. This shows that the map $\pi_1(\mathrm{Diff}(S^3,\mathrm{or})\simeq\pi_1(\mathrm{SO}(4))\lra \pi_1(\mathrm{Cont}(S^3,\xi))$ is trivial. 
\end{proof}

\subsection{Completion of the classification of Legendrian unknots in overtwisted contact structures on $S^3$} \label{ss:Leg class} 
 
As final application of the classification of minimal non-loose unknots  and of \thmref{t:chekanov} we complete the classification of Legendrian unknots in $S^3$. First, we consider  the contact structure $\xi_{-1}$. We now deal with non-loose unknots, then we consider the loose case. 

We already know the coarse classification for non-loose unknots: If $K\subset(S^3,\xi)$ is a non-loose unknot, then
\begin{align*}
\tb(K) & = n >0 & \mathrm{rot}(K) & = \pm(n-1) 
\end{align*} 
and there is an orientation preserving contact diffeomorphism mapping $K$ to $K_{n,1}$, respectively $K_{-n,-1}$, if $\mathrm{rot}(K)=n-1$, respectively $\mathrm{rot}(K)=-(n-1)$. Let us consider the first case.  We know two examples of knots with the same classical invariants as $K$, namely $K_{n,1}$ and $K_{-1,-n}$.

Furthermore, we have explained (\figref{b:isotopy}) that after $n-1$ negative stabilizations of $K_{1,n}$, respectively $K_{-1,-n}$, we end up with  $K_{1,1}$, respectively $K_{-1,-1}$. These two knots are not Legendrian isotopic by \thmref{t:different or, not isotopic}.  Therefore, $K_{n,1}$ is not isotopic to $K_{-1,-n}$. 

\begin{thm} \label{t:non-loose class}
Let $K\subset S^3$ be a non-loose Legendrian unknot with $\mathrm{tb}(K)=n>0$ and $\mathrm{rot}(K)=n-1$. Then $K$ is isotopic to either $K_{n,1}$ or $K_{-1,-n}$. 
\end{thm}

\begin{proof}
By Chekanov's  theorem and the coarse classification of non-loose unknots, $K$ is isotopic to one of the following knots:
$$
K_{n,1},\psi(K_{n,1})=K_{-1,-n}, \varphi(K_{n,1}), \varphi(\psi(K_{n,1}))
$$
where $\psi$ is the contactomorphism from \exref{ex:orrev} and $\varphi$ is a contactomorphism with $d(\varphi)=1$ and $\kappa(\varphi)=0$. We already know that $K_ {n,1}$ and $K_{-1,-n}$ are not isotopic. Therefore it suffices to show that $\varphi(K_{n,1})$ is isotopic to $K_{n,1}$. 

We may assume that $\varphi$ preserves $K_{1,1}$ together with a tubular neighborhood $N_0$ of $K_{1,1}$ pointwise. The claim follows if we show that $K_{n,1}$ is isotopic to a knot contained in $N_0$. 

Now $K_{n,1}$ can be obtained from $K_{1,1}$ by a $(n-1)$-fold band connected sum with $K_{1,0}$.  The $(n-1)$ copies of $K_{1,0}$ are unlinked copies of the boundary of an overtwisted disc (as in \figref{b:tb=n}). 

\begin{figure}[htb]
\begin{center}
\includegraphics[scale=0.7]{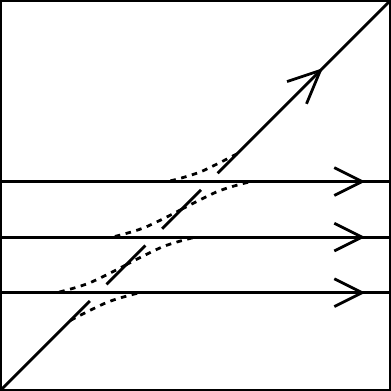}
\end{center}
\caption{Obtaining $K_{4,1}$ from $K_{1,1}$ and three copies of $L_{1,0}$}\label{b:tb=n}
\end{figure}

This collection of $n-1$ unknots is Legendrian isotopic  to a Legendrian link contained in $N_0$. But this is clear since a negative stabilization of $K_{1,1}$ (the stabilization is contained $N_0$) is isotopic to $K_{1,0}$. We have seen in \figref{b:isotopy} that this  isotopy can be chosen so that it does not move $K_{1,1}$ nor the region were the band connected sum is performed. 
\end{proof}

Again relying on \thmref{t:chekanov} we now give a complete classification of loose unknots in $(S^3,\xi)$ up to isotopy. Recall that Eliashberg and Fraser showed in \cite{elfr-long} that two  Legendrian unknots  $K_1,K_2$ in an overtwisted contact structure are isotopic when 
\begin{align*}
\tb(K_1) & = \tb(K_2) <0  & \mathrm{rot}(K_1) & = \mathrm{rot}(K_2).
\end{align*}
This is complemented by the following proposition. Before we state it, we fix $\psi\in\mathrm{Diff}_+(S^3,\xi_{-1})$ with $\kappa(\psi)\neq 0$.

\begin{prop} \label{p:loose class}
Let $K$ be a loose Legendrian unknot with $\mathrm{tb}(K)\ge 0$ in $(S^3,\xi_{-1})$. Every other loose Legendrian unknot $L\subset(S^3,\xi_{-1})$ with the same classical invariants is Legendrian isotopic to either $K$ or $\psi(K)$  
\end{prop}

\begin{proof}
Because $L$ and $K$ are loose, one can use \thmref{t:eliashclass} directly to construct a contact diffeomorphism $\varphi$ preserving the orientation of $\xi$ such that $\varphi(K)=L$. When $\mathrm{d}(\varphi)\neq 0$ then we compose $\varphi$ with a contact diffeomorphism from \remref{r:dym surj}, so we may assume $\mathrm{d}(\varphi)=0$.

By \thmref{t:chekanov} $\varphi$ is either isotopic to the identity or to $\psi$. Hence, $L$ is isotopic to $K$ or $\psi(K)$. What remains to be shown is that $\psi(K)$ and $K$ are not Legendrian isotopic.  

For this, we first assume that $\tb(K)=0$. If in addition $\mathrm{rot}(K)=r=\pm 1$, then the claim follows from \thmref{t:not transitive}. We assume that $r\ge 3$ (recall that the sum of the Thurston-Bennequin invariant and the rotation number of a null-homologous knot is odd). 

According to Proposition 3.1 in  \cite{honda} there is a convex disc $D$ with $\partial D=K$. Moreover, $\mathrm{rot}(D)=\chi(D_+)-\chi(D_-)$ by \eqref{e:rot number} on p.~\pageref{e:rot number} when $D$ is a convex surface with contact vector field $X$ transverse to $D$ such that $X$, respectively $-X$, coorients $\xi$ on the interior of $D_+$, respectively $D_-$.  

If $r\ge3$, then $\chi(D_+)\ge 2$. By the Legendrian realization principle (Theorem 3.7 in \cite{honda}) we may assume that there is a properly embedded Legendrian arc $\gamma_D\subset D$ such that
\begin{itemize}
\item all singular points of $\xi(D)$ along $\gamma_D$ and $\partial D$ have the same sign, and
\item there is a segment $\sigma_D\subset\partial D$ such that $\gamma_D\cup\sigma_D$ is a piecewise smooth Legendrian knot with $\mathrm{rot}(\gamma_D\cup\sigma_D)=1$ (we use the orientation of $\sigma_D$ to orient $\gamma_D\cup\sigma_D$) and $\tb(\gamma_D\cup\sigma_D)=0$. 
\end{itemize}

Now assume that $\psi(K)$ and $K$ are isotopic through Legendrian knots $K_t$. We obtain a contact  isotopy $\varphi_t$  such that $\varphi_0=\mathrm{id}$, $\varphi_t(K)=K_t$ and $\varphi_1(\sigma_D)=\psi(\sigma_D)$. However, $\varphi_1(D)=\widehat{D}$ and $\psi(D)$ do not coincide in general. 

There is a smooth isotopy $\widehat{\varphi}_t$ such that $\widehat{\varphi}_0=\mathrm{id}$ and $ \widehat{\varphi}_1(\widehat{D})=\psi(D)$ fixing $\psi(K)$. Moreover, we may assume that the isotopy is the identity on a ball $B_t$ which is disjoint from the interior of $\widehat{\varphi}_t(\widehat{D})$ and contains $\psi(K)$ in its boundary. This determines a family of contact structures $\widehat{\xi}_t=\left(\widehat{\varphi}_t\right)_*(\xi_{-1})$ such that $\widehat{\xi}_t=\xi_{-1}$ on $B_t$.  The restriction of $\xi_{-1}$ to $B_t$ is overtwisted since $\mathrm{tb}(\psi(K))=0$.  By Remark \ref{r:dym surj} we may assume that the loop $\widehat{\xi}_t$ of contact structures is trivial as a loop of plane fields. 

By \thmref{t:eliashclass} and the discussion following this statement $\widehat{\xi}_t$ is trivial as a loop of contact structures. Using Gray's theorem we can deform the  isotopy $\widehat{\varphi}_t$ to a contact isotopy $\widetilde{\varphi}_t$ of $(S^3,\xi_{-1})$ such that $\widetilde{\varphi}_0=\widehat{\varphi}_0$ and $\widetilde{\varphi}_1=\widehat{\varphi}_1$.

We have constructed a contact isotopy  which moves $D$ to $\psi(D)$. But this implies that $\gamma_D\cup\sigma_D$ and $\psi(\gamma_D\cup\sigma_D)$ are Legendrian isotopic.  This is a contradiction to \thmref{t:not transitive}. Hence, $K$ cannot be isotopic to $\psi(K)$ when $\mathrm{tb}(K)=0$. 

Finally, if $\tb(K)>0$, then every isotopy interpolating between $K$ and $\psi(K)$ induces an isotopy interpolating between the $\tb(K)$-fold positive stabilization of $K$ and its image under $\psi$. This is a contradiction to what we have already shown. 
\end{proof}

Finally, let $\xi$ be an overtwisted contact structure which is not isomorphic to $\xi_{-1}$. 

\begin{prop} \label{p:other contstr unknot Leg simple}
Legendrian unknots in $(S^3,\xi)$ are classified up to Legendrian isotopy by the Thurston-Bennequin invariant and the rotation number. 
\end{prop}

\begin{proof}
Let $K,K'$ be two Legendrian unknots with the same Thurston-Bennequin invariant and rotation number. Both unknots are loose by \thmref{t:coarse unknot general}, we fix an overtwisted disc which is disjoint from $K$. There is a diffeomorphism $\psi$ of $\xi$ which maps $K$ to $K'$. If $\mathrm{d}(\psi)=0$, then $\psi$ is isotopc to the identity through contact diffeomorphisms. If $\mathrm{d}(\psi)\neq 0$, then we precompose $\psi$ with a contact diffeomorphism $\varphi$ which is supported in a neighborhood of $D$ and satisfies $\mathrm{d}(\varphi)=1$. 
\end{proof}

\end{document}